\documentclass[11pt,svgnames]{article}
\usepackage{amsmath}
\usepackage{amssymb}
\usepackage{amsthm}
\usepackage{enumerate}
\usepackage{wasysym}
\usepackage{dsfont}
\usepackage{comment}
\usepackage{tikz}
\usepackage{tikz-cd}
\usepackage{extarrows}
\usepackage{mathrsfs}
\usepackage{hyperref}
\usepackage{fullpage}
\usepackage{enumitem}
\usepackage{titlesec}
\usepackage{stmaryrd}
\hypersetup{
    linkbordercolor = {PaleTurquoise},
    citebordercolor = {PaleGreen},
}

\theoremstyle{definition}

\newtheorem{defn}{Definition}[section]
\newtheorem{para}[defn]{}
\newtheorem{egg}[defn]{Example}
\newtheorem{rmk}[defn]{Remark}
\newtheorem{assumption}[defn]{Assumption}
\newtheorem*{assumption*}{Assumption}

\theoremstyle{plain}

\newtheorem{prop}[defn]{Proposition}
\newtheorem{theo}[defn]{Theorem}
\newtheorem{lem}[defn]{Lemma}
\newtheorem{cor}[defn]{Corollary}

\newtheorem*{claim*}{Claim}
\newtheorem*{just*}{Justification}
\newtheorem*{lem*}{Lemma}
\newtheorem*{prop*}{Proposition}

\newcommand{\J}{\mathscr{J}}
\newcommand{\V}{\mathscr{V}}

\newcommand{\ob}{\operatorname{\mathsf{ob}}}
\newcommand{\mor}{\operatorname{\mathsf{mor}}}

\newcommand{\tensor}{\otimes}

\newcommand{\Cat}{\mathsf{Cat}}

\newcommand{\underJ}{\kern -0.5ex \mathscr{J}}
\newcommand{\Set}{\mathsf{Set}}

\newcommand{\op}{\mathsf{op}}

\newcommand{\C}{\mathscr{C}}

\newcommand{\B}{\mathscr{B}}
\newcommand{\A}{\mathscr{A}}

\newcommand{\Mod}{\text{-}\mathsf{Mod}}

\newcommand{\y}{\mathsf{y}}

\newcommand{\Alg}{\mathsf{Alg}}

\newcommand{\F}{\mathscr{F}}

\newcommand{\K}{\mathcal{K}}
\newcommand{\M}{\mathscr{M}}
\newcommand{\limit}{\mathsf{lim}}

\newcommand{\II}{\mathbb{I}}

\newcommand{\colim}{\mathsf{colim}}
\newcommand{\G}{\mathscr{G}}
\newcommand{\Iso}{\mathsf{Iso}}

\newcommand{\h}{\mathscr{H}}
\newcommand{\W}{\mathscr{W}}

\newcommand{\D}{\mathscr{D}}

\newcommand{\scrT}{\mathscr{T}}

\newcommand{\StrongEpi}{\mathsf{StrongEpi}}
\newcommand{\Mono}{\mathsf{Mono}}

\newcommand{\Cts}{\mathsf{Cts}}
\newcommand{\X}{\mathscr{X}}

\newcommand{\Epi}{\mathsf{Epi}}

\newcommand{\sketch}{\Psi}
\newcommand{\Eclass}{\mathscr{E}}
\newcommand{\Mclass}{\mathscr{M}}
\newcommand{\E}{\mathscr{E}}

\newcommand{\StrongMono}{\mathsf{StrongMono}}
\newcommand{\Top}{\mathsf{Top}}

\newcommand{\scrTAlg}{\scrT\text{-}\Alg}
\newcommand{\CAT}{\text{-}\mathsf{CAT}}
\newcommand{\Day}{\mathsf{\scriptscriptstyle Day}}
\newcommand{\Ev}{\mathsf{Ev}}
\newcommand{\Sub}{\mathsf{Sub}}
\newcommand{\Y}{\mathscr{Y}}
\newcommand{\Z}{\mathscr{Z}}
\DeclareMathOperator{\tensorbar}{\bar{\tensor}}

\titleformat{\section}{\normalsize\bfseries}{\thesection}{1em}{}
\titleformat{\subsection}{\normalsize\bfseries}{\thesubsection}{1em}{}

\numberwithin{equation}{section}

\begin{document}

\title{\Large \textbf{Locally bounded enriched categories}}

\author{Rory B. B. Lucyshyn-Wright\let\thefootnote\relax\thanks{We acknowledge the support of the Natural Sciences and Engineering Research Council of Canada (NSERC), [funding reference numbers RGPIN-2019-05274, RGPAS-2019-00087, DGECR-2019-00273].  Cette recherche a été financée par le Conseil de recherches en sciences naturelles et en génie du Canada (CRSNG), [numéros de référence RGPIN-2019-05274, RGPAS-2019-00087, DGECR-2019-00273].} \medskip \\ Jason Parker \medskip
\\
\small Brandon University, Brandon, Manitoba, Canada}

\date{}

\maketitle

\begin{abstract}
We define and study the notion of a locally bounded \emph{enriched} category over a (locally bounded) symmetric monoidal closed category, generalizing the locally bounded ordinary categories of Freyd and Kelly.  In addition to proving several general results for constructing examples of locally bounded enriched categories and locally bounded closed categories, we demonstrate that locally bounded enriched categories admit fully enriched analogues of many of the convenient results enjoyed by locally bounded ordinary categories. In particular, we prove full enrichments of Freyd and Kelly's reflectivity and local boundedness results for orthogonal subcategories and categories of models for sketches and theories.  We also provide characterization results for locally bounded enriched categories in terms of enriched presheaf categories, and we show that locally bounded enriched categories admit useful adjoint functor theorems and a representability theorem.  We also define and study the notion of $\alpha$-bounded-small weighted limit enriched in a locally $\alpha$-bounded closed category, which parallels Kelly's notion of $\alpha$-small weighted limit enriched in a locally $\alpha$-presentable closed category, and we show that enriched categories of models of $\alpha$-bounded-small weighted limit theories are locally $\alpha$-bounded.
\end{abstract}

\section{Introduction}

Based on ideas of Freyd and Kelly in \cite{FreydKelly}, the notion of locally bounded category has its foundation in the theory of factorization systems and was explicitly introduced by Kelly in \cite[Chapter 6]{Kelly}, where a given locally bounded symmetric monoidal closed category $\V$ was used as the basis for a general treatment of enriched limit theories. Locally bounded categories subsume locally presentable categories and a significant number of other examples, including many topological categories and various elementary quasitoposes that are not locally presentable, as well as numerous categories of structures internal to these categories. As such, the notion of locally bounded category is much weaker than the notion of locally presentable category: for example, a locally bounded category (such as the category $\Top$ of topological spaces and continuous functions) need not have a small dense subcategory or even a strong generator. Nevertheless, locally bounded categories still retain some of the convenient features of locally presentable categories, such as reflectivity results for orthogonal subcategories (see \cite{FreydKelly, Kellytrans}) and results on the existence of free monads, colimits of monads, and colimits in categories of algebras (see \cite{Kellytrans} and further references therein). The enriched limit theories over locally bounded closed categories treated by Kelly in \cite[Chapter 6]{Kelly} generalize the enriched finite limit theories that Kelly introduced in \cite{Kellystr} along with the notions of locally presentable symmetric monoidal closed category and locally presentable $\V$-category. However, notably absent from the literature has been a notion of locally bounded $\V$-category that would complete the parallel between the locally bounded and locally presentable settings. 

Our aims in writing this paper are therefore twofold: firstly, we wish to define and study the notion of a locally bounded enriched category as an object of study in its own right. Secondly, the authors have discovered (and will further demonstrate in forthcoming work \cite{Pres, EP, Struct}) that locally bounded enriched categories provide a fruitful and more expansive environment for studying certain phenomena in enriched categorical algebra and, more generally, in the study of structures internal to enriched categories via enriched limit theories and related methods. Indeed, some results of this type that already appear in the present paper may be found in Section \ref{sketches}. Hence, we wished to develop a full treatment of locally bounded enriched categories that would aid in our investigation of these enriched algebraic phenomena and results.  

We now provide an outline of the paper. After briefly reviewing some notation and terminology in Section \ref{background}, we begin by reviewing and studying the notion of a locally bounded \emph{ordinary} category in Section \ref{locallyboundedordinary}. In Section \ref{locallyboundedVcategories} we proceed to introduce the notion of a locally bounded $\V$-category over a (not necessarily locally bounded) symmetric monoidal closed category $\V$ satisfying certain mild assumptions. To define this class of $\V$-categories, we first introduce the notion of a $\V$\emph{-factegory}, which is a $\V$-category equipped with an enriched factorization system that is suitably \emph{compatible} with a given enriched factorization system on $\V$. After first proving some basic results about locally bounded $\V$-categories, we then develop the notion of a \emph{bounding right adjoint}, and we show that any bounding right adjoint whose codomain is locally bounded induces a locally bounded structure on its domain. We then use this result to prove a characterization theorem for locally bounded $\V$-categories, to the effect that a $\V$-category (satisfying certain mild assumptions) is locally bounded iff it has a bounding right adjoint into a presheaf $\V$-category. 

To prove stronger results about locally bounded $\V$-categories, we begin Section \ref{locbdclosedsection} by reviewing and proving some new results about Kelly's notion of a locally bounded symmetric monoidal closed category $\V$. We then show that a large number of (elementary) quasitoposes (which may fail to be locally presentable) are locally bounded cartesian closed categories. We also develop classes of examples of locally bounded closed categories that are topological over the category of sets. 

In the remainder of the paper, we study locally bounded $\V$-categories enriched over a locally bounded closed category $\V$. In Section \ref{enrichedvsordinary} we examine the relationship between the local boundedness of a $\V$-category versus that of its underlying ordinary category. In Section \ref{boundedpresentable} we compare locally bounded $\V$-categories with locally presentable $\V$-categories and the enriched locally generated categories of \cite{Libertienriched}. In Section \ref{representability} we show that locally bounded $\V$-categories satisfy particularly useful adjoint functor and representability theorems. 

In Section \ref{commutation}, we establish analogues in locally bounded $\V$-categories of results on the commutation of $\alpha$-small limits and $\alpha$-filtered colimits in locally $\alpha$-presentable categories. In more detail, we adapt Kelly's notion of \emph{$\alpha$-small weight} enriched in a locally $\alpha$-presentable closed category \cite{Kellystr} by replacing the notion of $\alpha$-presentable object with the notion of $\alpha$-bounded object, thus defining the notion of $\alpha$\emph{-bounded-small weight} enriched in a locally $\alpha$-bounded closed category, and we show that $\alpha$-bounded-small limits commute with $\alpha$-filtered unions in any locally $\alpha$-bounded $\V$-category.  

In Section \ref{orth_subcats}, we prove a fully enriched analogue of Freyd and Kelly's result that certain orthogonal subcategories of suitable locally bounded ordinary categories are reflective and locally bounded (see \cite[4.1.3, 4.2.2]{FreydKelly}).

In Section \ref{sketches}, we use results from Sections \ref{commutation} and \ref{orth_subcats} to obtain enriched reflectivity and local boundedness results for $\V$-categories of models of enriched limit sketches and enriched limit theories in arbitrary locally bounded $\V$-categories. In this way, we obtain a full enrichment of another main result of Freyd and Kelly (see \cite[5.2.1]{FreydKelly}), namely that certain categories of models of sketches valued in locally bounded ordinary categories are reflective in the associated functor categories and are themselves locally bounded.  Thus we obtain theorems on the reflectivity and local boundedness of $\V$-categories of models in a locally bounded $\V$-category $\C$, recalling that in the $\V$-enriched context, Kelly had proved only the reflectivity in just the $\C = \V$ case \cite[\S 6.3]{Kelly}.  We prove that the $\V$-category of models of any $\alpha$\emph{-bounded-small limit theory} in a suitable locally $\alpha$-bounded $\V$-category is itself locally $\alpha$-bounded, thus obtaining an analogue of the well-known result that the $\V$-category of models of an enriched finite limit theory in a locally finitely presentable closed category is itself locally finitely presentable.  Taken together, our results on $\V$-categories of models of sketches and theories in locally bounded $\V$-categories establish locally bounded $\V$-categories as an expansive and yet convenient ambient environment for the study of structures internal to enriched categories.  We conclude Section \ref{sketches} with a result that illustrates the applicability of locally bounded enriched categories to the study of enriched algebraic theories.  In Section \ref{sec:lbclosed_vcats_sm_phitheories}, we show that certain $\V$-categories of models of \textit{symmetric monoidal weighted limit theories} are locally bounded symmetric monoidal closed $\V$-categories, thus providing a further source of locally bounded closed categories.

One topic that we have chosen \emph{not} to include in the present paper concerns $\V$-categories of algebras for $\V$-monads on locally bounded $\V$-categories, but we plan to present results on this topic in forthcoming work.

\section{Notation and terminology}
\label{background}

In this section we fix some notation and terminology from enriched category theory; for the most part, we use the notation of \cite{Kelly}. Throughout the paper we make use of the standard methods of enriched category theory that are treated in \cite{Dubucbook, Kelly} and \cite[Chapter 6]{Borceux2}.

We generally work with categories enriched over a symmetric monoidal closed category $\V = (\V_0, \tensor, I)$ with $\V_0$ locally small, about which we do not make any further blanket assumptions (specific assumptions about $\V$ will be made as needed). We distinguish in the usual way between \emph{small} and \emph{large} classes; we also refer to small classes as (small) sets. We let $\Set$ be the cartesian closed category of (small) sets; an \emph{(ordinary) category} $\C$ is then a $\Set$-enriched (i.e. locally small) category. 

A \emph{weight} is a $\V$-functor $W : \B \to \V$ with $\B$ a (not necessarily small) $\V$-category; we then have the usual notions of \emph{(weighted) (co)limit} in a $\V$-category, and preservation of such by a $\V$-functor. Following \cite{Kelly}, we usually just say ``colimit" rather than ``weighted colimit". If $W : \B \to \V$ is a weight, then we say that a $\V$-category $\C$ is $W$\emph{-(co)complete} if $\C$ admits all $W$-weighted (co)limits, and that a $\V$-functor is $W$\emph{-(co)continuous} if it preserves all $W$-weighted (co)limits. If $\Phi$ is a class of weights, then we define the notions of $\Phi$\emph{-(co)completeness} and $\Phi$\emph{-(co)continuity} analogously. In particular, a $\V$-category $\C$ is \emph{(co)complete} if it is $\Phi$\emph{-(co)complete} for $\Phi$ the class of all small weights. Given objects $V \in \ob\V$ and $C \in \ob\C$, we denote the \emph{cotensor} of $C$ by $V$ in $\C$ (if it exists) by $[V, C]$, and the \emph{tensor} of $C$ by $V$ in $\C$ (if it exists) by $V \tensor C$.

Finally, we need the notion of a \emph{regular cardinal}, which is an infinite cardinal $\alpha$ that is not the sum of a smaller number of smaller cardinals, so that $\aleph_0$ is the smallest regular cardinal. It is well known that there is always a regular cardinal larger than every element of a given set of regular cardinals; cf. e.g. the (stronger) result \cite[2.13(6)]{LPAC}.   

\section{Locally bounded ordinary categories}
\label{locallyboundedordinary}

We begin by recalling the definition of a \emph{locally bounded ordinary category} from \cite[Section 6.1]{Kelly}. Recall that an \emph{(orthogonal) factorization system} on an ordinary category $\C$ is a pair $(\Eclass, \Mclass)$ of classes of morphisms of $\C$ with the following properties: $\E$ and $\M$ both contain the isomorphisms and are closed under composition; every morphism of $\C$ can be factorized as an $\E$-morphism followed by an $\M$-morphism; and every $\E$-morphism is orthogonal to every $\M$-morphism. It follows that a morphism belongs to $\E$ iff it is orthogonal to every $\M$-morphism, and dually a morphism belongs to $\M$ iff every $\E$-morphism is orthogonal to it. It is also a simple consequence of this definition that any morphism in $\Eclass \cap \Mclass$ must be an isomorphism. A factorization system $(\Eclass, \Mclass)$ on $\C$ is \emph{proper} if $\Eclass$ is contained in the epimorphisms and $\Mclass$ in the monomorphisms; in this case, every regular epimorphism (and in particular, every retraction) lies in $\Eclass$ (and dually for $\M$), while $g \circ f \in \Eclass$ implies $g \in \Eclass$ and $g \circ f \in \Mclass$ implies $f \in \Mclass$ (see \cite[2.1.4]{FreydKelly}). We often refer to $\E$ as the left class and to $\M$ as the right class.

Since we mostly consider categories equipped with proper factorization systems in this paper, we introduce the following abbreviated terminology:

\begin{defn}
\label{factegory}
A \textbf{(proper) factegory} is a category $\C$ equipped with a proper factorization system $(\Eclass, \Mclass)$. A factegory $\C$ is \textbf{cocomplete} if the category $\C$ is cocomplete and every (even large) class of $\E$-morphisms with common domain has a cointersection (i.e. wide pushout) in $\C$. \qed
\end{defn}

It is shown in \cite[1.3]{Kellytrans} that the cocompleteness of a factegory actually implies that $\E$ must be contained in the epimorphisms. Of course, if $\C$ is cocomplete and $\E$\emph{-cowellpowered}, in the sense that every object of $C$ has just a set of isomorphism classes of $\E$-quotients (i.e. $\E$-morphisms with domain $C$), then the factegory $\C$ is automatically cocomplete. Throughout, we implicitly equip $\Set$ with its usual proper factorization system $(\Epi, \Mono)$, making $\Set$ a cocomplete factegory.  
 
Given an object $C$ of a factegory $\C$ (assumed proper), an \textit{$\M$-subobject} of $C$ is, in this paper, simply an $\M$-morphism with codomain $C$ (rather than equivalence class of such morphisms).  The $\M$-subobjects of $C$ are the objects of a full subcategory $\Sub_\M(C)$ of the slice category $\C \slash C$, and this category $\Sub_\M(C)$ is a preordered class since $\M$ consists of monomorphisms.  From one point of view, an $\M$-subobject of $C$ is an object $A$ of $\C$ equipped with a specified $\M$-morphism $m:A \rightarrow C$, so we sometimes write simply the domain $A$ to denote such an $\M$-subobject $m$.  In particular, if $m:A \rightarrow C$ and $n:B \rightarrow C$ are $\M$-subobjects, then we write that $A \cong B$ \textit{as $\M$-subobjects of $C$} to mean that $m \cong n$ in $\Sub_\M(C)$.

\begin{para}\label{union}
We now recall the notion of a \emph{union} of $\Mclass$-subobjects in a factegory $\C$ (see \cite[2.4]{FreydKelly}). Recall that, in general, a sink in a category $\C$ is a family of morphisms $(f_i : C_i \to C)_{i \in I}$ with common codomain, but in this paper we use the term \textit{sink} to refer only to small sinks, i.e. those whose indexing class $I$ is a (small) set. If $\C$ is a factegory, then we say that a sink $(f_i : C_i \to C)_{i \in I}$ is $\E$\emph{-tight}\footnote{Here we have adopted a variation on Kelly's term \textit{$\E$-tight (co)cone} \cite[2.2]{Kellytrans}.  Since $\M$ consists of monomorphisms, $\E$-tight sinks are precisely those (small) sinks that are orthogonal to $\M$-morphisms in the standard sense \cite[II.5.3]{Monoidaltop}.}, or is \emph{jointly in $\E$}, if a morphism $h : C \to B$ factors through an $\M$-subobject $m : A \to B$ iff each composite $h \circ f_i$ ($i \in I$) factors through $m$. If the coproduct $\coprod_i C_i$ exists in $\C$, then it is easy to see that $(f_i : C_i \to C)_{i \in I}$ is $\E$-tight iff the induced morphism $\coprod_i C_i \xrightarrow{[f_i]_i} C$ is in $\E$ (see \cite[2.4]{FreydKelly}). If the category $\C$ is cocomplete and $D : \A \to \C$ is a small diagram, then every colimit cocone $(s_A : DA \to \colim \ D)_{A \in \A}$ is $\E$-tight, because (by the construction of colimits from coproducts and coequalizers) the induced morphism $\coprod_A DA \to \colim \ D$ is a regular epimorphism and thus lies in $\E$. We also say that a sink $(m_i : C_i \to C)_{i \in I}$ is an $\M$\emph{-sink} if $m_i \in \M$ for each $i \in I$.   A sink $(f_i : C_i \to C)_{i \in I}$ \emph{factors} through a morphism $g : D \to C$ with the same codomain if there is a sink $(h_i : C_i \to D)_{i \in I}$ with $g \circ h_i = f_i$ for each $i \in I$, which is unique if $g$ is a monomorphism. Given an $\M$-sink $(m_i : C_i \to C)_{i \in I}$ in a factegory $\C$, we now say that an $\M$-subobject $m : D \to C$ is a \textbf{union} (or $\M$\textbf{-union}) of the given $\M$-sink if the $\M$-sink factors (uniquely) through $m$ via an $\E$-tight sink.  It is straightforward to show that if $m$ is a union of an $\M$-sink $(m_i : C_i \to C)_{i \in I}$, then $m$ is a join $\bigvee_i m_i$ in the preordered class $\Sub_\M(C)$.  In particular, the union of an $\M$-sink $(m_i : C_i \to C)_{i \in I}$ is unique up to isomorphism in $\Sub_\M(C)$ if it exists, in which case we write it as $\bigcup_i C_i \rightarrow C$, and we also write simply $\bigcup_i C_i$ to denote the latter $\M$-subobject.

Now supposing that $\C$ is a factegory with small coproducts, we have the following straightforward way of constructing unions in $\C$.  For each object $C \in \ob\C$, let us write $J_C : \Sub_\M(C) \to \C/C$ for the inclusion functor. It is well known (see e.g. \cite[5.5.5]{Borceux1}) that $J_C$ has a left adjoint $K_C : \C/C \to \Sub_\M(C)$, which sends each morphism $f : D \to C$ to the $\M$-component of its $(\E, \M)$-factorization.  But the slice $\C \slash C$ has small coproducts, formed as in $\C$, so if $(m_i : C_i \to C)_{i \in I}$ is an $\M$-sink then the join $\bigvee_i m_i$ in the preordered class $\Sub_\M(C)$ may be formed by first taking the coproduct in $\C/C$ of the same family $\left(J_C(m_i) : C_i \to C\right)_{i \in I}$, and then taking its reflection into $\Sub_\M(C)$; that is, $\bigvee_{i \in I} m_i = K_C\bigl(\coprod_{i \in I} J_C(m_i)\bigr)$.  But the coproduct $\coprod_{i \in I} J_C(m_i)$ in $\C \slash C$ is simply the induced morphism $\coprod_i C_i \to C$, so by taking the $(\E,\M)$-factorization $\coprod_i C_i \xrightarrow{e} D \xrightarrow{m} C$ of the latter morphism we find that $m = \bigvee_i m_i$ in $\Sub_\M(C)$.  But the given $\M$-sink $(m_i)$ clearly factors through $m$ by way of an $\E$-tight sink, so in fact this join $m = \bigvee_i m_i:D \rightarrow C$ is, moreover, a union of the sink $m_i$, that is, $D \cong \bigcup_i C_i$ as $\M$-subobjects of $C$.  In summary, every $\M$-sink $(m_i : C_i \to C)_{i \in I}$ in $\C$ has a union, obtained by taking the $(\E,\M)$-factorization
$\coprod_i C_i \xrightarrow{e} \bigcup_i C_i \xrightarrow{m} C$
of the canonical morphism.

Thus, any factegory with small coproducts (in particular, $\Set$) has unions of all $\M$-sinks.  If $\alpha$ is a regular cardinal, then an $\M$-sink $\left(m_i : C_i \to C\right)_{i \in I}$ is \textit{$\alpha$-filtered} if the $\M$-subobjects $m_i$ constitute an $\alpha$-filtered full subcategory of $\Sub_\M(C)$, equivalently, if $\{m_i \mid i \in I\}$ is a $\alpha$-directed subset of the preordered class $\Sub_\M(C)$, i.e., for every subset $J \subseteq I$ of cardinality less than $\alpha$, there is some $i \in I$ with $m_j$ factoring through $m_i$ for all $j \in J$.  An $\alpha$\textbf{-filtered union\footnote{We follow Kelly \cite[\S 6.1]{Kelly} in calling these $\alpha$-filtered unions rather than $\alpha$-directed unions.} of} $\Mclass$\textbf{-subobjects} or an $\alpha$\textbf{-filtered} $\Mclass$\textbf{-union} is the union of an $\alpha$-filtered $\M$-sink. A factegory $\C$ has ($\alpha$-filtered) $\M$-unions if every ($\alpha$-filtered) $\M$-sink in $\C$ has a union. \qed 
\end{para}

We now wish to formulate the notion of a functor \emph{preserving unions}. First, we require the following definition:

\begin{defn}
\label{preservesrightclass}
A \textbf{right-class functor} is a functor $U:\C \to \D$ between factegories $\C$ and $\D$ such that $U$ \textbf{preserves the right class}, meaning that $m \in \Mclass$ implies $Um \in \Mclass$.  Similarly, a \textbf{left-class functor} is a functor $F:\D \to \C$ between factegories $\D$ and $\C$ such that $F$ \textbf{preserves the left class}, meaning that $e \in \E$ implies $Fe \in \E$. \qed
\end{defn}

\noindent We immediately note the following useful fact:

\begin{lem}
\label{preservesrightclasslem}
Let $U : \C \to \D$ be a functor between factegories, and suppose that $U$ has a left adjoint $F$.  Then $U$ preserves the right class iff $F$ preserves the left class.
\end{lem}

\begin{proof}
From the first sentence of the proof of \cite[4.2.1]{FreydKelly} we know for any morphisms $f$ of $\D$ and $g$ of $\C$ that $Ff$ is orthogonal to $g$ iff $f$ is orthogonal to $Ug$. If $U$ preserves the right class, then to show that $F$ preserves the left class, it suffices to show that $e \in \E$ implies that $Fe$ is orthogonal to every $m \in \M$, i.e. (by the fact just mentioned) that $e$ is orthogonal to $Um$ for every $m \in \M$, which is true because $e \in \E$ and $Um \in \M$ by assumption. The converse implication is proved analogously. 
\end{proof}

\begin{prop}
\label{leftadjointpreservesEtightness}
Let $\C$ and $\D$ be factegories with small coproducts.  If $F : \D \to \C$ is a left-class functor that preserves small coproducts, then $F$ preserves $\E$-tight sinks.  In particular, if $U:\C \rightarrow \D$ is a right adjoint right-class functor, then its left adjoint preserves $\E$-tight sinks.
\end{prop}

\begin{proof}
Firstly, if $F:\D \rightarrow \C$ has a right adjoint that preserves the right class, then $F$ certainly preserves small coproducts and also preserves the left class by \ref{preservesrightclasslem}. Now supposing just these latter properties, let $(f_i : D_i \to D)_{i \in I}$ be an $\E$-tight sink, so that the induced morphism $\coprod_i D_i \to D$ lies in $\E$. Then the assumptions imply that the canonical morphism $\coprod_i FD_i \cong F\left(\coprod_i D_i\right) \to FD$ lies in $\E$, so that the induced sink $(Ff_i : FD_i \to FD)_{i \in I}$ is $\E$-tight, as desired.
\end{proof}

\begin{para}\label{uc}
If $U : \C \to \D$ is a right-class functor, then each object $C$ of $\C$ determines a functor
$$U_C\;:\;\Sub_\M(C) \longrightarrow \Sub_\M(UC)$$
that sends each $\M$-subobject $m:A \to C$ to the $\M$-subobject $Um:UA \to UC$.  In particular, $U_C$ is a monotone map between preordered classes. \qed
\end{para}

\begin{defn}\label{preservesunions}
Let $U:\C \to \D$ be a right-class functor.  Given a union $m:\bigcup_i C_i \to C$ of an $\M$-sink $(m_i : C_i \to C)_{i \in I}$ in $\C$, we say that $U$ \textbf{preserves the union} $m$ of $(m_i)_{i \in I}$ if $Um:U\left(\bigcup_i C_i\right) \rightarrow UC$ is a union of the $\M$-sink $(Um_i : UC_i \to UC)_{i \in I}$ in $\D$.  If $\alpha$ is a regular cardinal, then we say that $U$ \textbf{preserves} ($\alpha$\textbf{-filtered}) $\Mclass$-\textbf{unions} if $U$ preserves every ($\alpha$-filtered) $\M$-union that exists in $\C$. \qed 
\end{defn}

Given a right-class functor $U:\C \rightarrow \D$ between factegories with $\M$-unions, since the union $m:\bigcup_i C_i \rightarrow C$ of an $\M$-sink $(m_i:C_i \to C)_{i \in I}$ in $\C$ is a join $\bigvee_i m_i$ in $\Sub_\M(C)$, we find that $U$ preserves the union $\bigcup_i C_i \rightarrow C$ iff the map $U_C:\Sub_\M(C) \rightarrow \Sub_\M(UC)$ preserves the join $\bigvee_i m_i$, iff $U\left(\bigcup_i C_i\right) \cong \bigcup_i UC_i$ in $\Sub_\M(UC)$.

\begin{para}\label{EMgen}
If $\G$ is a small set of objects in a factegory $\C$ with small coproducts, then $\G$ is an $(\Eclass, \Mclass)$\emph{-generator} if for each $C \in \ob\C$, the canonical morphism $\coprod_{G \in \G} \C(G, C) \cdot G \to C$ is in $\Eclass$, where $\C(G, C) \cdot G$ is the set-indexed copower in $\C$. As noted in \cite[\S 2]{KellyLack}, $\G$ is an $(\Eclass, \Mclass)$-generator iff the functors $\C(G, -) : \C \to \Set$ ($G \in \G$) are \emph{jointly} $\Mclass$\emph{-conservative}, in the sense that $m \in \Mclass$ is an isomorphism if $\C(G, m)$ is a bijection for every $G \in \G$; in \ref{enrichedEMgeneratoralt}, we establish an enriched generalization of this equivalence. Note that if $\E = \Epi$, then an $(\E, \M)$-generator is just a generator in the usual sense (see \cite[4.5.2]{Borceux1}). \qed    
\end{para}

To recall the definition of locally bounded category, we require also the following concept:

\begin{defn}
\label{alphaboundedobject}
Let $\C$ be a factegory with $\M$-unions. Given a regular cardinal $\alpha$, an object $C \in \ob\C$ is $\alpha$\textbf{-bounded} if the functor $\C(C, -) : \C \to \Set$ preserves $\alpha$-filtered $\M$-unions.  An object $C \in \ob\C$ is \textbf{bounded} if it is $\alpha$-bounded for some $\alpha$.\qed
\end{defn}

\noindent Note that the representable functor $\C(C, -) : \C \to \Set$ preserves the right class by properness, recalling that $\Set$ carries its usual proper factorization system. Concretely, an object $C \in \ob\C$ is $\alpha$-bounded if for every $\alpha$-filtered $\M$-sink $(m_i : D_i \to D)_{i \in I}$, every morphism $f : C \to \bigcup_i D_i$ factors through some $D_i$.

\begin{defn}
\label{locallyboundedordinarycategory}
Let $\alpha$ be a regular cardinal. A \textbf{locally} $\alpha$\textbf{-bounded category} is a cocomplete factegory $\C$ equipped with an $(\E, \M)$-generator consisting of $\alpha$-bounded objects. A category is \textbf{locally bounded} if it is locally $\alpha$-bounded for some regular cardinal $\alpha$. \qed  
\end{defn}

\noindent Note that a locally $\alpha$-bounded category is also locally $\beta$-bounded for any regular cardinal $\beta > \alpha$ (because an $\alpha$-bounded object is also $\beta$-bounded). We provide many examples of locally bounded (closed) categories in Section \ref{examples} below.

\begin{para}
\label{every_obj_bdd}
By \cite[3.1.2]{FreydKelly} (or \cite[\S 6.1]{Kelly}), every object of a locally bounded category is bounded. \qed 
\end{para}

We conclude this section with the following characterization result for $\E$\emph{-cowellpowered} locally $\alpha$-bounded categories proved by Sousa in \cite[2.8]{Sousa}, which is analogous to the corresponding characterization of locally $\alpha$-presentable categories in terms of $\alpha$-presentable objects and $\alpha$-filtered colimits. We shall make use of this important result in proving our representability and adjoint functor theorems for locally bounded categories in subsection \ref{representability} below. If $(m_i : C_i \to C)_{i \in I}$ is an $\M$-sink in a factegory $\C$, then we say that the codomain object $C$ \emph{is a union of} $(m_i)_i$ if the identity ($\M$-)morphism $1_C : C \to C$ is a union of $(m_i)_i$, which is clearly equivalent to the $\M$-sink $(m_i)_i$ being $\E$-tight.   

\begin{theo}[Sousa \cite{Sousa}]
\label{Sousathm}
Let $\C$ be a cocomplete and $\E$-cowellpowered factegory. Then $\C$ is locally $\alpha$-bounded iff there is a set $\h$ of $\alpha$-bounded objects of $\C$ such that every object $C \in \ob\C$ is an $\alpha$-filtered $\M$-union of objects of $\h$, in the sense that $C$ is a union of some $\alpha$-filtered $\M$-sink $(m_i : H_i \to C)_{i \in I}$ with each $H_i \in \h$. \qed
\end{theo} 

\section{Locally bounded \texorpdfstring{$\V$}{V}-categories over an arbitrary \texorpdfstring{$\V$}{V}}
\label{locallyboundedVcategories}

In this section we introduce the notion of a locally bounded $\V$\emph{-category} over a symmetric monoidal closed category $\V$ satisfying certain modest assumptions, without requiring $\V$ to be locally bounded.  

\subsection{\texorpdfstring{$\V$}{V}-factegories and enriched \texorpdfstring{$(\E, \M)$}{(E,M)}-generators}
\label{definitionsbasicresults}

We first recall from \cite{enrichedfact} the notion of an \emph{enriched factorization system} on a $\V$-category $\C$ (for an arbitrary symmetric monoidal closed category $\V$). If $e : C \to C'$ and $m : D \to D'$ are morphisms in $\C$, then we say that $e$ is $\V$\emph{-orthogonal} to $m$, which we write as $e \perp_\V m$, if the following commutative square in $\V$ is a pullback:
\begin{equation}\label{Vorthogonal}
\begin{tikzcd}
	{\C(C', D) } && {\C(C', D')} \\
	\\
	{\C(C, D)} && {\C(C, D')}.
	\arrow["{\C(C', m)}", from=1-1, to=1-3]
	\arrow["{\C(e, D)}"', from=1-1, to=3-1]
	\arrow["{\C(C, m)}"', from=3-1, to=3-3]
	\arrow["{\C(e, D')}", from=1-3, to=3-3]
\end{tikzcd}
\end{equation}
If $(\E, \M)$ is a pair of classes of morphisms in a $\V$-category $\C$, then $(\E, \M)$ is an \emph{enriched} \emph{factorization system} on $\C$ if the following conditions are satisfied (see \cite[5.2]{enrichedfact}): $\E$ and $\M$ are closed under composition and contain the isomorphisms; every $\E$-morphism is $\V$-orthogonal to every $\M$-morphism; and every morphism of $\C$ can be factored as an $\E$-morphism followed by an $\M$-morphism.  Every enriched factorization system $(\E,\M)$ is, in particular, an \textit{enriched prefactorization system} \cite[3.1, 5.2]{enrichedfact}, meaning that $\E$ is the class of all morphisms $\V$-orthogonal to all $\M$-morphisms and $\M$ is the class of all morphisms to which all $\E$-morphisms are $\V$-orthogonal.

If $\C$ is tensored, then it follows from \cite[5.7]{enrichedfact} that a pair of classes of morphisms $(\E, \M)$ of $\C$ is an enriched factorization system on $\C$ iff $(\E, \M)$ is an (ordinary) factorization system on $\C_0$ and $\E$ is stable under tensoring in $\C$, meaning that $e \in \E$ implies $V \tensor e \in \E$ for every $V \in \ob\V$. And if $\C$ is cotensored, then it follows from the same result that a pair of classes of morphisms $(\E, \M)$ of $\C$ is an enriched factorization system on $\C$ iff $(\E, \M)$ is an (ordinary) factorization system on $\C_0$ and $\M$ is stable under cotensoring in $\C$, meaning that $m \in \M$ implies $[V, m] \in \M$ for every $V \in \ob\V$. If $\C$ is tensored, then it follows by \cite[2.4]{enrichedfact} that the class of monomorphisms in $\C_0$ is equal to the class of $\V$\emph{-monomorphisms} in $\C$, i.e. morphisms $m$ in $\C$ that are sent to monomorphisms in $\V$ by each representable $\V$-functor $\C(C, -) : \C \to \V$ ($C \in \ob\C$). Dually, if $\C$ is cotensored, then it follows by the same result that the class of epimorphisms in $\C_0$ is equal to the class of $\V$\emph{-epimorphisms} in $\C$, i.e. morphisms $e$ in $\C$ that are sent to monomorphisms by every contravariant representable $\V$-functor $\C(-, C) : \C^\op \to \V$ ($C \in \ob\C$). By a \emph{proper} enriched factorization system on a $\V$-category $\C$ we mean an enriched factorization system on $\C$ such that every $\E$-morphism is a $\V$-epimorphism (and hence an epimorphism) and every $\M$-morphism is a $\V$-monomorphism (and hence a monomorphism). Thus, an enriched proper factorization system on $\C$ is in particular a proper factorization system on $\C_0$. Since $\V$ is a tensored and cotensored $\V$-category, an enriched (proper) factorization system on $\V$ is just a (proper) factorization system on $\V_0$ whose left class is stable under tensoring or, equivalently, whose right class is stable under cotensoring. We now define the following convenient terminology:

\begin{defn}
\label{closedfactegory}
A \textbf{(symmetric monoidal) closed factegory} is a symmetric monoidal closed category $\V$ equipped with an enriched proper factorization system $(\E, \M)$. A closed factegory $\V$ is \textbf{cocomplete} if $\V_0$ is cocomplete and has arbitrary cointersections of $\E$-morphisms. \qed
\end{defn}

\begin{assumption}
For the remainder of Section \ref{locallyboundedVcategories} we suppose that $\V$ is a cocomplete closed factegory, and that $\V_0$ is complete. \qed
\end{assumption}

\begin{defn}
\label{compatible}
Let $\C$ be a $\V$-category with an enriched factorization system $(\E_\C, \M_\C)$. Then $(\E_\C, \M_\C)$ \textbf{is compatible with} $(\E, \M)$ if for every $C \in \ob\C$, the functor $\C(C, -) : \C \to \V$ preserves the right class. \qed
\end{defn}

\begin{defn}
\label{Vfactegory}
A \textbf{(proper)} $\V$\textbf{-factegory} is a $\V$-category $\C$ equipped with an enriched proper factorization system $(\E_\C, \M_\C)$ that is compatible with $(\E, \M)$. The $\V$-factegory $\C$ is \textbf{cocomplete} if the $\V$-category $\C$ is cocomplete and has arbitrary (conical) cointersections of $\E_\C$-morphisms.\footnote{Unless otherwise stated, all ordinary (co)limits in a $\V$-category will be assumed to be conical.} \qed
\end{defn}

\noindent We often write simply $(\E, \M)$ to denote $(\E_\C, \M_\C)$ when there is no cause for confusion. 

\begin{rmk}
\label{VitselfVfactegory}
Note that $\V$ itself is a cocomplete $\V$-factegory, because $(\E, \M)$ is stable under cotensoring and compatible with itself by virtue of being enriched, while the arbitrary cointersections of $\E$-morphisms that $\V_0$ possesses are conical because $\V$ is cotensored. \qed
\end{rmk}

\noindent If $\C$ is tensored, then we deduce the following useful lemma from \ref{preservesrightclasslem}: 

\begin{lem}
\label{compatiblelemma}
Let $\C$ be a tensored $\V$-category with an enriched factorization system. Then for every $C \in \ob\C$, the tensor $\V$-functor $(-) \tensor C : \V \to \C$ preserves the left class iff the representable $\V$-functor $\C(C, -) : \C \to \V$ preserves the right class. In particular, each $\V$-functor $(-) \tensor C : \V \to \C$ ($C \in \ob\C$) preserves the left class if $\C$ is a tensored $\V$-factegory. \qed  
\end{lem}

\noindent We provide the following initial result for constructing $\V$-factegories. By an \textit{($\M$-)union} in a $\V$-factegory, we mean an $\M$-union in the underlying ordinary factegory. 

\begin{prop}
\label{functorfactegory}
Let $\B$ be a cotensored $\V$-factegory and $\A$ a small $\V$-category. Then the $\V$-category $[\A, \B]$ is a $\V$-factegory when equipped with the pointwise factorization system. If the $\V$-factegory $\B$ is cocomplete, then the $\V$-factegory $[\A, \B]$ is also cocomplete, and unions are formed pointwise in $[\A,\B]$.
\end{prop}

\begin{proof}
We must define an enriched proper factorization system $(\E', \M')$ on $[\A, \B]$ that is compatible with $(\E, \M)$. We define $\E'$ to consist of the $\V$-natural transformations that are pointwise in $\E_\B$ and $\M'$ to consist of the $\V$-natural transformations that are pointwise in $\M_\B$. The proof of \cite[6.4]{Kelly} then easily extends to show that this is indeed an ordinary proper factorization system on $[\A, \B]_0$. Since $\B$ is cotensored, it follows that $[\A, \B]$ is also cotensored.

To show that $(\E', \M')$ is enriched, equivalently, that $\M'$ is stable under cotensoring, let $V \in \ob\V$ and $\gamma : F \to F'$ in $\M'$, and let us show that $[V, \gamma] : [V, F] \to \left[V, F'\right]$ is in $\M'$, i.e. that $[V, \gamma]_A : [V, F]A \to [V, F']A$ is in $\M_\B$ for all $A \in \ob\A$. Since cotensors in $[\A, \B]$ are formed pointwise, we have $[V, \gamma]_A = \left[V, \gamma_A\right] : [V, FA] \to [V, FA']$. But $\gamma_A \in \M_\B$ because $\gamma \in \M'$, and hence $[V, \gamma_A] \in \M_\B$ because $\M_\B$ is stable under cotensoring, since $(\E_\B, \M_\B)$ is enriched.

To show that $(\E', \M')$ is compatible with $(\E, \M)$, let $F : \A \to \B$ be a $\V$-functor, and let us show that $[\A, \B](F, -) : [\A, \B] \to \V$ sends $\M'$-morphisms to $\M$-morphisms. So let $\gamma : S \to S'$ be in $\M'$, and let us show that the morphism $[\A, \B](F, S) \xrightarrow{[\A, \B](F, \gamma)} [\A, \B]\left(F, S'\right)$, i.e. the morphism
\[ \int_{A \in \A} \B(FA, SA) \xrightarrow{\int_{A \in \A} \B\left(FA, \gamma_A\right)} \int_{A \in \A} \B(FA, S'A), \] is in $\M$. But $\gamma_A \in \M_\B$ for each $A \in \ob\A$ by assumption, and hence $\B(FA, \gamma_A) \in \M$ for each $A \in \ob\A$ because $(\E_\B, \M_\B)$ is compatible with $(\E, \M)$, so that $\int_{A \in \A} \B(FA, \gamma_A) \in \M$ because $\M$ is stable under limits (e.g., by \cite[4.3]{BinKellylimits}). This proves that $[\A, \B]$ is a $\V$-factegory when equipped with the pointwise factorization system $(\E', \M')$.

If the $\V$-factegory $\B$ is cocomplete, then $[\A,\B]$ has small colimits, formed pointwise, and $[\A,\B]$ also has arbitrary cointersections of $\E'$-morphisms, also formed pointwise.
\end{proof}

\begin{defn}
A \textbf{right-class $\V$-functor} is a $\V$-functor $U:\C \rightarrow \D$ between $\V$-factegories $\C$ and $\D$ such that $U$ preserves the right class (cf. \ref{preservesrightclass}).  Analogously, a left-class $\V$-functor is a $\V$-functor $F$ between $\V$-factegories such that $F$ preserves the left class. \qed
\end{defn}

\begin{egg}\label{reps_rclass_vfunc}
If $C$ is an object of a $\V$-factegory $\C$, then $\C(C,-):\C \rightarrow \V$ is a right-class $\V$-functor, by compatibility of the factorization systems. \qed
\end{egg}

\begin{defn}\label{Mconservative}
Let $U:\C \to \D$ be a $\V$-functor between $\V$-factegories.  We say that $U$ is \mbox{\textbf{$\M$-conservative}} if for every $m \in \M$, if $Um$ is an isomorphism then $m$ is an isomorphism.  We say that $U$ \textbf{reflects the left class} if for every morphism $e$ in $\C$, if $Ue \in \E$ then $e \in \E$. \qed
\end{defn}

\begin{prop}
\label{boundingequivalent}
Let $U : \C \to \D$ be a a right-class $\V$-functor. Then $U$ is $\M$-conservative iff $U$ reflects the left class.
\end{prop}

\begin{proof}
Suppose first that $U$ is $\M$-conservative. To show that $U$ reflects the left class, suppose $Ue \in \E$ with $e : C \to C'$, and let us show that $e \in \E$. If $C \xrightarrow{e'} C'' \xrightarrow{m} C'$ is the $(\E, \M)$-factorization of $e$, then it suffices to show that $m \in \M$ is an isomorphism, for which it suffices by hypothesis to show that $Um$ is an isomorphism. But since $Ue = Um \circ Ue'$ and $Ue \in \E$, it follows by properness that $Um \in \E$, so that $Um \in \E \cap \M$ is an isomorphism, as desired. 

Conversely, suppose that $U$ reflects the left class, and let $m \in \M$ with $Um$ an isomorphism. Then $Um \in \E$, so that $m \in \E$ by hypothesis, and hence $m \in \E \cap \M$ is an isomorphism.
\end{proof}

\noindent If the $\V$-functor of \ref{boundingequivalent} has a left adjoint, then we can extend the equivalent conditions thereof as follows.

\begin{prop}
\label{moreboundingequivalent}
Let $U : \C \to \D$ be a right-class $\V$-functor with a left adjoint $F$.  Then the equivalent properties of \ref{boundingequivalent} are also equivalent to the following two properties:
\begin{itemize}[leftmargin=*]
\item The counit of the adjunction $F \dashv U$ is pointwise in $\E$.
\item $U$ reflects the order of $\M$-subobjects, i.e. the functor $U_C : \Sub_\M(C) \to \Sub_\M(UC)$ of \ref{uc} is fully faithful for each $C \in \ob\C$.   
\end{itemize}
\end{prop}

\begin{proof}
We first show that the counit $\varepsilon : FU \to 1_\C$ is pointwise in $\E$ iff $U$ reflects the left class. Assuming the former property, let $e : C \to C'$ in $\C$ with $Ue \in \E$. Since $F \dashv U$ and $U$ preserves the right class, it follows by \ref{preservesrightclasslem} that $F$ preserves the left class, so that $FUe \in \E$. By the naturality of $\varepsilon$ we have $\varepsilon_{C'} \circ FUe = e \circ \varepsilon_C$, and the left composite is in $\E$ because $\varepsilon_{C'} \in \E$, so $e \circ \varepsilon_C \in \E$, which entails $e \in \E$ by properness.

Supposing now that $U$ reflects the left class, let $C \in \ob\C$ and let us show that $\varepsilon_C \in \E$. By assumption, it suffices to show that $U\varepsilon_C \in \E$. By one of the triangle equalities for the adjunction we have $U\varepsilon_C \circ \eta_{UC} = 1_{UC}$, so that $U\varepsilon_C$ is a retraction and hence lies in $\E$ by properness. 

Now suppose that the counit is pointwise in $\E$, and let us show for any $C \in \ob\C$ that $U_C : \Sub_\M(C) \to \Sub_\M(UC)$ is fully faithful. So let $m : A \to C$ and $n : B \to C$ be $\M$-subobjects, and suppose that the $\M$-subobject $UA \xrightarrow{Um} UC$ factors through the $\M$-subobject $UB \xrightarrow{Un} UC$ via a (unique) $\M$-morphism $UA \xrightarrow{m'} UB$. Since the counit $FUA \xrightarrow{\varepsilon_A} A$ lies in $\E$, it will follow by orthogonality that $m$ factors through $n$ if the composite $FUA \xrightarrow{\varepsilon_A} A \xrightarrow{m} C$ factors through $B \xrightarrow{n} C$, which is true because $m \circ \varepsilon_A = \varepsilon_C \circ FUm = \varepsilon_C \circ F(Un \circ m') = \varepsilon_C \circ FUn \circ Fm' = n \circ \varepsilon_B \circ Fm'$. 

Conversely, suppose that $U$ reflects the order of $\M$-subobjects, and let us show that $U$ is $\M$-conservative.  Let $m:B \rightarrow C$ be an $\M$-morphism in $\C$, and suppose that $Um$ is an isomorphism.  Then $Um \cong 1_{UC} = U1_C$ in $\Sub_\M(UC)$, so since $U_C$ is fully faithful we deduce that $m \cong 1_C$ in $\Sub_\M(C)$, and hence $m$ is an isomorphism in $\C$.
\end{proof}

\noindent If $U : \C \to \D$ is a right-class $\V$-functor, then we say that $U$ \emph{reflects} $\M$\emph{-unions} if for every $\M$-sink $(m_i : C_i \to C)_{i \in I}$ in $\C$ and every $\M$-morphism $n : D \to C$, if $Un$ is a union of the $\M$-sink $(Um_i : UC_i \to UC)_{i \in I}$, then $n$ is a union of the $\M$-sink $(m_i)_i$. 

\begin{cor}
\label{functorreflectsunions}
Let $U : \C \to \D$ be a right-class $\V$-functor between $\V$-factegories with $\M$-unions.  Suppose that $U$ has a left adjoint and satisfies any of the equivalent properties of \ref{boundingequivalent} and \ref{moreboundingequivalent}.  Then $U$ reflects $\M$-unions.
\end{cor}

\begin{proof}
For each object $C$ of $\C$, $U_C:\Sub_\M(C) \rightarrow \Sub_\M(UC)$ is a fully faithful functor between preordered classes, so $U_C$ reflects joins, and the result follows since $\M$-unions are small joins of $\M$-subobjects.
\end{proof}

We now discuss the notion of enriched $(\E,\M)$-generator, which originated with enriched notions of generator and strong generator in \cite{Kelly,Kellystr,Daytotality}.  If $\C$ is a tensored $\V$-factegory (i.e. a $\V$-factegory with tensors) and $\G \subseteq \ob\C$ is a set, then for every $G \in \G$ and $C \in \ob\C$ we have a canonical $\C$-morphism $\C(G, C) \tensor G \to C$, namely the counit of the adjunction $(-)\otimes G \dashv \C(G,-):\C \rightarrow \V$.

\begin{defn}
\label{enrichedEMgenerator}
Let $\C$ be a cocomplete $\V$-factegory. An \textbf{enriched} $(\E, \M)$\textbf{-generator} for $\C$ is a (small) set of objects $\G \subseteq \ob\C$ such that for every $C \in \ob\C$, the canonical sink $$(\C(G, C) \tensor G \to C)_{G \in \G}$$
is $\E$-tight, i.e. the canonical morphism $\coprod_{G \in \G} \C(G, C) \tensor G \to C$ lies in $\E$. \qed
\end{defn}

\noindent We have the following initial example of an enriched $(\E, \M)$-generator:

\begin{lem}
\label{unitobjectgenerator}
The unit object $I$ is an enriched $(\E, \M)$-generator of $\V$.
\end{lem}

\begin{proof}
Noting that $\V$ is a cocomplete $\V$-factegory by \ref{VitselfVfactegory}, for every $V \in \ob\V$ the canonical $\V$-morphism $\V(I, V) \tensor I \to V$ is an isomorphism and hence certainly lies in $\E$.
\end{proof}

\noindent From \ref{moreboundingequivalent} we can deduce the following equivalent formulations of the notion of enriched $(\E, \M)$-generator, one of which generalizes a characterization of ordinary $(\E,\M)$-generators discussed in \ref{EMgen}. By \ref{functorfactegory}, any presheaf $\V$-category becomes a $\V$-factegory when equipped with the pointwise factorization system. If $\C$ is a cocomplete $\V$-category and $\G \subseteq \ob\C$ is a set, or equivalently a small full sub-$\V$-category $i:\G \hookrightarrow \C$, then the restricted Yoneda $\V$-functor $\y_\G : \C \to [\G^\op, \V]$ defined by $\y_\G(C) = \C(i-, C)$ is clearly a right-class $\V$-functor and has a left adjoint $(-) \ast i : [\G^\op, \V] \to \C$ sending $W : \G^\op \to \V$ to the colimit $G \ast i$. 

\begin{prop}
\label{enrichedEMgeneratoralt}
Let $\G$ be a small full sub-$\V$-category of a cocomplete $\V$-factegory $\C$. The following are equivalent:
\begin{enumerate}[leftmargin=*]
\item $\G$ is an enriched $(\E, \M)$-generator for $\C$.

\item The restricted Yoneda $\V$-functor $\y_\G : \C \to [\G^\op, \V]$ is $\M$-conservative. 

\item The $\V$-functors $\C(G, -) : \C \to \V$ ($G \in \G$) are jointly $\M$-conservative (i.e. $m \in \M$ is an isomorphism if each $\C(G, m)$ ($G \in \G$) is an isomorphism).
\end{enumerate}
\end{prop}

\begin{proof}
The equivalence of (2) and (3) is essentially immediate, so it remains to show that (1) and (2) are equivalent. Since $\y_\G$ preserves the right class and has a left adjoint $(-) \ast i$, it now suffices by \ref{moreboundingequivalent} to show that $\G$ is an enriched $(\E, \M)$-generator iff the counit of the adjunction $(-) \ast i \dashv \y_\G$ is pointwise in $\E$. For every $C \in \ob\C$, the domain of the counit at $C$ is $\int^{G \in \G} \C(G, C) \tensor G$. The canonical morphism $\coprod_{G \in \G} \C(G, C) \tensor G \to C$ factors through the counit $\int^{G \in \G} \C(G, C) \tensor G \to C$ via an $\E$-morphism (even a regular epimorphism; see \cite[3.68]{Kellystr}). It follows that the counit lies in $\E$ iff the canonical morphism $\coprod_{G \in \G} \C(G, C) \tensor G \to C$ lies in $\E$, which yields the result.            
\end{proof}

\begin{cor}
\label{enrichedGreflectsunions}
Let $\C$ be a cocomplete $\V$-factegory with an enriched $(\E, \M)$-generator $\G$. Then the $\V$-functors $\C(G, -) : \C \to \V$ ($G \in \G$) jointly reflect $\M$-unions, i.e. if $(m_i : C_i \to C)_{i \in I}$ is an $\M$-sink in $\C$ and $n : D \to C$ is an $\M$-subobject in $\C$ such that $\C(G, D) \cong \bigcup_i\C\left(G, C_i\right)$ as $\M$-subobjects of $\C(G, C)$ for each $G \in \G$, then $D \cong \bigcup_i C_i$ as $\M$-subobjects of $C$, and hence $n$ is a union of $(m_i : C_i \to C)_{i \in I}$.  
\end{cor}

\begin{proof}
By \ref{enrichedEMgeneratoralt} and \ref{functorreflectsunions}, the restricted Yoneda $\V$-functor $\y_\G : \C \to [\G^\op, \V]$ (which is a right adjoint right-class $\V$-functor) reflects $\M$-unions, and the result follows since $\M$-unions in $[\G^\op,\V]$ are formed pointwise.
\end{proof}

\begin{egg}\label{densegen}
If $i:\G \hookrightarrow \C$ is a small, dense, full sub-$\V$-category of a cocomplete $\V$-factegory $\C$, then the restricted Yoneda $\V$-functor $\y_\G : \C \to [\G^\op, \V]$ is fully faithful, hence conservative, hence $\M$-conservative, so certainly $\G$ is an $(\E,\M)$-generator for $\C$.  In particular, if $\A$ is a small $\V$-category, then the Yoneda lemma entails that the set $\{\A(A,-) \mid A \in \ob\A\}$ is dense in $[\A,\V]$ and so is an $(\E,\M)$-generator for $[\A,\V]$, noting that the latter $\V$-category is a cocomplete $\V$-factegory by \ref{VitselfVfactegory} and \ref{functorfactegory}. \qed
\end{egg}

\noindent The following useful result generalizes \cite[4.2]{Kellytotality} to apply to $(\E,\M)$-generators:

\begin{prop}
\label{enrichedtoordinarygenerator}
Let $\C$ be a cocomplete $\V$-factegory. If $\V_0$ has an ordinary $(\E, \M)$-generator $\G$ and $\C$ has an enriched $(\E, \M)$-generator $\h$, then $\C_0$ has an ordinary $(\E, \M)$-generator $\G \tensor \h := \left\{G \tensor H \mid G \in \G, H \in \h\right\}$.
\end{prop}

\begin{proof}
In order to employ the $\Set$-enriched version of \ref{enrichedEMgeneratoralt}, let $m : C \to C'$ be an $\M$-morphism such that $\C_0(G \tensor H, m) : \C_0(G \tensor H, C) \to \C_0(G \tensor H, C')$ is a bijection for all $G \in \G$ and $H \in \h$.  Then $\V_0(G, \C(H, m)) : \V_0(G, \C(H, C)) \to \V_0(G, \C(H, C'))$ is a bijection for all $G \in \G$ and $H \in \h$.  Since $\G$ is an ordinary $(\E, \M)$-generator for $\V_0$, it follows by the $\Set$-enriched version of \ref{enrichedEMgeneratoralt} and the fact that $\C(H, m) \in \M$ (by compatibility) that $\C(H, m) : \C(H, C) \to \C(H, C')$ is an isomorphism in $\V$ for all $H \in \h$. So because $\h$ is an enriched $(\E, \M)$-generator for $\C$, we finally deduce by \ref{enrichedEMgeneratoralt} that $m$ is an isomorphism, as desired. 
\end{proof}

\noindent We also have the following:

\begin{prop}
\label{ordinarytoenrichedgenerator}
Let $\C$ be a cocomplete $\V$-factegory. If $\h$ is an ordinary $(\E, \M)$-generator for $\C_0$, then $\h$ is also an enriched $(\E, \M)$-generator for $\C$.
\end{prop}

\begin{proof}
If $m \in \M$ and $\C(H, m)$ is an isomorphism for every $H \in \h$, then $\C_0(H, m)$ is an isomorphism in $\Set$ for every $H \in \h$, so $m$ is an isomorphism by the $\Set$-enriched version of \ref{enrichedEMgeneratoralt}. Hence $\h$ is an enriched $(\E, \M)$-generator by \ref{enrichedEMgeneratoralt}. 
\end{proof}

\noindent Proposition \ref{ordinarytoenrichedgenerator} shows that being an enriched $(\E, \M)$-generator is weaker than being an ordinary $(\E, \M)$-generator; to see that it is \emph{strictly} weaker in general, note that by \ref{unitobjectgenerator} the unit object $I$ is always itself an enriched $(\E, \M)$-generator for any closed cocomplete factegory $\V$, while $I$ need not be an \emph{ordinary} $(\E, \M)$-generator (e.g. when $\V = \Cat$ equipped with $(\StrongEpi, \Mono)$, the terminal category is not an ordinary strong generator).

We have the following result for obtaining examples of enriched $(\E, \M)$-generators:

\begin{prop}
\label{boundingEMgenerator}
Let $U : \C \to \D$ be an $\M$-conservative right-class $\V$-functor with a left adjoint $F$. If $\G$ is an enriched $(\E, \M)$-generator for $\D$, then $F\G := \left\{FG \mid G \in \G\right\}$ is an enriched $(\E, \M)$-generator for $\C$.
\end{prop}

\begin{proof}
Using \ref{enrichedEMgeneratoralt}, let $m : C \to C'$ be an $\M$-morphism in $\C$ with $\C(FG, m) : \C(FG, C) \to \C(FG, C')$ an isomorphism for every $G \in \G$, and let us show that $m$ is an isomorphism. Because $F \dashv U$, it follows for every $G \in \G$ that $\D(G, Um) : \D(G, UC) \to \D(G, UC')$ is an isomorphism. Since $U$ preserves the right class and $\G$ is an enriched $(\E, \M)$-generator for $\D$, it follows by \ref{enrichedEMgeneratoralt} that $Um$ is an isomorphism. But since $U$ is $\M$-conservative, it then follows that $m$ is an isomorphism, as desired.
\end{proof}

\subsection{Locally bounded \texorpdfstring{$\V$}{V}-categories}
\label{locbdVcategories}

Before we can define our notion of locally bounded $\V$-category, it remains to define the notion of an \emph{enriched} $\alpha$-bounded object (for a regular cardinal $\alpha$). We say that a $\V$-factegory has ($\alpha$-filtered) $\M$-unions if the underlying ordinary factegory has ($\alpha$-filtered) $\M$-unions, and that a right-class $\V$-functor between $\V$-factegories with ($\alpha$-filtered) $\M$-unions preserves ($\alpha$-filtered) $\M$-unions if the underlying ordinary functor does so. 

\begin{defn}
\label{enrichedboundedobject}
Let $C$ be an object of a $\V$-factegory $\C$ with $\M$-unions. Given a regular cardinal $\alpha$, $C$ is an \textbf{enriched} $\alpha$\textbf{-bounded object} of $\C$ if the $\V$-functor $\C(C, -) : \C \to \V$ preserves \mbox{$\alpha$-filtered} $\M$-unions.  $C$ is an \textbf{enriched bounded object} if $C$ is an enriched $\alpha$-bounded object for some $\alpha$.\qed
\end{defn}

\noindent Note that $\C(C, -)$ is a right-class $\V$-functor by \ref{reps_rclass_vfunc}.  If $C$ is an $\alpha$-bounded object of the ordinary category $\C_0$, meaning that the functor $\C_0(C, -) : \C_0 \to \Set$ preserves $\alpha$-filtered $\M$-unions, then we say that $C$ is an \emph{ordinary} $\alpha$-bounded object (of $\C_0$). 

We have the following initial examples of enriched ($\aleph_0$-)bounded objects:

\begin{lem}
\label{unitbounded}
The unit object $I \in \ob\V$ is an enriched $\aleph_0$-bounded object.
\end{lem}

\begin{proof}
Noting that $\V$ is a cocomplete $\V$-factegory by \ref{VitselfVfactegory}, the result follows immediately from the fact that $\V(I,-) : \V \to \V$ is isomorphic to the identity $\V$-functor.
\end{proof}

\begin{lem}
\label{representablesbounded}
Let $\A$ be a small $\V$-category. Then every representable $\V$-functor from $\A$ to $\V$ is an enriched $\aleph_0$-bounded object of the presheaf $\V$-category $[\A, \V]$.
\end{lem}

\begin{proof}
By \ref{densegen}, $[\A, \V]$ is a cocomplete $\V$-factegory. Given a representable $\V$-functor $\y A = \A(A,-):\A \rightarrow \V$ with $A \in \ob\A$, the $\V$-functor $[\A, \V](\y A, -) : [\A, \V] \to \V$ is (by the enriched Yoneda lemma) isomorphic to the evaluation $\V$-functor $\Ev_A : [\A, \V] \to \V$, which preserves (filtered) $\M$-unions because they are formed pointwise in $[\A, \V]$, by \ref{functorfactegory}.  
\end{proof}

We now make the central definition of this paper:

\begin{defn}
\label{locallyboundedVcategory}
Let $\alpha$ be a regular cardinal. A \textbf{locally} $\alpha$\textbf{-bounded} $\V$\textbf{-category} is a cocomplete $\V$-factegory $\C$ equipped with an enriched $(\E, \M)$-generator $\G$ consisting of enriched $\alpha$-bounded objects. A $\V$-category is \textbf{locally bounded} if it is locally $\alpha$-bounded for some regular cardinal $\alpha$. \qed
\end{defn} 

\begin{para}
\label{lbv_remarks}
When $\V = \Set$, a locally $\alpha$-bounded $\V$-category is therefore just a locally $\alpha$-bounded ordinary category. Note that a locally $\alpha$-bounded $\V$-category is also a locally $\beta$-bounded $\V$-category for any regular cardinal $\beta > \alpha$ (because an enriched $\alpha$-bounded object is also an enriched $\beta$-bounded object). \qed
\end{para}

\noindent We first show that any locally bounded $\V$-category is \emph{total} and hence complete. Recall that a $\V$-category $\C$ is \emph{total} (see \cite[5.1]{Kellytotality}) if it admits the (possibly large) colimit $W \ast 1_\C$ for every $\V$-functor $W : \C^\op \to \V$, which is equivalent to saying (see \cite[5.2]{Kellytotality}) that the enriched Yoneda embedding $\y : \C \to [\C^\op, \V]$ admits a left adjoint (where $[\C^\op, \V]$ is regarded as a $\V'$-category for some universe extension $\V'$ of $\V$). 

\begin{prop}
\label{locbdcomplete}
Every locally bounded $\V$-category $\C$ is total and complete. 
\end{prop}  

\begin{proof}
Since the $\V$-factegory $\C$ is cocomplete and has an enriched $(\E, \M)$-generator, it follows by (the remark following) \cite[Theorem 1]{Daytotality} that $\C$ is total.\footnote{Note that the proof of this result does not actually require $\V_0$ to be complete or have arbitrary intersections of monomorphisms, which are nevertheless blanket assumptions made in the introduction to \cite{Daytotality}.} Since $\V_0$ is complete, it then follows by (the proof of) \cite[5.6]{Kellytotality} that $\C$ is complete. 
\end{proof} 

\noindent We have the following easy result (cf. also \ref{VlocallyboundedVcat} below).

\begin{prop}
\label{easyVlocallybounded}
$\V$ is a locally $\aleph_0$-bounded $\V$-category with enriched $(\E, \M)$-generator $\{I\}$.
\end{prop}

\begin{proof}
$\V$ is a cocomplete $\V$-factegory by \ref{VitselfVfactegory}, and the unit object $I \in \ob\V$ is an enriched $(\E, \M)$-generator by \ref{unitobjectgenerator} and an enriched $\aleph_0$-bounded object by \ref{unitbounded}.  
\end{proof}

\noindent We can also prove that, moreover, every presheaf $\V$-category is locally bounded:

\begin{prop}
\label{locallyboundedpresheafcat}
Let $\A$ be a small $\V$-category.  Then the presheaf $\V$-category $[\A, \V]$ is locally $\aleph_0$-bounded when equipped with its pointwise factorization system and its enriched $(\E, \M)$-generator $\left\{\A(A,-) \mid A \in \ob\A\right\}$. 
\end{prop}

\begin{proof}
By \ref{densegen}, $[\A, \V]$ is a cocomplete $\V$-factegory with an $(\E,\M)$-generator $\left\{\A(A,-) \mid A \in \ob\A\right\}$, which consists of enriched $\aleph_0$-bounded objects by \ref{representablesbounded}.  
\end{proof}

\noindent We henceforth regard presheaf $\V$-categories as being locally bounded with the canonical such structure defined in \ref{locallyboundedpresheafcat}.  

\subsection{Bounding adjunctions}
\label{boundingadjunctions}

We now develop the notion of a \emph{bounding adjunction} between $\V$-factegories with $\M$-unions, which we shall use to prove characterization results for locally bounded $\V$-categories and to obtain various further examples of such $\V$-categories.

\begin{defn}
\label{boundedVfunctor}
Let $U : \C \to \D$ be a right-class $\V$-functor between $\V$-factegories with $\M$-unions. Given a regular cardinal $\alpha$, we say that $U$ is $\alpha$\textbf{-bounded} if $U : \C \to \D$ preserves $\alpha$-filtered $\M$-unions.  $U$ is \textbf{bounded} if it is $\alpha$-bounded for some regular cardinal $\alpha$. \qed 
\end{defn}

\noindent Note that a composite of $\alpha$-bounded right-class $\V$-functors is again $\alpha$-bounded.

\begin{rmk}
\label{boundedrmk}
In \ref{boundedVfunctor} we have defined boundedness of a $\V$-functor $U$ only under the prior assumption that $U$ preserves the right class.  However, boundedness can be defined without assuming preservation of the right class as follows:  If $U : \C \to \D$ is a $\V$-functor between $\V$-factegories, then (even if $U$ does not preserve the right class) one can say that $U$ is $\alpha$\emph{-bounded} if $U$ sends each $\E$-tight $\alpha$-filtered $\M$-sink to an $\E$-tight sink; we also express the latter property by saying that $U$ \textit{preserves the $\E$-tightness of $\alpha$-filtered $\M$-sinks} (cf. \cite[2.3]{Kellytrans}).  It is then straightforward to show that if $U$ happens to preserve the right class, then $U$ is bounded in this latter sense iff it is bounded in the sense of our \ref{boundedVfunctor} (cf. \cite[2.3]{Kellytrans}). In particular, every $\alpha$-bounded $\V$-functor $U : \C \to \D$ in the sense of \ref{boundedVfunctor} preserves the $\E$-tightness of $\alpha$-filtered $\M$-sinks. \qed
\end{rmk}

\begin{prop}
\label{leftadjointpreservesbounded}
Let $U:\C \rightarrow \D$ be an $\alpha$-bounded right-class $\V$-functor between $\V$-factegories with $\M$-unions, and suppose that $U$ has a left adjoint $F:\D \rightarrow \C$.  Then $F$ preserves enriched $\alpha$-bounded objects.     
\end{prop}

\begin{proof}
Let $G \in \ob\D$ be an enriched $\alpha$-bounded object of $\D$, and let us show that $FG$ is an enriched $\alpha$-bounded object of $\C$. Given an $\alpha$-filtered $\M$-sink $(m_i : C_i \to C)_{i \in I}$ in $\C$, we compute that
\[ \C\left(FG, \bigcup_i C_i\right) \cong \D\left(G, U\left(\bigcup_i C_i\right)\right) \cong \D\left(G, \bigcup_i UC_i\right) \cong \bigcup_i \D(G, UC_i) \cong \bigcup_i \C(FG, C_i) \]
as $\M$-subobjects of $\C(FG, C)$.
\end{proof}

\begin{defn}
\label{boundingadjoint}
Let $\C$ and $\D$ be $\V$-factegories with $\M$-unions, and let $\alpha$ be a regular cardinal. A $\V$-functor $U : \C \to \D$ is an $\alpha$\textbf{-bounding right adjoint} if $U$ is an $\alpha$-bounded right-class $\V$-functor with a left adjoint $F$ whose counit $\varepsilon:FU \rightarrow 1_\C$ is pointwise in $\E$. We say that $U : \C \to \D$ is a \textbf{bounding right adjoint} if it is an $\alpha$-bounding right adjoint for some regular cardinal $\alpha$. \qed
\end{defn}

\begin{rmk} 
Every bounding right adjoint $U : \C \to \D$ is automatically $\V$-faithful (and hence $U_0 : \C_0 \to \D_0$ is faithful) by \cite[Proposition 0.3]{Dubucbook}, because the counit of the adjunction is pointwise in $\E$ and hence is pointwise $\V$-epimorphic by properness.  \qed
\end{rmk} 

By \ref{moreboundingequivalent} we obtain the following characterization of bounding right adjoints:

\begin{prop}
\label{charn_bounding_radjs}
A right adjoint right-class $\V$-functor $U : \C \to \D$ is an $\alpha$-bounding right adjoint iff it is $\alpha$-bounded and $\M$-conservative, iff it is $\alpha$-bounded and reflects the left class. \qed
\end{prop}

\begin{theo}
\label{recogthm}
Let $\C$ be a cocomplete $\V$-factegory and $\G \subseteq \ob\C$ a set. Then $\C$ is a locally $\alpha$-bounded $\V$-category with enriched $(\E, \M)$-generator $\G$ iff the restricted Yoneda $\V$-functor $\y_\G : \C \to [\G^\op, \V]$ is $\alpha$-bounded and $\M$-conservative, iff $\y_\G : \C \to [\G^\op, \V]$ is an $\alpha$-bounding right adjoint.   
\end{theo}

\begin{proof}
As noted earlier, $\y_\G$ preserves the right class. It is essentially immediate from the definitions and the pointwise nature of $\M$-unions in $[\G^\op, \V]$ that every $G \in \G$ is an enriched $\alpha$-bounded object iff the $\V$-functor $\y_\G : \C \to [\G^\op, \V]$ is $\alpha$-bounded. So the first equivalence follows by \ref{enrichedEMgeneratoralt}. The second equivalence follows by \ref{charn_bounding_radjs} and the fact that $\y_\G$ is a right adjoint right-class $\V$-functor (by the remarks preceding \ref{enrichedEMgeneratoralt}).  
\end{proof}

\noindent We now prove a central property of bounding right adjoints: under certain hypotheses, any bounding right adjoint into a locally bounded $\V$-category induces a locally bounded structure on its domain.  We state the following result, and most (if not all) subsequent definitions and results, in terms of a single regular cardinal. However, it should be clear that if the assumptions of a given result are instead satisfied with respect to several distinct regular cardinals, then these assumptions will still be satisfied with respect to the largest of these regular cardinals, in view of \ref{lbv_remarks}.

\begin{theo}
\label{boundingrightadjointthm}
Let $\D$ be a locally $\alpha$-bounded $\V$-category with enriched $(\E, \M)$-generator $\G$. If $\C$ is a cocomplete $\V$-factegory and $U : \C \to \D$ is an $\alpha$-bounding right adjoint with left adjoint $F$, then $\C$ is a locally $\alpha$-bounded $\V$-category with enriched $(\E, \M)$-generator $\left\{FG \mid G \in \G\right\}$.   
\end{theo}

\begin{proof}
\ref{boundingEMgenerator} shows that $\left\{FG \mid G \in \G\right\}$ is an enriched $(\E, \M)$-generator for $\C$ (since $U$ is $\M$-conservative by \ref{charn_bounding_radjs}), and \ref{leftadjointpreservesbounded} shows that $FG$ is an enriched $\alpha$-bounded object of $\C$ for every $G \in \G$.          
\end{proof}

\begin{rmk}
\label{reflectivelocallybounded}
Theorem \ref{boundingrightadjointthm} easily entails that certain reflective sub-$\V$-categories of locally bounded $\V$-categories are locally bounded.  Indeed, if $i:\C \rightarrow \D$ is a fully faithful, right adjoint, $\alpha$-bounded right-class $\V$-functor such that $\C$ has arbitrary cointersections of $\E_\C$-morphisms and $\D$ is a locally $\alpha$-bounded $\V$-category, then $\C$ is a locally $\alpha$-bounded $\V$-category.  Indeed, $\C$ is a cocomplete $\V$-category because $\C$ is reflective in the cocomplete $\V$-category $\D$, and the counit of the reflection is an isomorphism and hence certainly lies pointwise in $\E_\C$. Note that $i : \C \rightarrow \D$ is therefore an $\alpha$-bounding right adjoint. \qed     
\end{rmk}

From \ref{recogthm} and \ref{boundingrightadjointthm} we now obtain the following characterization theorem for locally bounded $\V$-categories, which is analogous to Kelly's definition/characterization in \cite[3.1]{Kellystr} of the locally finitely presentable $\V$-categories as the cocomplete $\V$-categories that admit a right adjoint, conservative, finitary $\V$-functor into a presheaf $\V$-category. In the following result, we essentially replace ``conservative" by ``$\M$-conservative" and ``finitary" by ``$\alpha$-bounded", requiring also preservation of the right class.   

\begin{theo}
\label{recogcor}
Let $\C$ be a cocomplete $\V$-factegory. Then $\C$ is a locally $\alpha$-bounded $\V$-category iff there exists a small $\V$-category $\A$ with an $\alpha$-bounding right adjoint $U : \C \to [\A, \V]$.   
\end{theo}

\begin{proof}
If $\C$ is a locally $\alpha$-bounded $\V$-category with enriched $(\E, \M)$-generator $\G$, then $\y_\G : \C \to [\G^\op, \V]$ is an $\alpha$-bounding right adjoint by \ref{recogthm}. Conversely, if $\A$ is a small $\V$-category and $U : \C \to [\A, \V]$ is an $\alpha$-bounding right adjoint, then since $[\A, \V]$ is a locally $\aleph_0$-bounded and hence locally $\alpha$-bounded $\V$-category by \ref{locallyboundedpresheafcat}, it follows from \ref{boundingrightadjointthm} that $\C$ is a locally $\alpha$-bounded $\V$-category. 
\end{proof}

\noindent To conclude this section, we show that every $\V$-category of $\V$-functors valued in a locally bounded $\V$-category is itself locally bounded. First, we require the following result:

\begin{prop}
\label{2functorbounding}
Given $\V$-factegories $\C$ and $\D$ with $\M$-unions, if $U : \C \to \D$ is an $\alpha$-bounded right-class $\V$-functor (resp. an $\alpha$-bounding right adjoint) then $[\A, U] : [\A, \C] \to [\A, \D]$ is an $\alpha$-bounded right-class $\V$-functor (resp. an $\alpha$-bounding right adjoint).
\end{prop}

\begin{proof}
By \ref{functorfactegory}, $[\A, \C]$ and $[\A, \D]$ are $\V$-factegories, each with the pointwise factorization system and with $\M$-unions formed pointwise, from which the result follows readily, using the fact that $[\A,-]$ preserves right adjoints by 2-functoriality.
\end{proof}

\begin{prop}
\label{functorcategorylocbd}
Let $\C$ be a locally $\alpha$-bounded $\V$-category. If $\A$ is a small $\V$-category, then the functor $\V$-category $[\A, \C]$ is a locally $\alpha$-bounded $\V$-category.
\end{prop}

\begin{proof}
We know by \ref{functorfactegory} that $[\A, \C]$ is a cocomplete $\V$-factegory when equipped with the pointwise factorization system $(\E', \M')$. So it suffices by \ref{recogcor} to show that there is a small $\V$-category $\B$ with an $\alpha$-bounding right adjoint $U : [\A, \C] \to [\B, \V]$. Since $\C$ is locally $\alpha$-bounded with enriched $(\E, \M)$-generator $\G$, we know by \ref{recogthm} that $\y_\G : \C \to [\G^\op, \V]$ is an $\alpha$-bounding right adjoint. Then by \ref{2functorbounding}, it follows that $[\A, \y_\G] : [\A, \C] \to [\A, [\G^\op, \V]] \cong [\A \tensor \G^\op, \V]$ is an $\alpha$-bounding right adjoint with $\A \tensor \G^\op$ small, as desired.       
\end{proof}

\section{Locally bounded symmetric monoidal closed categories}
\label{locbdclosedsection}

In order to obtain various further results about locally bounded $\V$-categories, we shall need to assume that the base of enrichment $\V$ is a locally bounded symmetric monoidal closed category, in the following sense, which by \ref{Kellydefinition} below is equivalent to Kelly's sense of the term \cite[\S 6.1]{Kelly}.

\begin{defn}
\label{locbdclosed}
Given a regular cardinal $\alpha$, a \textbf{locally} $\alpha$\textbf{-bounded (symmetric monoidal) closed category} is a cocomplete closed factegory $\V$ equipped with an ordinary $(\E, \M)$-generator $\G$ consisting of ordinary $\alpha$-bounded objects, such that the unit object $I$ is $\alpha$-bounded and $G \tensor G'$ is $\alpha$-bounded for all $G, G' \in \G$.

A symmetric monoidal closed category $\V$ is \textbf{locally bounded (as a symmetric monoidal closed category)} if there is some regular cardinal $\alpha$ for which $\V$ is a locally $\alpha$-bounded closed category. \qed
\end{defn}

\noindent Note that if $\V$ is a locally $\alpha$-bounded closed category, then $\V$ is also a locally $\beta$-bounded closed category for any regular cardinal $\beta > \alpha$. 

\begin{rmk}
\label{Kellydefinition}
Although Kelly's original definition of locally bounded closed category (see \cite[6.1]{Kelly} and \cite{KellyLack}) omits the requirement of $\alpha$-boundedness of $I$ and of $G \otimes G'$ for all $G,G' \in \G$, a closed category $\V$ is locally bounded in the sense of our \ref{locbdclosed} iff it is locally bounded in Kelly's sense. Indeed, the `only if' implication is immediate, and if $\V$ is locally bounded in Kelly's sense with the $(\E, \M)$-generator $\G$ consisting of $\alpha$-bounded objects for some regular cardinal $\alpha$, then by \ref{every_obj_bdd} we know for every object $V \in \ob\V$ that there is some regular cardinal $\alpha_V \geq \alpha$ such that $V$ is $\alpha_V$-bounded. Since $\G$ is small, we can then find a regular cardinal $\beta \geq \alpha$ such that $I$ is $\beta$-bounded and every monoidal product of elements of $\G$ is $\beta$-bounded, so that $\V$ is locally ($\beta$-)bounded as a closed category in the sense of our \ref{locbdclosed}.  Our definition \ref{locbdclosed} entails that the $\alpha$-bounded objects are closed under the monoidal structure; see \ref{enrichedboundedprop} below.  We have augmented Kelly's definition of locally bounded closed category in this way because it enables a more convenient theoretical development, and because it more closely accords with Kelly's definition of locally $\alpha$-\emph{presentable} closed category in \cite[5.5]{Kellystr}, where the class of $\alpha$-presentable objects \emph{is} required to be closed under the monoidal structure. \qed  
\end{rmk}

\noindent We provide many examples of locally bounded closed categories in Section \ref{examples} (and also in \ref{symmonlimtheories1} and \ref{symmonlimtheories2}). We first analyze the relationship between enriched and ordinary boundedness and show that in a locally $\alpha$-bounded closed category, the monoidal product of \emph{any} two ordinary $\alpha$-bounded objects is $\alpha$-bounded (see \ref{enrichedboundedprop} below). We thus have the following sequence of results, analogous to Kelly's results \cite[5.1--5.3]{Kellystr} regarding (enriched) $\alpha$-presentable objects in cocomplete $\V$-categories over a closed category $\V$ with $\V_0$ locally $\alpha$-presentable.  

\begin{lem}
\label{enrichedboundedlemma}
Let $\V$ be a cocomplete closed factegory with an ordinary $(\E, \M)$-generator $\G$ consisting of ordinary $\alpha$-bounded objects, and let $\C$ be a tensored $\V$-factegory with $\M$-unions. The following are equivalent for every $C \in \ob\C$:
\begin{enumerate}[leftmargin=*]
\item $C$ is an enriched $\alpha$-bounded object of $\C$.
\item $V \tensor C$ is an ordinary $\alpha$-bounded object of $\C_0$ for every ordinary $\alpha$-bounded object $V$ of $\V_0$.

\item $G \tensor C$ is an ordinary $\alpha$-bounded object of $\C_0$ for every $G \in \G$. 
\end{enumerate}
\end{lem}
\begin{proof}
If (1) holds, then for each $\alpha$-bounded object $V$ of $\V_0$ and each $\alpha$-filtered $\M$-sink $(m_i : D_i \to D)_{i \in I}$ in $\C$, we compute that
\[ \V_0\left(V, \C\left(C, \bigcup_i D_i\right)\right) \cong \V_0\left(V, \bigcup_i \C(C, D_i)\right) \cong \bigcup_i \V_0\left(V, \C(C, D_i)\right) \]
as subobjects of $\V_0(V,\C(C,D))$, so $\C_0\left(V \tensor C, \bigcup_i D_i\right) \cong \bigcup_i \C_0(V \tensor C, D_i)$ as subobjects of $\C_0(V \tensor C, D)$. This proves that (1) implies (2).

(2) trivially implies (3).  Suppose that (3) holds, and let $(m_i : D_i \to D)_{i \in I}$ be an $\alpha$-filtered $\M$-sink in $\C$.  For each $G \in \G$, we compute that
\[ \V_0\left(G,\C\left(C,\bigcup_i D_i\right)\right) \cong \C_0\left(G \tensor C, \bigcup_i D_i\right) \cong \bigcup_i \C_0\left(G \tensor C, D_i\right) \cong \bigcup_i \V_0(G, \C(C, D_i)) \]
as subobjects of $\V_0(G, \C(C, D))$.  Hence, by the $\Set$-enriched version of \ref{enrichedGreflectsunions} we deduce that $\C\left(C, \bigcup_i D_i\right) \cong \bigcup_i \C(C, D_i)$ as $\M$-subobjects of $\C(C, D)$, showing that (1) holds.
\end{proof} 

\begin{prop}
\label{enrichedboundedprop}
Let $\V$ be a cocomplete closed factegory with an ordinary $(\E, \M)$-generator $\G$ consisting of ordinary $\alpha$-bounded objects. The following are equivalent:
\begin{enumerate}[leftmargin=*]
\item Every ordinary $\alpha$-bounded object of $\V_0$ is an enriched $\alpha$-bounded object of $\V$.

\item The class of ordinary $\alpha$-bounded objects of $\V_0$ is closed under $\tensor$.

\item $G \tensor G'$ is an ordinary $\alpha$-bounded object of $\V_0$ for all $G, G' \in \G$.
\end{enumerate}

\noindent Thus, if $\V$ is a locally $\alpha$-bounded closed category, then the class of ordinary $\alpha$-bounded objects is closed under $\tensor$. 
\end{prop}

\begin{proof}
(1) implies (2) because if $X, Y$ are ordinary $\alpha$-bounded objects of $\V_0$, then $Y$ is also an enriched $\alpha$-bounded object of $\V$ by (1), so that $X \tensor Y$ is an ordinary $\alpha$-bounded object of $\V_0$ by \ref{enrichedboundedlemma}. (2) trivially implies (3), so suppose (3) and let us show (1). By \ref{enrichedboundedlemma} we first obtain that every $G \in \G$ is an enriched $\alpha$-bounded object of $\V$. Now let $V$ be an ordinary $\alpha$-bounded object of $\V_0$, and let us show that $V$ is an enriched $\alpha$-bounded object of $\V$. By \ref{enrichedboundedlemma}, it suffices to show that $G \tensor V \cong V \tensor G$ is an ordinary $\alpha$-bounded object of $\V_0$ for every $G \in \G$, which now follows by \ref{enrichedboundedlemma} since every $G \in \G$ is an enriched $\alpha$-bounded object.   
\end{proof}

\begin{cor}
\label{enrichedboundedcor}
Let $\V$ be a cocomplete closed factegory with an ordinary $(\E, \M)$-generator consisting of ordinary $\alpha$-bounded objects. The following are equivalent:
\begin{enumerate}[leftmargin=*]
\item Every enriched $\alpha$-bounded object of $\V$ is an ordinary $\alpha$-bounded object of $\V_0$.

\item The unit object $I$ is an ordinary $\alpha$-bounded object of $\V_0$.

\item If $\C$ is any tensored $\V$-factegory with $\M$-unions, then every enriched $\alpha$-bounded object of $\C$ is an ordinary $\alpha$-bounded object of $\C_0$.
\end{enumerate}
\end{cor}

\begin{proof}
(1) implies (2) because $I$ is necessarily an enriched $\alpha$-bounded object of $\V$ by \ref{unitbounded}. If (2) holds and $C$ is an enriched $\alpha$-bounded object of a tensored $\V$-factegory $\C$ with $\M$-unions, then $I \tensor C \cong C$ is an ordinary $\alpha$-bounded object of $\C_0$ by \ref{enrichedboundedlemma}. And (3) implies (1) by \ref{VitselfVfactegory}.  
\end{proof}

\noindent From \ref{enrichedboundedprop} and \ref{enrichedboundedcor} we now obtain:

\begin{cor}
\label{ordinaryequalsenrichedbounded}
The ordinary $\alpha$-bounded objects coincide with the enriched $\alpha$-bounded objects in any locally $\alpha$-bounded closed category $\V$. \qed 
\end{cor}

\begin{cor}
\label{enrichedtensorbounded}
Let $\V$ be a locally $\alpha$-bounded closed category and let $\C$ be a tensored $\V$-factegory with $\M$-unions. If $V$ is an enriched $\alpha$-bounded object of $\V$ and $C$ is an enriched $\alpha$-bounded object of $\C$, then $V \tensor C$ is an enriched $\alpha$-bounded object of $\C$.
\end{cor}

\begin{proof}
By \ref{enrichedboundedlemma}, it suffices to show that $X \tensor (V \tensor C) \cong (X \tensor V) \tensor C$ is an ordinary $\alpha$-bounded object of $\C_0$ for every ordinary $\alpha$-bounded object $X$ of $\V_0$. But $V$ is also an ordinary $\alpha$-bounded object of $\V_0$ by \ref{ordinaryequalsenrichedbounded}, so that $X \tensor V$ is an ordinary $\alpha$-bounded object of $\V_0$ by \ref{enrichedboundedprop}, whence the result follows by \ref{enrichedboundedlemma}. 
\end{proof}

\begin{rmk}
\label{VlocallyboundedVcat}
If $\V$ is a locally $\alpha$-bounded closed category with ordinary $(\E, \M)$-generator $\G$, then $\V$ is also a locally $\alpha$-bounded $\V$-category with enriched $(\E, \M)$-generator $\G$ by \ref{VitselfVfactegory}, \ref{ordinarytoenrichedgenerator}, and \ref{ordinaryequalsenrichedbounded} (note that $\V_0$ is complete by the $\Set$-enriched version of \ref{locbdcomplete}). In other words, if $\V$ is locally $\alpha$-bounded as a closed category, then $\V$ is also a locally $\alpha$-bounded $\V$-category \textit{with the same data}. However, the converse is certainly \emph{not} true in general. For example, \ref{easyVlocallybounded} shows that even if $\V$ is just a cocomplete closed factegory, then $\V$ is a locally $\aleph_0$\emph{-bounded} $\V$-category whose enriched $(\E, \M)$-generator is just the unit object $I$, while $I$ need not be an \emph{ordinary} $(\E, \M)$-generator.  For example, $\V = \Cat$ is a cocomplete closed factegory with $(\E, \M) = (\StrongEpi, \Mono)$, and the unit object of $\Cat$ is the terminal category, which is not an ordinary strong generator. \qed   
\end{rmk}

\subsection{Cocomplete quasitoposes with generators}

In this subsection and the next, we prove some general results that will supply many examples of locally bounded closed categories. The reader can safely skip ahead to Section \ref{examples} if they just wish to see a list of these examples.

We first exhibit a large class of examples of locally bounded \emph{cartesian closed} categories: the cocomplete quasitoposes with generators and arbitrary cointersections of epimorphisms. Recall (see \cite[A2.6.1]{Johnstone2}) that an \emph{(elementary) quasitopos} is a finitely complete and finitely cocomplete category that is locally cartesian closed and has a subobject classifier for \emph{strong} monomorphisms. 

It is shown in \cite[C2.2.13]{Johnstone2} that any cocomplete quasitopos with a \emph{strong} generator (a generating set in the terminology of \cite[A1.2]{Johnstone2}) is locally presentable. We shall shortly prove the following variation of this result, with weaker hypothesis and weaker conclusion: any cocomplete quasitopos with arbitrary wide cointersections and a generator (a separating set in the terminology of \cite[A1.2]{Johnstone2}) is locally \emph{bounded} as a cartesian closed category. In fact, we shall first show that any cocomplete factegory with an $(\E, \M)$-generator and suitable pullback stability properties is locally bounded. 

We first require the following lemma. If $\C$ is a factegory with pullbacks, we say that $\E$-morphisms are \emph{stable under pullback} in $\C$ if the pullback of an $\E$-morphism along any morphism is still an $\E$-morphism. Such factorization systems have been called \emph{stable} in the literature (see e.g. \cite{Kellyrelations}).   Every morphism $f : C \to D$ induces a pullback functor $f^* : \C/D \to \C/C$, and since $\M$ is stable under pullback we know that $f^*$ restricts to a functor $f^{-1} : \Sub_\M(D) \to \Sub_\M(C)$ whose value at each $\M$-subobject $m:B \rightarrow D$ we write as $f^{-1}(m):f^{-1}(B) \rightarrow C$.  Now supposing also that $\C$ has small coproducts, we say that small coproducts are \emph{stable under pullback} in $\C$ if every pullback functor $f^* : \C/D \to \C/C$ preserves small coproducts, noting that small coproducts are formed in the slice categories as in $\C$.  We also say that \emph{$\M$-unions are stable under pullback} in $\C$ if for every morphism $f : C \to D$ in $\C$ and every $\M$-sink $(m_i : D_i \to D)_{i \in I}$ with union $m : \bigcup_i D_i \to D$, the pullback $f^{-1}(m) : f^{-1}\left(\bigcup_i D_i\right) \to C$ is a union of the `pullback' $\M$-sink $\left(f^{-1}(m_i) : f^{-1}(D_i) \to C\right)_{i \in I}$, equivalently, $f^{-1}\left(\bigcup_i D_i\right) \cong \bigcup_i f^{-1}(D_i)$ as $\M$-subobjects of $C$.  By \ref{union}, $\M$-unions are stable under pullback iff every pullback functor $f^{-1} : \Sub_\M(D) \to \Sub_\M(C)$ preserves small joins.

\begin{lem}
\label{pullbackspreserveunions}
Let $\C$ be a factegory with pullbacks and small coproducts. If $\E$-morphisms and small coproducts are stable under pullback in $\C$, then so are $\M$-unions.
\end{lem}

\begin{proof}
Given any morphism $f : C \to D$ in $\C$, we have the following square, where $J_C, J_D$ are the inclusions and $K_C, K_D$ are the respective left adjoints, as discussed in \ref{union}:
\[\begin{tikzcd}
	{\Sub_\M(D)} &&& {\Sub_\M(C)} \\
	\\
	{\C/D} &&& {\C/C}
	\arrow["{f^{-1}}", from=1-1, to=1-4]
	\arrow["{J_D}"', shift right=1, from=1-1, to=3-1]
	\arrow["{f^*}"', from=3-1, to=3-4]
	\arrow["{J_C}"', shift right=3, from=1-4, to=3-4]
	\arrow["{K_D}"', shift right=3, from=3-1, to=1-1]
	\arrow["{K_C}"', shift right=1, from=3-4, to=1-4]
\end{tikzcd}\]
$f^{-1}$ is a restriction of $f^*$, so $J_C \circ f^{-1} = f^* \circ J_D$. Using the pullback-stability of $\E$, the uniqueness (up to isomorphism) of $(\E, \M)$-factorizations, and a well-known pullback cancellation property, it is then straightforward to show that $f^{-1} \circ K_D \cong K_C \circ f^*$.  Also, small joins in $\Sub_\M(C)$ are formed (as noted in \ref{union}) by first taking the coproduct in $\C \slash C$ and then applying the reflector $K_C$, and similarly for $\Sub_\M(D)$.  Using these observations and the assumption that $f^*$ preserves small coproducts, it follows readily that $f^{-1} : \Sub_\M(D) \to \Sub_\M(C)$ preserves small joins. 
\end{proof}  

\noindent We can now prove the following result:

\begin{prop}
\label{prequasitoposprop}
Let $\C$ be a cocomplete factegory with an $(\E, \M)$-generator, and suppose that $\E$-morphisms and small coproducts are stable under pullback in $\C$. Then $\C$ is a locally bounded category.
\end{prop}

\begin{proof}
Note first that $\C$ is complete by \cite[2.2]{KellyLack} and thus has pullbacks. It remains to show that there is a regular cardinal $\alpha$ for which each object of the (small) $(\E, \M)$-generator $\G$ is $\alpha$-bounded, and for this it suffices to show that each object of $\G$ is bounded. By \ref{pullbackspreserveunions}, we know that $\M$-unions are stable under pullback, which then entails the result by \cite[3.1.2]{FreydKelly}.
\end{proof}

\noindent Recall that a \emph{locally cartesian closed category} can be defined as a category with pullbacks in which each pullback functor $f^*$ has a right adjoint, or equivalently as a category $\C$ whose slice categories $\C/C$ are all cartesian closed (so if $\C$ has a terminal object, then $\C$ is itself cartesian closed).   

\begin{cor}
\label{loccartesianclosedcor}
If $\C$ is a cocomplete and locally cartesian closed category with a generator and arbitrary cointersections of epimorphisms, then $\C$ is a locally bounded cartesian closed category. In particular, any cocomplete quasitopos with a generator and arbitrary cointersections of epimorphisms is a locally bounded cartesian closed category. 
\end{cor}

\begin{proof}
The assumptions entail by the dual of \cite[2.3.4]{FreydKelly} that $\C$ admits the $(\Epi, \StrongMono)$ proper factorization system, so that $\C$ is then a cocomplete factegory with an $(\Epi, \StrongMono)$-generator. To show that $\C$ is a locally bounded category, it remains by \ref{prequasitoposprop} to show that epimorphisms and small coproducts are stable under pullback, which easily follows from the pullback functors being left adjoints.    

$\C$ is cartesian closed because $\C$ is locally cartesian closed and has a terminal object by \ref{locbdcomplete}. To prove that $\C$ is locally bounded as a cartesian closed category, it remains by \ref{Kellydefinition} to show that the proper factorization system $(\Epi, \StrongMono)$ on $\C$ is $\C$-enriched, which follows because each product functor $C \times (-) : \C \to \C$ is a left adjoint (by cartesian closedness) and hence preserves epimorphisms. 
\end{proof}

\subsection{Topological categories}
\label{topologicalsection}

It is known (see e.g. \cite[2.3]{Sousa}) that any topological category over $\Set$ is locally $\aleph_0$-bounded.  In this section, we begin by discussing how every topological functor $U:\C \rightarrow \Set$ is an $\aleph_0$-bounding right adjoint, and we discuss the resulting locally $\aleph_0$-bounded structure on $\C$; we then establish wide classes of examples of locally bounded \textit{closed} categories that are topological over $\Set$.

Recall that a functor $U : \C \to \D$ between categories is \emph{topological} if every $U$-structured source in $\D$ has a $U$-initial lift; see e.g. \cite[21.1]{AHS} for an explicit definition, which we shall not employ directly. A category $\C$ is \emph{topological} over a category $\D$ if there is a topological functor $U : \C \to \D$. Any topological functor is in particular faithful (see \cite[21.3]{AHS}), and has a left adjoint functor $F : \D \to \C$ that sends each set $X$ to the \textit{discrete} object on $X$ (see \cite[21.12]{AHS}). 

\begin{prop}
\label{topologicalsetprop}
Every topological functor $U:\C \rightarrow \Set$ is an $\aleph_0$-bounding right adjoint when $\C$ is equipped with the factorization system $(\Epi, \StrongMono)$.  Consequently, every category $\C$ topological over $\Set$ is locally $\aleph_0$-bounded when equipped with the latter factorization system and the generator consisting of just the discrete object on a singleton set.
\end{prop}

\begin{proof}
Suppose $U:\C \rightarrow \Set$ is topological, and identify each morphism in $\C$ with its underlying function.  Then it is well known that the epimorphisms in $\C$ are the surjective morphisms, the strong monomorphisms in $\C$ are the $U$-initial injective morphisms (also called \textit{embeddings}), and these classes constitute a factorization system $(\E,\M)$ under which $\C$ is a cocomplete factegory (see \cite[\S 21]{AHS}).  Hence $U$ is a right-class functor and reflects epimorphisms.  Also, a sink in $\C$ is $\E$-tight iff it is jointly surjective, and it follows that $\M$-unions in $\C$ can be formed by taking the union of the underlying sink of monomorphisms in $\Set$ and equipping it with the $U$-initial structure, so $U$ preserves all $\M$-unions and, in particular, is $\aleph_0$-bounded.  The result now follows, by \ref{charn_bounding_radjs} and \ref{boundingrightadjointthm}, since $\Set$ is locally $\aleph_0$-bounded with generator $1$.
\end{proof}

\noindent We can now prove the following result, which will yield several examples of locally bounded closed categories (see Section \ref{examples}):

\begin{prop}
\label{topologicalclosed}
Let $\V$ be a symmetric monoidal closed category with a topological functor $U : \V_0 \to \Set$, and let $F$ denote the left adjoint to $U$.
\begin{enumerate}[leftmargin=*]

\item $\V$ is locally bounded as a closed category with $(\Epi, \StrongMono)$-generator $F1$. 

\item If $I \cong F1$, or equivalently if $U \cong \V_0(I, -)$, then $\V$ is locally $\aleph_0$-bounded as a closed category with $(\Epi, \StrongMono)$-generator $I$. 
\end{enumerate}
\end{prop}

\begin{proof}
We know that $\V_0$ is a locally $\aleph_0$-bounded ordinary category with $(\Epi, \StrongMono)$-generator $F1$ (which is $\aleph_0$-bounded) by \ref{topologicalsetprop}. To show that $\V$ is locally bounded as a closed category (see \ref{Kellydefinition}), it remains to show that the factorization system $(\Epi, \StrongMono)$ on $\V$ is enriched, which is true because $X \tensor (-) : \V \to \V$ (being a left adjoint) preserves epimorphisms for every $X \in \ob\V$. This proves (1).

The equivalence of the two conditions in the hypothesis of (2) follows from the fact that $U \cong \V_0(F1, -)$. Now supposing $I \cong F1$, then since $F1$ is $\aleph_0$-bounded and $I \tensor I \cong I$, we deduce (2), using (1).  
\end{proof}

\subsection{Examples of locally bounded closed categories}
\label{examples}

\begin{egg}
\label{exa_lp}
All of the symmetric monoidal closed categories of \cite[Section 1.1]{Kelly} are locally bounded closed categories, as shown on \cite[Page 115]{Kelly}. This includes any symmetric monoidal closed category $\V$ such that $\V_0$ is locally presentable, e.g. the categories of small categories, small groupoids, partially ordered sets, abelian groups, differential graded modules over a commutative ring, the two-element preorder $\mathbf{2}$, and any Grothendieck quasitopos (see \cite[C2.2.13]{Johnstone2}).  Further examples discussed in \cite[Page 115]{Kelly} include the poset $([0, \infty]^\op, +, 0)$ of non-negative extended real numbers (with the reverse ordering), the category $\mathsf{Ban}$ of Banach spaces and linear maps of norm $\leq 1$, and the topological examples $\mathsf{CGTop}, \mathsf{CGTop}_{\ast}, \mathsf{HCGTop}, \mathsf{QTop}$ of, respectively, compactly generated topological spaces and pointed such, compactly generated Hausdorff spaces, and quasitopological spaces. \qed
\end{egg}

\begin{egg}
Generalizing the posetal examples $\mathbf{2}$ and $[0,\infty]^\op$ above, let $\V$ be a \textit{commutative unital quantale} (see e.g. \cite[II.1.10]{Monoidaltop}), or  equivalently a symmetric monoidal closed category that is posetal and cocomplete. Then $\V$ is locally $\aleph_0$-bounded as a closed category.  Indeed, this is straightforwardly verified by equipping $\V$ with the trivial (yet here proper) factorization system $(\mathsf{All},\Iso)$ in which $\textsf{All}$ consists of all morphisms and $\Iso$ consists of the isomorphisms (which are the identity morphisms, noting that every object of $\V$ is then $\aleph_0$-bounded).  But $\V$ need not be locally $\aleph_0$-\textit{presentable}:  For example, the quantale $([0,\infty]^\op,+,0)$ is locally $\aleph_0$-bounded as a closed category, but $[0,\infty]^\op$ is not an algebraic lattice and so is not locally $\aleph_0$-presentable, by \cite[1.10]{LPAC}. \qed
\end{egg}

\begin{egg}
If $\V$ is a locally bounded closed category, then it is shown in \cite[5.6]{KellyLack} that $\V\text{-}\Cat$ is also a locally bounded closed category. \qed
\end{egg}

\begin{egg}
A \emph{concrete quasitopos} \cite{Dubucquasitopoi} is a category of \emph{quasispaces} (also called concrete sheaves \cite{Baezsmooth}) on a (possibly large) concrete site.  As described on \cite[Page 243]{Dubucquasitopoi}, some prominent examples of concrete quasitoposes are the categories of bornological spaces and quasitopological spaces, and categories of convergence spaces such as filter spaces, limit spaces, Choquet or pseudotopological spaces, and subsequential spaces. The categories of Chen spaces, diffeological spaces, and simplicial complexes are also concrete quasitoposes (shown in \cite{Baezsmooth}), as is the category of quasi-Borel spaces (shown in \cite{Heunenprobability}). As remarked on \cite[Page 245]{Dubucquasitopoi}, a concrete quasitopos is in particular an elementary quasitopos that is topological over $\Set$, and so we actually have two ways of showing that every concrete quasitopos $\C$ is a locally bounded cartesian closed category.  Firstly, it is well known that any topological category over $\Set$ is cocomplete and has a generator and wide cointersections of epimorphisms (as discussed in Section \ref{topologicalsection}), so $\C$ satisfies the hypotheses of the second statement of \ref{loccartesianclosedcor}.  Secondly, the associated topological functor $U : \C \to \Set$ is represented by the terminal object \cite[\S 1]{Dubucquasitopoi}, so that $\C$ is a locally $\aleph_0$\emph{-bounded} cartesian closed category by \ref{topologicalclosed}. \qed
\end{egg}

\begin{egg}
Every topological category $\C$ over $\Set$ carries a canonical symmetric monoidal closed structure whose unit object is the discrete object on a singleton set \cite[Section 3]{Wischnewsky}, \cite[Section 2.2]{Sato}.  By \ref{topologicalclosed}, it thus follows that \emph{any} topological category over $\Set$ is a locally $\aleph_0$-bounded closed category with respect to this canonical symmetric monoidal closed structure. This includes (e.g.) the category $\Top$ of topological spaces and continuous maps and the category $\mathsf{Meas}$ of measurable spaces and measurable maps (see \cite[Section 2.1]{Sato}). In the case of $\Top$, the tensor product $X \tensor Y$ of spaces $X$ and $Y$ is obtained by equipping the product of the underlying sets with the \textit{topology of separate continuity}, while the internal hom $[X, Y]$ is obtained by equipping the set of all continuous maps $X \to Y$ with the \textit{topology of pointwise convergence}; see, e.g., \cite[7.1.6]{Borceux2}. \qed      
\end{egg}

\begin{egg}
The full subcategory $\Top_{\mathcal{C}}$ of $\Top$ consisting of the $\mathcal{C}$\emph{-generated spaces} for a \emph{productive class} $\mathcal{C}$ of topological spaces \cite{EscardoCCC} is a locally $\aleph_0$-bounded cartesian closed category.  Indeed, $\Top_{\mathcal{C}}$ is concretely coreflective by the remarks following \cite[3.1]{EscardoCCC} and hence topological over $\Set$ by \cite[21.33]{AHS}, and $\Top_{\mathcal{C}}$ is cartesian closed by \cite[3.6]{EscardoCCC}; the terminal object of $\Top_\mathcal{C}$ is the usual one-point space (which is discrete and hence $\mathcal{C}$-generated), so \ref{topologicalclosed}(2) applies.  Examples of $\Top_\mathcal{C}$ include the categories of compactly generated spaces, core compactly generated spaces, locally compactly generated spaces, and sequentially generated spaces (see \cite[3.3]{EscardoCCC}). \qed
\end{egg}

\begin{egg}
In Theorem \ref{Daylocbd} below, we show that categories of models of small \textit{symmetric monoidal limit theories} in locally bounded closed categories $\V$ are themselves locally bounded closed categories.  This provides a further source of locally bounded closed categories whose objects are structures internal to any of the above locally bounded closed categories; see \ref{symmonlimtheories1} and \ref{symmonlimtheories2}, for example. \qed  
\end{egg}

\section{Enriched versus ordinary local boundedness}
\label{enrichedvsordinary}

Given a locally bounded closed category $\V$, we now study the relation between ordinary and enriched local boundedness of $\V$-categories.  Firstly, enriched implies ordinary local boundedness, but the generator changes:

\begin{theo}
\label{enrichedtoordinarylocallybounded}
Let $\V$ be a locally $\alpha$-bounded closed category with ordinary $(\E, \M)$-generator $\G$, and let $\C$ be a locally $\alpha$-bounded $\V$-category with enriched $(\E, \M)$-generator $\h$. Then $\C_0$ is a locally $\alpha$-bounded ordinary category with ordinary $(\E, \M)$-generator
$$\G \tensor \h := \left\{ G \tensor H \mid G \in \G, H \in \h\right\}.$$ 
\end{theo}

\begin{proof}
That $\C_0$ is cocomplete and has arbitrary cointersections of $\E$-morphisms follows because $\C$ has these properties. It follows from \ref{enrichedboundedlemma} that $G \tensor H$ is an ordinary $\alpha$-bounded object of $\C_0$ for all $G \in \G$ and $H \in \h$, and from \ref{enrichedtoordinarygenerator} that $\G \tensor \h$ is an ordinary $(\E, \M)$-generator for $\C_0$.   
\end{proof}

\noindent Toward a result in the opposite direction, we first prove the following lemma:

\begin{lem}
\label{enrichedunionpreservation}
Let $\V$ be a locally bounded closed category. If $\C$ is a tensored $\V$-factegory such that $\C_0$ is locally bounded, then every $C \in \ob\C$ is an enriched bounded object.     
\end{lem}

\begin{proof}
Let $C \in \ob\C$, and let $\G$ be the ordinary $(\E, \M)$-generator of $\V_0$. Since $\C_0$ is locally bounded and $\G$ is small, it follows by \ref{every_obj_bdd} that there is a regular cardinal $\alpha$ such that $G \tensor C$ is an ordinary $\alpha$-bounded object of $\C_0$ for every $G \in \G$, so that $C$ is an enriched $\alpha$-bounded object of $\C$ by \ref{enrichedboundedlemma}.      
\end{proof}

\noindent Local boundedness of the ordinary category underlying a cocomplete $\V$-factegory entails its local boundedness as a $\V$-category, but the regular cardinal $\alpha$ changes:

\begin{theo}
\label{ordinarylocallyboundedtoenriched}
Let $\V$ be a locally bounded closed category. If $\C$ is a cocomplete $\V$-factegory for which $\C_0$ is locally bounded with ordinary $(\E, \M)$-generator $\h$, then $\C$ is a locally bounded $\V$-category with enriched $(\E, \M)$-generator $\h$. 
\end{theo}

\begin{proof}
We deduce from \ref{ordinarytoenrichedgenerator}, \ref{enrichedunionpreservation}, and the smallness of $\h$ that $\h$ is an \emph{enriched} $(\E, \M)$-generator consisting of enriched $\alpha$-bounded objects for some $\alpha$.  
\end{proof}

\noindent We can now prove an enrichment of Freyd and Kelly's result \cite[3.1.2]{FreydKelly} that every object of a locally bounded ordinary category is an ordinary bounded object: 

\begin{theo}
\label{locallyboundedcardinal}
Let $\V$ be a locally bounded closed category, and let $\C$ be a locally bounded \mbox{$\V$-category}. Then every $C \in \ob\C$ is an enriched bounded object.  
\end{theo}

\begin{proof}
We deduce from \ref{enrichedtoordinarylocallybounded} that $\C_0$ is locally bounded, and then \ref{enrichedunionpreservation} yields the result.
\end{proof}

\noindent The following corollary provides a characterization of locally bounded $\V$-categories in terms of an enriched generalization of the notion of \textit{bounded category with a generator} in the sense of \cite{FreydKelly}:

\begin{cor}
Let $\V$ be a locally bounded closed category.  Then a cocomplete $\V$-factegory $\C$ is locally bounded if and only if $\C$ has an enriched $(\E,\M)$-generator and every $C \in \ob\C$ is an enriched bounded object. \qed
\end{cor}

\section{Local boundedness versus local presentability of enriched categories}
\label{boundedpresentable}

We know by e.g. \cite[3.2.3]{FreydKelly} that every locally presentable ordinary category is locally bounded. We now extend this result to the enriched context in the case where $\V$ is locally presentable. Recall from \cite[5.5]{Kellystr} that $\V$ is \emph{locally} $\alpha$\emph{-presentable as a closed category} if $\V_0$ is locally $\alpha$-presentable and the class of $\alpha$-presentable objects in $\V_0$ is closed under the monoidal product and contains the unit object.  A $\V$-category $\C$ is \emph{locally} $\alpha$\emph{-presentable} if it is cocomplete and has an enriched strong generator of enriched $\alpha$-presentable objects (see \cite[3.1, 7.4]{Kellystr}).  In referring to results in \cite{Kellystr} that are stated only for the case where $\alpha = \aleph_0$, we tacitly employ the generalizations of these results to an arbitrary $\alpha$, which are valid by \cite[7.4]{Kellystr}.

We recall that if $\C$ is a tensored and cotensored $\V$-category, then the $\V$-monomorphisms in $\C$ coincide with the monomorphisms in $\C_0$, and the enriched strong epimorphisms in $\C$ coincide with the ordinary strong epimorphisms in $\C_0$ (see \cite[6.8]{enrichedfact}).

\begin{prop}
\label{locprestolocbd}
Let $\V$ be a locally $\alpha$-presentable closed category. If $\C$ is a locally $\alpha$-presentable $\V$-category, then $\C$ is a locally $\alpha$-bounded $\V$-category with respect to $(\StrongEpi, \Mono)$.  
\end{prop}

\begin{proof}
Since $\V_0$ is locally presentable, it admits the proper (and enriched) factorization system \linebreak $(\StrongEpi, \Mono)$ by \cite[1.61]{LPAC}, so that $\V$ is a cocomplete closed factegory (since $\V_0$ is cocomplete and cowellpowered by \cite[1.58]{LPAC}). By \ref{recogcor}, it suffices to show that $\C$ is a cocomplete $\V$-factegory for which there is a small $\V$-category $\A$ and an $\alpha$-bounding right adjoint $U : \C \to [\A, \V]$. By \cite[7.5]{Kellystr}, $\C_0$ is locally $\alpha$-presentable and hence also admits the proper factorization system $(\StrongEpi, \Mono)$. This factorization system is enriched because cotensoring preserves monomorphisms by \cite[2.11]{enrichedfact} (note that $\C$ is complete by \cite[7.2]{Kellystr}), and it is compatible with $(\StrongEpi, \Mono)$ on $\V$ because each $\C(C, -) : \C \to \V$ preserves monomorphisms. Since $\C$ is a cocomplete $\V$-category and $\C_0$ is cowellpowered by \cite[1.58]{LPAC}, it follows that $\C$ is a cocomplete $\V$-factegory. 

Since $\C$ is locally $\alpha$-presentable, we know by \cite[3.1]{Kellystr} that there is a small $\V$-category $\A$ with a conservative left adjoint $\V$-functor $U : \C \to [\A, \V]$ that preserves conical $\alpha$-filtered colimits. Since the right adjoint $U$ preserves monomorphisms, it remains by \ref{charn_bounding_radjs} to show that $U$ is $\alpha$-bounded, i.e. preserves $\alpha$-filtered unions of monomorphisms. But this is true because $U$ preserves $\alpha$-filtered colimits, and $\alpha$-filtered unions of monomorphisms are examples of $\alpha$-filtered colimits in the locally $\alpha$-presentable categories $\C_0$ and $[\A, \V]_0$ by \cite[1.63]{LPAC}, noting that $[\A,\V]_0$ is locally $\alpha$-presentable by \cite[3.1, 7.5]{Kelly}.        
\end{proof}

\noindent Note that we could have also invoked \ref{ordinarylocallyboundedtoenriched} to (more easily) deduce that $\C$ is a locally bounded $\V$-category, but then we would not have been able to maintain the same cardinal bound.   

We now want to characterize when a locally bounded ($\V$-)category is locally \emph{presentable}. We first show the following result; recall that an object $C$ of a cocomplete category $\C$ is said to be $\alpha$\emph{-generated} if $\C(C, -) : \C \to \Set$ preserves the colimit of every $\alpha$-directed diagram of monomorphisms \cite[1.67]{LPAC}. 

\begin{prop}
\label{boundedtogenerated}
Let $\C$ be a cocomplete category with $(\StrongEpi, \Mono)$-factorizations. If the colimit cocone of every $\alpha$-directed diagram of monomorphisms in $\C$ consists of monomorphisms, then every $\alpha$-bounded object of $\C$ is also $\alpha$-generated.  
\end{prop}

\begin{proof}
Let $C \in \ob\C$ be $\alpha$-bounded. It is shown in \cite[3.2]{FreydKelly} that $C$ then has \emph{Barr rank} $\leq \alpha$, meaning that if $D : J \to \C$ is an $\alpha$-directed diagram on which there exists a cocone of monomorphisms, then $\C(C, -) : \C \to \Set$ preserves the colimit of $D$ (cf. also \cite[2.6]{Wolffmonads}). But if $D : J \to \C$ is an $\alpha$-directed diagram of monomorphisms, then its colimit cocone consists of monomorphisms by assumption, so that $\C(C, -) : \C \to \Set$ preserves the colimit of $D$ and thus $C$ is $\alpha$-generated.   
\end{proof}

\noindent We now have the following theorem characterizing locally presentable categories among locally bounded categories with respect to $(\StrongEpi, \Mono)$:

\begin{theo}
\label{boundedtopresentable}
A category $\C$ is locally presentable iff $\C$ is $\StrongEpi$-cowellpowered and there is a regular cardinal $\alpha$ such that $\C$ is locally $\alpha$-bounded with respect to $(\StrongEpi, \Mono)$ and the colimit cocone of every $\alpha$-directed diagram of monomorphisms in $\C$ consists of monomorphisms.  
\end{theo}

\begin{proof}
If $\C$ is locally $\alpha$-presentable, then $\C$ is locally $\alpha$-bounded with respect to $(\StrongEpi, \Mono)$ by \ref{locprestolocbd} (with $\V = \Set$), and $\C$ satisfies the colimit cocone property by \cite[1.62]{LPAC} and is cowellpowered by \cite[1.58]{LPAC}. Conversely, if $\C$ satisfies the stated properties, then $\C$ is in particular cocomplete and has a strong generator whose objects are $\alpha$-bounded and hence $\alpha$-generated by \ref{boundedtogenerated}. It now follows by Gabriel-Ulmer's definition of locally generated categories (see \cite[1.72]{LPAC}) that $\C$ is locally $\alpha$-generated, so that $\C$ is locally presentable by \cite[1.70]{LPAC}. 
\end{proof}

\begin{rmk}
\label{necessarycolimitcocone}
For $\C$ to be locally presentable, it is \emph{not} in general sufficient for $\C$ to just be ($\StrongEpi$-cowellpowered and) locally bounded with respect to $(\StrongEpi, \Mono)$, since Freyd and Kelly give in \cite[5.2.3]{FreydKelly} an example (which they attribute to John Isbell) of a category that is locally bounded with respect to $(\StrongEpi, \Mono)$ but is \emph{not} locally presentable (at least assuming the non-existence of measurable cardinals). Hence, the colimit cocone property must be imposed. \qed 
\end{rmk}

\noindent We can now prove an enrichment of \ref{boundedtopresentable} if the base $\V$ is locally presentable:

\begin{theo}
\label{enrichedboundedtopresentable}
Let $\V$ be a locally presentable closed category, and let $\C$ be a $\V$-category.  The following are equivalent: (1) $\C$ is locally presentable; (2) $\C$ is $\StrongEpi$-cowellpowered, and there is some regular cardinal $\alpha$ such that $\C$ is locally $\alpha$-bounded with respect to $(\StrongEpi, \Mono)$ and the colimit cocone of every $\alpha$-directed diagram of monomorphisms in $\C_0$ consists of monomorphisms. 
\end{theo}

\begin{proof}
Each of the properties in (2) is stable under passing to a higher cardinal $\beta > \alpha$, and we now tacitly use this throughout.  Suppose (1).  Then $\C_0$ is locally presentable by \cite[7.5]{Kellystr}, so that $\C$ is $\StrongEpi$-cowellpowered and satisfies the colimit cocone property by \ref{boundedtopresentable}. Moreover, it follows by \ref{locprestolocbd} that $\C$ is a locally bounded $\V$-category with respect to $(\StrongEpi, \Mono)$. 

Conversely, suppose (2).  Since $\V$ is a locally bounded closed category by \ref{exa_lp}, it follows by \ref{enrichedtoordinarylocallybounded} that $\C_0$ is locally bounded with respect to $(\StrongEpi, \Mono)$, which then entails by \ref{boundedtopresentable} that $\C_0$ is locally presentable, and hence has an ordinary strong generator $\h$. Since $\C$ is a cocomplete $\V$-factegory, it follows by \ref{ordinarytoenrichedgenerator} that $\h$ is also an \emph{enriched} strong generator for $\C$.  Now let $H \in \h$.   $\V_0$ has a strong generator $\G$ of ordinary $\alpha$-presentable objects for some $\alpha$, and for each $G \in \G$ the tensor $G \tensor H$ is an ordinary presentable object of $\C_0$ by \cite[p. 22]{LPAC}. Since $\G$ is small, there is a regular cardinal $\beta_H \geq \alpha$ such that $G \tensor H$ is an ordinary $\beta_H$-presentable object of $\C_0$ for every $G \in \G$. So then $H$ is an enriched $\beta_H$-presentable object of $\C$ by  \cite[5.1]{Kellystr}. Since $\h$ is small, there is then a regular cardinal $\beta$ such that every $H \in \h$ is an enriched $\beta$-presentable object, which (since $\C$ is cocomplete) entails that $\C$ is locally ($\beta$-)presentable.     
\end{proof}

We conclude this subsection by considering the relationship between locally $\alpha$-bounded $\V$-categories and the $\M$-locally $\alpha$-generated $\V$-categories of \cite{Libertienriched}. If $(\E, \M)$ is an enriched factorization system on a cocomplete $\V$-category $\C$ and $\alpha$ is a regular cardinal, then $(\E, \M)$ is said to be $\alpha$\emph{-convenient} \cite[4.3]{Libertienriched} if $\C$ is $\E$-cowellpowered and for every $\alpha$-directed diagram $D : I \to \C_0$ of $\M$-morphisms, every colimit cocone for $D$ consists of $\M$-morphisms, and the factorizing morphism from $\colim \ D$ to the vertex of any cocone consisting of $\M$-morphisms itself lies in $\M$. An object $C \in \ob\C$ is an \emph{enriched} $\alpha$\emph{-generated object w.r.t.} $\M$ \cite[4.1]{Libertienriched} if $\C(C, -) : \C \to \V$ preserves conical $\alpha$-directed colimits of $\M$-morphisms. Finally, a cocomplete $\V$-category $\C$ with an $\alpha$-convenient enriched factorization system $(\E, \M)$ is $\M$\emph{-locally} $\alpha$\emph{-generated} \cite[4.4]{Libertienriched} if $\C$ has a set $\G$ of enriched $\alpha$-generated objects w.r.t. $\M$ such that every object of $\C$ is a conical $\alpha$-directed colimit of objects from $\G$ and morphisms from $\M$. If $\V$ is locally $\alpha$-presentable as a closed category, then every $\M$-locally $\alpha$-generated $\V$-category is in fact locally presentable by \cite[4.14]{Libertienriched}, and hence is complete. We now show the following result, which requires $(\E, \M)$ to be \emph{proper}:

\begin{prop}
\label{Libertienrichedprop}
Let $\V$ be locally $\alpha$-presentable as a closed category, and let $\C$ be a cocomplete $\V$-category with an $\alpha$-convenient enriched proper factorization system $(\E, \M)$. If $\C$ is $\M$-locally $\alpha$-generated, then $\C$ is locally $\alpha$-bounded with respect to $(\E, \M)$. 
\end{prop} 

\begin{proof}
We first deduce as in the proof of \ref{locprestolocbd} that $\V$ is a closed cocomplete factegory with respect to $(\StrongEpi, \Mono)$. The properness of $(\E, \M)$ entails that $(\E, \M)$ is compatible with $(\StrongEpi, \Mono)$, so that $\C$ is a cocomplete $\V$-factegory because $\C$ is cocomplete and $\E$-cowellpowered. By \cite[4.17]{Libertienriched}, we deduce that $\C_0$ has an ordinary strong generator $\G$ consisting of enriched $\alpha$-generated objects w.r.t. $\M$. So then $\G$ is an enriched strong generator for $\C$ by \ref{ordinarytoenrichedgenerator}, and hence is an enriched $(\E, \M)$-generator because $\E$ contains all strong epimorphisms (by properness).  So it remains to show that $\G$ consists of enriched $\alpha$-bounded objects, for which it suffices to show that every enriched $\alpha$-generated object w.r.t. $\M$ is an enriched $\alpha$-bounded object. To show this, we first prove that $\alpha$-filtered $\M$-unions in $\C$ can be defined in terms of $\alpha$-directed colimits of $\M$-morphisms, in the following sense. Let $(m_i : C_i \to C)_{i \in I}$ be an $\alpha$-filtered $\M$-sink in $\C$. Then $(m_i)_i$ induces an $\alpha$-directed diagram of $\M$-morphisms $D : I \to \C_0$ by setting $i \leq j$ ($i, j \in I$) iff $m_i$ factors (uniquely) through $m_j$. Let $(s_i : C_i \to \colim \ D)_{i \in I}$ be a colimit cocone for this diagram. By $\alpha$-convenience of $(\E, \M)$, we know that each $s_i$ ($i \in I$) lies in $\M$, and that the factorizing morphism $\colim \ D \xrightarrow{m} C$ induced by the $\M$-cocone $(m_i)_i$ lies in $\M$. Then because the original $\M$-sink $(m_i)_i$ factors through $m$ via the $\E$-tight colimit sink $(s_i : C_i \to \colim \ D)_{i \in I}$, it follows that $\colim \ D \xrightarrow{m} C$ is a union of the $\M$-sink $(m_i)_i$. 

Now if $X \in \ob\C$ is an enriched $\alpha$-generated object and $(m_i : C_i \to C)_{i \in I}$ is an $\alpha$-filtered $\M$-sink with union $\colim_i \ C_i \xrightarrow{m} C$ (as just shown), then \[ \colim_i \ \C(X, C_i) \cong \C(X, \colim_i \ C_i) \xrightarrow{\C(X, m)} \C(X, C) \] is a union of the sink of monomorphisms $(\C(X, m_i) : \C(X, C_i) \to \C(X, C))_i$, because this sink factors through the displayed monomorphism via the $\E$-tight colimit sink $\left(\C(X, C_i) \to \colim_i \ \C(X, C_i)\right)_i$.         
\end{proof}  

\begin{rmk}
If we do not assume properness of $(\E, \M)$ in \ref{Libertienrichedprop}, then we cannot \emph{a priori} maintain the same cardinal bound or factorization system in the conclusion. What we \emph{do} have is that if $\V$ is locally $\alpha$-presentable as a closed category and $\C$ is a cocomplete $\V$-category with an $\alpha$-convenient (but not necessarily proper) enriched factorization system $(\E, \M)$, then $\C$ is locally bounded with respect to $(\StrongEpi, \Mono)$ if $\C$ is $\M$-locally $\alpha$-generated. Indeed, we deduce from \cite[4.14]{Libertienriched} that $\C$ is locally $\beta$-presentable for some $\beta \geq \alpha$, so that $\C$ is then locally $\beta$-bounded with respect to $(\StrongEpi, \Mono)$ by \ref{locprestolocbd}. \qed 
\end{rmk}

\section{Adjoint functor and representability theorems}
\label{representability}

It is well known that locally \emph{presentable} categories satisfy useful adjoint functor theorems: namely, a functor between locally presentable categories has a left adjoint iff it preserves small limits and also preserves $\alpha$-filtered colimits (i.e. has rank $\alpha$) for some regular cardinal $\alpha$; see e.g. \cite[1.66]{LPAC}. It is also well known that such categories satisfy a useful representability theorem: a $\Set$-valued functor on a locally presentable (even accessible) category is representable iff it preserves small limits and has rank (by \cite[4.88]{Kelly} and \cite[5.3.7, 5.5.5]{Borceux2}). We now wish to show that locally bounded categories also satisfy useful adjoint functor and representability theorems. Recall from \cite[Page 79]{Kelly} that a functor $P : \C \to \Set$ from an ordinary category $\C$ is \emph{weakly accessible} if there is a small set $\h \subseteq \ob\C$ such that for any $C \in \ob\C$ and $x \in PC$, there exist $H \in \h$, $y \in PH$, and $f : H \to C$ with $(Pf)(y) = x$. 

\begin{prop}
\label{weaklyaccessibleprop}
Let $\C$ be a locally bounded and $\E$-cowellpowered category, and let $P : \C \to \Set$ be a bounded right-class functor.  Then $P$ is weakly accessible.
\end{prop}

\begin{proof}
Let $\alpha$ be a regular cardinal for which $\C$ is locally $\alpha$-bounded and $P$ is $\alpha$-bounded. Since $\C$ is $\E$-cowellpowered, it follows by \ref{Sousathm} that there is a small set $\h \subseteq \ob\C$ of $\alpha$-bounded objects with the property that every object of $\C$ is an $\alpha$-filtered union of $\M$-subobjects with domains in $\h$. Now let $C \in \ob\C$ and $x \in PC$. We know that there is an $\alpha$-filtered $\M$-sink $(m_i : H_i \to C)_{i \in I}$ with each $H_i \in \h$ ($i \in I$) such that $C$ is a union of $(m_i)_i$, so that $(m_i)_i$ is $\E$-tight (see the remarks preceding \ref{Sousathm}). Since $P$ is $\alpha$-bounded and hence preserves the $\E$-tightness of $\alpha$-filtered $\M$-sinks by \ref{boundedrmk}, it follows that the functions $Pm_i : PH_i \to PC$ ($i \in I$) are jointly surjective. So because $x \in PC$, there are $i \in I$ and $y \in PH_i$ with $(Pm_i)(y) = x$, as desired. 
\end{proof}

\noindent We now have the following useful (ordinary) representability theorem:

\begin{theo}
\label{representabilitythm}
Let $\C$ be a locally bounded and $\E$-cowellpowered category, and let $P : \C \to \Set$ be a right-class functor. Then $P$ is representable iff $P$ is bounded and preserves small limits. 
\end{theo}

\begin{proof}
Note that $\C$ is complete by \ref{locbdcomplete}. If $P \cong \C(C, -)$ for some $C \in \ob\C$, then $P$ certainly preserves small limits and $P$ is bounded because $C$ is bounded by \ref{locallyboundedcardinal}. Conversely, if $P$ is bounded and preserves small limits, then $P$ is weakly accessible by \ref{weaklyaccessibleprop}, so that $P$ is representable by \cite[4.88]{Kelly}.   
\end{proof}

\noindent We can enrich \ref{representabilitythm} as follows:

\begin{theo}
\label{enrichedrepresentabilitythm}
Let $\V$ be a locally bounded closed category, let $\C$ be a locally bounded and $\E$-cowellpowered $\V$-category, and let $P : \C \to \V$ be a right-class $\V$-functor.  Then $P$ is representable iff $P$ is bounded and preserves small limits.
\end{theo}

\begin{proof}
If $P \cong \C(C, -)$ for some $C \in \ob\C$, then $P$ certainly preserves small limits and is bounded because $C$ is an enriched bounded object by \ref{locallyboundedcardinal}. Conversely, if $P$ is bounded and preserves small limits, then $P_0 : \C_0 \to \V_0$ is also bounded and preserves small limits, and $V = \V_0(I, -) : \V_0 \to \Set$ has these properties as well (since $I$ is an ordinary bounded object of $\V_0$). So the composite functor $\C_0 \xrightarrow{P_0} \V_0 \xrightarrow{V} \Set$ is bounded and preserves small limits, and hence is representable by \ref{representabilitythm} (since $\C_0$ is locally bounded by \ref{enrichedtoordinarylocallybounded}). Since $\C$ is cotensored by \ref{locbdcomplete} and $P$ preserves cotensors by assumption, it then follows by \cite[4.85]{Kelly} that $P$ is representable, as desired.  
\end{proof}

\noindent We now wish to use the representability theorem \ref{enrichedrepresentabilitythm} to deduce adjoint functor theorems for locally bounded enriched categories. We first require the following:

\begin{prop}
\label{rightadjointhasrank}
Let $\C$ and $\D$ be locally bounded $\V$-categories over a locally bounded closed category $\V$, and let $U : \C \to \D$ be a right adjoint right-class $\V$-functor.  Then $U$ is bounded.
\end{prop}

\begin{proof}
Let $F : \D \to \C$ be the left adjoint of $U$. We must find a regular cardinal $\gamma$ for which $U$ preserves $\gamma$-filtered $\M$-unions. For every $H$ in the enriched $(\E, \M)$-generator $\h$ of $\D$, we know by \ref{locallyboundedcardinal} that $FH \in \ob\C$ is an enriched $\beta_H$-bounded object for some regular cardinal $\beta_H$. Now let $\gamma$ be a regular cardinal greater than each $\beta_H$ (which is possible because $\h$ is small), and let us prove that $U$ is $\gamma$-bounded. So let $(m_i : C_i \to C)_{i \in I}$ be a $\gamma$-filtered $\M$-sink in $\C$, and let us show that $U\left(\bigcup_i C_i\right) \cong \bigcup_i UC_i$ as $\M$-subobjects of $UC$. By \ref{enrichedGreflectsunions}, it suffices to show for every $H \in \h$ that $\D\left(H, U\left(\bigcup_i C_i\right)\right) \cong\bigcup_i \D(H, UC_i)$ as $\M$-subobjects of $\D(H, UC)$. But because $FH$ is $\gamma$-bounded, we have isomorphisms of $\M$-subobjects
\[ \D\left(H, U\left(\bigcup_i C_i\right)\right) \cong \C\left(FH, \bigcup_i C_i\right) \cong \bigcup_i \C(FH, C_i) \cong \bigcup_i \D(H, UC_i). \]      
\end{proof}

\noindent We now have the following useful adjoint functor theorem for locally bounded enriched categories:

\begin{theo}
\label{enrichedadjfunctorthm}
Let $\V$ be a locally bounded closed category, let $\C$ and $\D$ be locally bounded $\V$-categories with $\C$ being $\E$-cowellpowered, and let $U : \C \to \D$ be a right-class $\V$-functor.  Then $U$ has a left adjoint iff $U$ is bounded and preserves small limits.  
\end{theo}

\begin{proof}
If $U$ has a left adjoint, then $U$ certainly preserves small limits and is also bounded by \ref{rightadjointhasrank}. Conversely, suppose $U$ is bounded and preserves small limits.  To show that $U$ has a left adjoint, it is equivalent to show for every $D \in \ob\D$ that the $\V$-functor $\D(D, U-) : \C \to \V$ is representable.  But $\D(D,-):\D \rightarrow \V$ is a bounded right-class $\V$-functor by \ref{locallyboundedcardinal}, so the composite functor $\D(D, U-)$ is a bounded right-class $\V$-functor since $U$ is so.  Hence, since $\D(D, U-)$ also preserves small limits (since $U$ and $\D(D, -)$ both do), we deduce from \ref{enrichedrepresentabilitythm} that $\D(D, U-)$ is representable.  
\end{proof}

\noindent In special case where $\V = \Set$, note that the proof of the preceding theorem shows that, under the hypotheses of the theorem, if $U$ is bounded then the functor $\D(D, U-) : \C \to \Set$ is bounded for each $D \in \ob\D$, so by \ref{weaklyaccessibleprop} each $\D(D,U-)$ is weakly accessible, which means precisely that $U$ satisfies Freyd's \textit{solution set condition} \cite[Chapter 3, Exercise J]{Freydabelian}.

We also have the following useful result for obtaining right adjoints:

\begin{prop}
\label{otherenrichedadjfunctorthm}
Let $\V$ be a closed cocomplete factegory such that $\V_0$ is complete and has an $(\E, \M)$-generator, let $\C$ be a cocomplete $\V$-factegory with enriched $(\E, \M)$-generator and $\D$ an arbitrary $\V$-category, and let $F : \C \to \D$ be a $\V$-functor. Then $F$ has a right adjoint iff $F$ preserves small colimits.  
\end{prop}

\begin{proof}
For the less obvious direction, given that $\C_0$ is a cocomplete factegory with ordinary $(\E, \M)$-generator by \ref{enrichedtoordinarygenerator}, it follows by \cite[2.1]{KellyLack} that $F_0 : \C_0 \to \D_0$ has a right adjoint. Since $\C$ is tensored and $F$ preserves tensors, we then deduce from the dual of \cite[4.85]{Kelly} that $F$ has a right adjoint.  
\end{proof}

\section{Commutation of \texorpdfstring{$\alpha$}{alpha}-bounded-small limits and \texorpdfstring{$\alpha$}{alpha}-filtered unions}
\label{commutation}

In this section, we establish results about commutation of suitably small limits and suitably filtered unions in locally bounded (enriched) categories, in analogy with results about commutation of suitably small limits and suitably filtered colimits in locally presentable (enriched) categories (see e.g. \cite[1.59]{LPAC} and \cite[4.9]{Kellystr}).

\begin{assumption}
For the remainder of the paper, we suppose that $\V$ is a locally $\alpha$-bounded closed category (which is therefore complete by \ref{locbdcomplete}). \qed
\end{assumption}

\noindent We first define the notion of an $\alpha$\emph{-bounded-small weight} enriched in $\V$, which is analogous to Kelly's definition \cite[4.1]{Kellystr} of a finite weight (or finite indexing type) enriched in a locally finitely presentable closed category.

\begin{defn}
\label{alphasmallweight}
A small \mbox{$\V$-category} $\B$ is \textbf{$\alpha$-bounded-small} if the cardinality of
$\ob\B$ is less than\footnote{\cite[4.1]{Kellystr} just requires less than $\alpha$ \textit{isomorphism classes} of objects, so that the category $\B$ need not be small in the strict sense but rather \textit{essentially small}; here we instead employ strict notions of smallness and $\alpha$-smallness.} $\alpha$ and, for all $B,B' \in \ob\B$, the hom-object $\B(B, B')$ is an enriched $\alpha$-bounded object of $\V$.  A weight $W : \B \to \V$ is $\alpha$\textbf{-bounded-small} if $\B$ is an $\alpha$-bounded-small $\V$-category and, for all $B \in \ob\B$, the object $WB$ is an enriched $\alpha$-bounded object of $\V$. \qed 
\end{defn}

\begin{rmk}
\label{boundedsmallrmk}
By invoking \ref{locallyboundedcardinal} and the fact that every small set of regular cardinals is bounded above by some regular cardinal, it follows that every small weight $W : \B \to \V$ is $\beta$-bounded-small for some $\beta \geq \alpha$. \qed 
\end{rmk}

\begin{rmk}
\label{alphaboundedrmk}
One might wonder why we have chosen to use the term ``$\alpha$-bounded-small weight" rather than ``$\alpha$-bounded weight". While every $\alpha$-bounded-small weight $W : \B \to \V$ will be an enriched $\alpha$-bounded object of the presheaf $\V$-category $[\B, \V]$ (see \ref{alphasmallweightcor} below), the converse need not be true. For example, if $\B$ is a small $\V$-category with $\geq \alpha$ objects and $W : \B \to \V$ is representable, then $W$ will be an enriched $\alpha$-bounded object of $[\B, \V]$ by \ref{representablesbounded} without being an $\alpha$-bounded-small weight. So we have chosen the term \textit{$\alpha$-bounded-small} to avoid the false implication that the weights considered in \ref{alphasmallweight} are exactly the enriched $\alpha$-bounded objects of presheaf $\V$-categories. \qed   
\end{rmk}

Our main objective of this section is to show that $\alpha$-bounded-small limits commute with $\alpha$-filtered $\M$-unions in every locally $\alpha$-bounded $\V$-category over the locally $\alpha$-bounded closed category $\V$ (see \ref{maincommutationresult}). We first give a precise definition of this notion of commutation:

\begin{defn}
\label{commutationdefn}
Let $\C$ be a $\V$-factegory with $\M$-unions, and let $W : \B \to \V$ be a small weight for which $\C$ has $W$-limits. The $W$-limit $\V$-functor $\{W, -\} : [\B, \C] \to \C$ preserves the right class by \cite[4.5]{enrichedfact}. We say that $W$\textbf{-limits commute with} $\alpha$\textbf{-filtered} $\M$\textbf{-unions in} $\C$ if $\{W, -\} : [\B, \C] \to \C$ preserves $\alpha$-filtered $\M$-unions (or equivalently, is $\alpha$-bounded). \qed
\end{defn} 

\noindent We have the following initial class of examples of $\alpha$-bounded-small weights:

\begin{prop}
\label{cotensorsboundedsmall}
The weights for $\alpha$-bounded cotensors (i.e. cotensors by enriched $\alpha$-bounded objects of $\V$) are $\alpha$-bounded-small.
\end{prop}

\begin{proof}
If $V \in \ob\V$ is an enriched $\alpha$-bounded object, then the corresponding weight $[V] : \II \to \V$ is $\alpha$-bounded-small (where $\II$ is the unit $\V$-category) because $\ob\II = \{\ast\}$ is finite, the unique hom-object $\II(\ast, \ast) = I$ is an enriched $\alpha$-bounded object of $\V$ by \ref{unitbounded}, and the object $[V](\ast) = V$ is an enriched $\alpha$-bounded object of $\V$ by assumption. 
\end{proof}

\noindent In Theorem \ref{maincommutationresult} we show that $\alpha$-bounded-small limits commute with $\alpha$-filtered unions in any locally $\alpha$-bounded $\V$-category over the locally $\alpha$-bounded $\V$.  We begin with the following special case:

\begin{prop}
\label{unions_commute_alpha_bdd_cot}
Let $\C$ be a locally $\alpha$-bounded $\V$-category.  Then $\alpha$-bounded cotensors commute with $\alpha$-filtered $\M$-unions in $\C$.  Equivalently, for every enriched $\alpha$-bounded object $V \in \ob\V$, the cotensor $\V$-functor $[V, -] : \C \to \C$ is $\alpha$-bounded. 
\end{prop}  

\begin{proof}
Note that $\C$ is indeed cotensored by \ref{locbdcomplete}, and that $[V, -] : \C \to \C$ preserves the right class since $\M$-morphisms are stable under cotensors (Section \ref{definitionsbasicresults}). Let $(m_i : C_i \to C)_{i \in I}$ be an $\alpha$-filtered $\M$-sink in $\C$. By \ref{enrichedGreflectsunions}, it suffices to show for every $H \in \h$ (the enriched $(\E, \M)$-generator of $\C$) that $\C\left(H, \left[G, \bigcup_i C_i\right]\right) \cong \bigcup_i \C(H, [G, C_i])$ as $\M$-subobjects of $\C(H, [G, C])$. Since $V$ is an enriched $\alpha$-bounded object of $\V$ and $H$ is an enriched $\alpha$-bounded object of $\C$, we have isomorphisms of $\M$-subobjects
\[ \V\left(G, \C\left(H, \bigcup_i C_i\right)\right) \cong \V\left(G, \bigcup_i \C(H, C_i)\right) \cong \bigcup_i \V(G, \C(H, C_i)), \] so that the desired result holds by the cotensor adjunction. 
\end{proof}

\noindent Next we establish another class of examples of $\alpha$-bounded-small limits, namely \textit{$\alpha$-small conical limits}, but even before doing so we show that $\alpha$-small conical limits commute with $\alpha$-filtered $\M$-unions in every locally $\alpha$-bounded $\V$-category over the locally $\alpha$-bounded closed category $\V$ (see \ref{unionscommutelocbd}). For this we require the following two lemmas.  

\begin{lem}
\label{limitsunionscommuteSet}
$\alpha$-small conical limits commute with $\alpha$-filtered unions of monomorphisms in $\Set$.
\end{lem}

\begin{proof}
Since $\Set$ is locally $\alpha$-presentable, it follows that $\alpha$-small limits commute with $\alpha$-filtered colimits in $\Set$. But it also follows by \cite[1.63]{LPAC} that $\alpha$-filtered unions of monomorphisms in $\Set$ are certain $\alpha$-filtered colimits, which yields the result.  
\end{proof}

\begin{lem}
\label{commutationlemma}
If $\alpha$-small conical limits commute with $\alpha$-filtered $\M$-unions in $\V$, then they do so in every locally $\alpha$-bounded $\V$-category. 
\end{lem}

\begin{proof}
Let $\C$ be a locally $\alpha$-bounded $\V$-category (which is complete by \ref{locbdcomplete}) and let $\A$ be an $\alpha$-small ordinary category and $(m_i : F_i \to F)_{i \in I}$ an $\alpha$-filtered $\M$-sink in $[\A_\V, \C]$ (where $\A_\V$ is the free $\V$-category on $\A$), and let us show that $\limit\left(\bigcup_i F_i\right) \cong \bigcup_i \limit F_i$ as $\M$-subobjects of $\limit \ F$. By \ref{enrichedGreflectsunions}, it suffices to show for every $H \in \h$ (the enriched $(\E, \M)$-generator of $\C$) that $\C\left(H, \limit\left(\bigcup_i F_i\right)\right) \cong \bigcup_i \C(H, \limit F_i)$ as $\M$-subobjects of $\C(H, \limit F)$. We have the following isomorphisms of $\M$-subobjects, as required:
\[ \C\left(H, \limit\left(\bigcup_i F_i\right)\right) \cong \limit \ \C\left(H, \bigcup_i F_i-\right) \cong \limit\bigcup_i \C(H, F_i-) \cong \bigcup_i \limit \ \C(H, F_i-) \cong \bigcup_i \C(H, \limit F_i). \]      
The first isomorphism exists because $\C(H, -) : \C \to \V$ preserves conical limits and unions in $[\A_\V, \C]$ are formed pointwise, the second because $H$ is an enriched $\alpha$-bounded object of $\C$, the third by assumption on $\V$, and the last again because $\C(H, -)$ preserves conical limits.     
\end{proof}

\begin{cor}
\label{as_limits_comm_w_afilt_unions_in_lbcats}
$\alpha$-small conical limits commute with $\alpha$-filtered $\M$-unions in every locally $\alpha$-bounded ordinary category.
\end{cor}
\begin{proof}
This follows from \ref{limitsunionscommuteSet} and \ref{commutationlemma} (in the $\Set$-enriched case).
\end{proof}

\begin{theo}
\label{unionscommutelocbd}
If $\C$ is a locally $\alpha$-bounded $\V$-category, then $\alpha$-small conical limits and \mbox{$\alpha$-bounded} cotensors commute with $\alpha$-filtered $\M$-unions in $\C$.
\end{theo}  

\begin{proof}
By \ref{as_limits_comm_w_afilt_unions_in_lbcats}, $\alpha$-small conical limits commute with $\alpha$-filtered $\M$-unions in $\V$, and the result follows by \ref{commutationlemma} and  \ref{unions_commute_alpha_bdd_cot}.
\end{proof}

\noindent Before we can show that the weights for $\alpha$-small conical limits are $\alpha$-bounded-small, we also require the following important lemma:

\begin{lem}
\label{boundedconicalcolimits}
Let $\C$ be a cocomplete $\V$-factegory. The full sub-$\V$-category of $\C$ consisting of the enriched $\alpha$-bounded objects is closed under $\alpha$-small conical colimits. Explicitly, if $F : \A \to \C_0$ is a functor with $\A$ an $\alpha$-small ordinary category and each $FA$ ($A \in \ob\A$) is an enriched $\alpha$-bounded object of $\C$, then the conical colimit $\colim \ F$ is an enriched $\alpha$-bounded object of $\C$.  
\end{lem}

\begin{proof}
For each $\alpha$-filtered $\M$-sink $(m_i : C_i \to C)_{i \in I}$ in $\C$, we compute that
\begin{align*}
&\ \ \ \C\left(\colim_A \ FA, \bigcup_i C_i\right) \\
&\cong \limit_A \ \C\left(FA, \bigcup_i C_i\right) && \\
&\cong \limit_A \ \bigcup_i \C(FA, C_i) &&(\text{since $FA$ is an enriched } \alpha\text{-bounded object}) \\
&\cong \bigcup_i \limit_A \ \C(FA, C_i) &&(\text{by \ref{unionscommutelocbd}}) \\
&\cong \bigcup_i \C(\colim_A \ FA, C_i) &&     
\end{align*}
as $\M$-subobjects of $\C(\colim_A \ FA, C)$.
\end{proof}

\noindent We now have the following additional example of $\alpha$-bounded-small weights:

\begin{prop}
\label{alphasmalllem}
The weights for $\alpha$-small conical limits are $\alpha$-bounded-small. 
\end{prop}

\begin{proof}
Let $\A$ be an $\alpha$-small ordinary category, and write $\Delta_I : \A_\V \to \V$ for the associated conical weight.  Then $\ob\A_\V = \ob\A$ is $\alpha$-small, and $\Delta_I A = I$ is an enriched $\alpha$-bounded object for every $A \in \ob\A$ by \ref{unitbounded}.  For all $A, B \in \ob\A$, the hom-object $\A_\V(A, B)$ is an $\alpha$-small copower $\A(A, B) \cdot I$ of the enriched $\alpha$-bounded object $I$ of $\V$ and so is an enriched $\alpha$-bounded object by \ref{boundedconicalcolimits}.
\end{proof}

\begin{lem}
\label{lem_clos_abdd_objs}
Let $\A$ be a full sub-$\V$-category of a cocomplete $\V$-factegory $\C$.  Then the closure of $\A$ in $\C$ under $\alpha$-bounded-small colimits coincides with the closure of $\A$ in $\C$ under $\alpha$-small conical colimits and $\alpha$-bounded tensors.
\end{lem}
\begin{proof}
Let $W:\B^\op \rightarrow \V$ be an $\alpha$-bounded-small weight, and let $D:\B \rightarrow \C$ a $\V$-functor taking its values in $\A$.  Then $W \ast D = \int^{B \in \B} \ WB \tensor DB$, and by \cite[3.68]{Kelly} this coend is a conical coequalizer of the form
\[\begin{tikzcd}
	{\coprod_{B, B' \in \B} \ \B(B, B') \otimes \left(WB' \otimes DB\right)} && {\coprod_{B \in \B} \ WB \otimes DB} && {\int^{B \in \B} \ WB \otimes DB}.
	\arrow[shift left=1, from=1-1, to=1-3]
	\arrow[shift right=1, from=1-1, to=1-3]
	\arrow[two heads, from=1-3, to=1-5]
\end{tikzcd}\] 
Since $W$ is $\alpha$-bounded-small, it follows by \ref{enrichedtensorbounded} that $W * D$ lies in the closure of $\A$ in $\C$ under $\alpha$-small conical colimits and $\alpha$-bounded tensors.  On the other hand, $\alpha$-small conical colimits and $\alpha$-bounded tensors are examples of $\alpha$-bounded-small colimits by \ref{cotensorsboundedsmall} and \ref{alphasmalllem}, and the result follows.
\end{proof}

\noindent It now follows that enriched $\alpha$-bounded objects are closed under $\alpha$-bounded-small colimits (analogously to \cite[4.14]{Kellystr}):

\begin{prop}
\label{closurealphasmallcolimits}
Let $\C$ be a cocomplete $\V$-factegory. The full sub-$\V$-category $\C_\alpha$ of $\C$ consisting of the enriched $\alpha$-bounded objects is closed under $\alpha$-bounded-small colimits.
\end{prop}

\begin{proof}
By \ref{enrichedtensorbounded} and \ref{boundedconicalcolimits}, $\C_\alpha$ is closed under $\alpha$-small conical colimits and $\alpha$-bounded tensors, so by \ref{lem_clos_abdd_objs} it is closed under $\alpha$-bounded-small colimits.
\end{proof}

\begin{cor}
\label{alphasmallweightcor}
Every $\alpha$-bounded-small weight $W : \B \to \V$ is an enriched $\alpha$-bounded object of $[\B, \V]$.
\end{cor}

\begin{proof}
If $\y : \B^\op \to [\B, \V]$ is the enriched Yoneda embedding, then we have $W \cong W \ast \y$ by \cite[3.17]{Kelly}, whence the result follows by \ref{closurealphasmallcolimits}, since $[\B, \V]$ is a cocomplete $\V$-factegory by \ref{functorfactegory} and every representable $\y B = \B(B, -)$ ($B \in \ob\B$) is an enriched $\alpha$-bounded object of $[\B, \V]$ by \ref{representablesbounded}.     
\end{proof}

\noindent We now wish to provide an equivalent characterization of the existence and preservation of $\alpha$-bounded-small limits, similar to Kelly's analogous result \cite[4.3]{Kellystr} for finite weighted limits. Recall from e.g. \cite[2.8]{KS} that the \emph{saturation} $\Phi^*$ of a class of small weights $\Phi$ is defined as follows: a small weight $W$ belongs to $\Phi^*$ iff every $\Phi$-complete $\V$-category is $W$-complete and every $\Phi$-continuous $\V$-functor between $\Phi$-complete $\V$-categories is $W$-continuous.

\begin{theo}
\label{alphasmallweightalt}
The saturation of the class of $\alpha$-bounded-small weights is equal to the saturation of the class of weights for $\alpha$-small conical limits and $\alpha$-bounded cotensors.

Therefore, a $\V$-category $\C$ has $\alpha$-bounded-small limits iff $\C$ has $\alpha$-small conical limits and $\alpha$-bounded cotensors, and a $\V$-functor $F : \C \to \D$ between $\V$-categories with $\alpha$-bounded-small limits preserves such limits iff $F$ preserves $\alpha$-small conical limits and $\alpha$-bounded cotensors. 
\end{theo} 

\begin{proof}
The second assertion follows from the first by the definition of saturation.  Let $\Phi_\alpha$ be the class of $\alpha$-bounded-small weights, and let $\Psi_\alpha$ be the class of $\alpha$-small conical weights and weights for $\alpha$-bounded cotensors.  Given a small $\V$-category $\A$, if $\Phi$ is any class of small weights, then by \cite[Theorem 5.1]{AKclosure} (also see \cite[3.8]{KS}) a weight $W:\A^\op \rightarrow \V$ lies in the saturation $\Phi^*$ iff $W$ lies in the closure $\Phi(\A)$ of the representables under $\Phi$-colimits in $[\A^\op,\V]$.  But by \ref{lem_clos_abdd_objs}, $\Phi_\alpha(\A) = \Psi_\alpha(\A)$ for every small $\V$-category $\A$, so $\Phi_\alpha^* = \Psi_\alpha^*$.
\end{proof}

\noindent We now show our main (and final) result of this section:

\begin{theo}
\label{maincommutationresult}
Let $\C$ be a locally $\alpha$-bounded $\V$-category. Then $\alpha$-bounded-small limits commute with $\alpha$-filtered $\M$-unions in $\C$, in the sense that if $W : \B \to \V$ is an $\alpha$-bounded-small weight, then the $W$-limit $\V$-functor $\{W, -\} : [\B, \C] \to \C$ preserves $\alpha$-filtered $\M$-unions.
\end{theo}

\begin{proof}
Note that $\C$ does indeed have $W$-limits by \ref{locbdcomplete}. Let $(m_i : D_i \to D)_{i \in I}$ be an $\alpha$-filtered $\M$-sink in $[\B, \C]$. Then for each object $H$ of the enriched $(\E, \M)$-generator $\h$ of $\C$,
\begin{align*}
&\ \ \ \C\left(H, \left\{W, \bigcup_i D_i\right\}\right) \\
&\cong [\B, \V]\left(W, \C\left(H, \bigcup_i D_i-\right)\right) && \\
&\cong [\B, \V]\left(W, \bigcup_i \C(H, D_i-)\right) &&(H \text{ is an enriched } \alpha\text{-bounded object of } \C) \\
&\cong \bigcup_i [\B, \V]\left(W, \C(H, D_i-)\right) &&(W \text{ is an enriched } \alpha\text{-bounded object of } [\B, \V] \text{ by } \ref{alphasmallweightcor}) \\
&\cong \bigcup_i \C\left(H, \{W, D_i\}\right) &&   
\end{align*}
as $\M$-subobjects of $\C(H, \{W, D\})$, and hence $\left\{W, \bigcup_i D_i\right\} \cong \bigcup_i \left\{W, D_i\right\}$ as $\M$-subobjects of $\{W, D\}$, by \ref{enrichedGreflectsunions}.
\end{proof}

\section{Reflectivity and local boundedness of enriched orthogonal subcategories}
\label{orth_subcats}

In this section, we extend to the enriched context the classic results of Freyd and Kelly \cite[4.1.3, 4.2.2]{FreydKelly} on the reflectivity and local boundedness of orthogonal subcategories of locally bounded categories.  Indeed, we show in Theorem \ref{orthoreflective} that certain enriched orthogonal subcategories of arbitrary locally bounded $\V$-categories are reflective, and are locally bounded under an additional cowellpoweredness assumption.  The reflectivity theorem of Freyd and Kelly \cite[4.1.3]{FreydKelly} had been proved under a cowellpoweredness assumption, while as discussed just before Theorem 6.5 in \cite{Kelly}, Kelly later showed that the latter assumption can be omitted from the reflectivity result, and our Theorem \ref{orthoreflective} also enriches this refined result of Kelly.  Kelly also showed in \cite[Theorem 6.5]{Kelly} that if $\V$ is a locally bounded closed category, then certain enriched orthogonal subcategories of \textit{presheaf} $\V$-categories are reflective, and our Theorem \ref{orthoreflective} also generalizes this by replacing presheaf $\V$-categories with arbitrary locally bounded $\V$-categories.

If $\theta : M \to N$ is a morphism in a $\V$-category $\B$, then $\theta$ is $\V$-\emph{orthogonal} to an object $B \in \ob\B$, which we also write as $\theta \perp_\V B$, if the $\V$-morphism $\B(\theta, B) : \B(N, B) \to \B(M, B)$ is an isomorphism. If $\B$ is tensored, then it is remarked on \cite[Page 117]{Kelly} that $\theta$ is $\V$-orthogonal to $B$ iff for every $V \in \ob\V$ the morphism $V \tensor \theta$ is orthogonal to $B$ in the ordinary sense, i.e. $\Set$-orthogonal to $B$. If $\Theta$ is a class of morphisms in $\B$, then we say that $\Theta$ is $\V$-orthogonal to an object $B \in \ob\B$, which we also write as $\Theta \perp_\V B$, if $\theta \perp_\V B$ for every $\theta \in \Theta$. We write $\Theta^{\perp_\V}$ for the full sub-$\V$-category of $\B$ consisting of the objects to which $\Theta$ is $\V$-orthogonal, and we write $\Theta^\perp$ for the full sub-$\V$-category of objects of $\B$ to which $\Theta$ is orthogonal in the ordinary sense. So if $\B$ is tensored, then we have $\Theta^{\perp_\V} = \left\{V \tensor \theta : V \in \ob\V, \theta \in \Theta\right\}^\perp$.  We call $\Theta^{\perp_\V}$ the \textbf{$\V$-orthogonal sub-$\V$-category} of $\B$ described by $\Theta$.

We now recall the seminal results of Freyd and Kelly for orthogonal subcategories of locally bounded ordinary categories, which we both use and generalize in (the proof of) our enriched result \ref{orthoreflective}; \ref{FreydKellytheorems} comes from (the proofs of) \cite[4.1.3, 4.2.2]{FreydKelly}, while \ref{Kellytheorem} is the result quoted before \cite[6.5]{Kelly}. Note that while \ref{Kellytheorem} dispenses with the cowellpoweredness assumption of \ref{FreydKellytheorems}, it stops short of proving the local boundedness of the given orthogonal subcategory as in \ref{FreydKellytheorems}. 

\begin{theo}[Freyd, Kelly \cite{FreydKelly}]
\label{FreydKellytheorems}
Let $\B$ be an $\E$-cowellpowered locally bounded category, let $\Theta$ be a class of morphisms in $\B$ with $\Theta = \Phi \cup \Psi$ where $\Phi$ is small and $\Psi \subseteq \E$, and let $\C = \Theta^\perp$ be the orthogonal subcategory of $\B$ described by $\Theta$.  Then $\C$ is a reflective subcategory of $\B$ and has a proper factorization system $(\E_\C, \M_\C)$ with $\M_\C = \M \cap \mor \C$, which makes $\C$ into an $\E_\C$-cowellpowered locally bounded category.  Also, if $\beta$ is a regular cardinal for which the domain of every morphism in $\Phi$ is $\beta$-bounded, then the inclusion $\C \hookrightarrow \B$ is $\beta$-bounded. \qed
\end{theo} 

\begin{theo}[Kelly \cite{Kelly}]
\label{Kellytheorem}
Let $\Theta = \{\theta : M_\theta \to N_\theta\}$ be a class of morphisms in a locally bounded category $\B$ such that the class $\{N_\theta \mid \theta \in \Theta, \theta \notin \E\}$ is essentially small, and let $\C = \Theta^\perp$ be the orthogonal subcategory of $\B$ described by $\Theta$. Then $\C$ is a reflective subcategory of $\B$. \qed  
\end{theo}

To generalize these results to the setting of locally bounded enriched categories, we employ the following technical lemma that we shall also apply later in Section \ref{sketches}. This lemma is a generalization of an argument employed by Kelly in the proof of \cite[6.5]{Kelly}, as well as a generalization and enrichment of Freyd and Kelly's \cite[5.1.1]{FreydKelly}. 

\begin{lem}
\label{bifunctorlemma}
Let $\X, \Y, \Z$ be $\V$-factegories with $\X$ and $\Z$ cocomplete, and suppose that $\X$ has an enriched $(\E, \M)$-generator $\G$ and that $\Z$ has a terminal object.  Let $\Theta$ be a class of morphisms in $\Y$, and let $\ast : \X \tensor \Y \to \Z$ be a $\V$-functor such that each $(-) \ast Y : \X \to \Z$ $(Y \in \ob\Y)$ preserves colimits and preserves the left class.  Let $\Delta$ be the class of $\Z$-morphisms $X \ast \theta$ for $X \in \ob\X$ and $\theta \in \Theta$, and let $\Delta_1$ be the class of all $\Z$-morphisms $G \ast \theta$ with $G \in \G$ and $\theta \in \Theta$. Then there is a class of $\Z$-morphisms $\Omega \subseteq \E$ such that $\Delta^{\perp_\V} = \left(\Delta_1 \cup \Omega\right)^{\perp_\V}$.   
\end{lem}

\begin{proof}
Let us write $\theta : M_\theta \to N_\theta$ for each $\theta \in \Theta$, and for each $X \in \ob\X$, let $RX := \coprod_{G \in \G} \X(G, X) \tensor G$ and $\kappa_X : RX \to X$ be the canonical morphism, which lies in $\E$ (because $\G$ is an enriched $(\E, \M)$-generator).  Consider the following diagram in $\Z$, where the inner square is a pushout:
\[\begin{tikzcd}
	{RX \ast M_\theta} && {RX \ast N_\theta} \\
	\\
	{X \ast M_\theta} && {Y_{X, \theta}} \\
	&&& {X \ast N_\theta}
	\arrow["{RX \ast \theta}", from=1-1, to=1-3]
	\arrow["{\kappa_X \ast M_\theta}"', from=1-1, to=3-1]
	\arrow["{s_{X, \theta}}", from=1-3, to=3-3]
	\arrow["{r_{X, \theta}}"', from=3-1, to=3-3]
	\arrow["{p_{X, \theta}}"{description}, dashed, from=3-3, to=4-4]
	\arrow["{\kappa_X \ast N_\theta}"{description}, from=1-3, to=4-4]
	\arrow["{X \ast \theta}"{description}, from=3-1, to=4-4]
\end{tikzcd}\]
Now $p_{X, \theta} \circ s_{X, \theta} = \kappa_X \ast N_\theta \in \E$ since $\kappa_X \in \E$, so $p_{X, \theta} \in \E$ by properness.  Hence, letting $\Omega := \left\{ p_{X, \theta} \mid X \in \ob\X, \theta \in \Theta\right\}$, we have $\Omega \subseteq \E$. We shall require the following:

\medskip

\noindent\textit{Claim.} If $Z \in \ob\Z$ and $Z \in \Delta_1^{\perp_\V}$, then $r_{X, \theta} \perp_\V Z$ for all $X \in \ob\X$ and $\theta \in \Theta$.

\medskip

\noindent\textit{Proof of Claim.}
We have $RX \ast M_\theta = \left(\coprod_{G \in \G} \X(G, X) \tensor G\right) \ast M_\theta \cong \coprod_{G \in \G} \X(G, X) \tensor (G \ast M_\theta)$ because $(-) \ast M_\theta : \X \to \Z$ preserves colimits, and similarly $RX \ast N_\theta \cong \coprod_{G \in \G} \X(G, X) \tensor (G \ast N_\theta)$, so that $RX \ast \theta$ is (isomorphic to) the morphism 
\[ \coprod_{G \in \G} \X(G, X) \tensor (G \ast M_\theta) \xrightarrow{\coprod_{G \in \G} \X(G, X) \tensor (G \ast \theta)} \coprod_{G \in \G} \X(G, X) \tensor (G \ast N_\theta). \] Now, the class $Z^{\uparrow_\V}$ of morphisms of $\Z$ that are $\V$-orthogonal to $Z$ is the left class of an enriched prefactorization system (\S \ref{definitionsbasicresults}) by \cite[3.5]{enrichedfact}, since $Z^{\uparrow_\V}$ is the class of morphisms that are $\V$-orthogonal to the unique morphism from $Z$ to the (conical) terminal object of $\Z$. Since $Z \in \Delta_1^{\perp_\V}$, we know for each $G \in \G$ that $G \ast \theta \perp_\V Z$, i.e. that $G \ast \theta \in Z^{\uparrow_\V}$, which then entails by \cite[4.6]{enrichedfact} that $\X(G, X) \tensor (G \ast \theta) \in Z^{\uparrow_\V}$ for each $G \in \G$. By the dual of \cite[4.4]{enrichedfact}, it then follows that $RX \ast \theta \perp_\V Z$, in view of the above representation of $RX \ast \theta$. Since $r_{X, \theta}$ is the pushout of $RX \ast \theta$ along $\kappa_X \ast M_\theta$, the same result then entails that $r_{X, \theta} \perp_\V Z$, and the Claim is thus proved.

\medskip

We now prove that $\Delta^{\perp_\V} = \left(\Delta_1 \cup \Omega\right)^{\perp_\V}$.  Let $Z \in \ob\Z$.   Firstly, if $Z \in \Delta^{\perp_\V}$, then certainly $Z \in \Delta_1^{\perp_\V}$, since $\Delta_1 \subseteq \Delta$, so for every $X \in \ob\X$ and $\theta \in \Theta$ we know that $r_{X, \theta} \perp_\V Z$ by the Claim, but $p_{X, \theta} \circ r_{X, \theta} = X \ast \theta$, and $X \ast \theta \perp_\V Z$ since $Z \in \Delta^{\perp_\V}$, so $p_{X, \theta} \perp_\V Z$ by the dual of \cite[4.4]{enrichedfact}, showing that $Z \in \left(\Delta_1 \cup \Omega\right)^{\perp_\V}$.  Conversely, if $Z \in \left(\Delta_1 \cup \Omega\right)^{\perp_\V}$ then $Z \in \Delta_1^{\perp_\V}$, so $r_{X, \theta} \perp_\V Z$ by the Claim, but $Z \in \Omega^{\perp_\V}$ and hence $p_{X, \theta} \perp_\V Z$, so $X \ast \theta = p_{X, \theta} \circ r_{X, \theta}$ is $\V$-orthogonal to $Z$ since $Z^{\uparrow_\V}$ is closed under composition.   
\end{proof}

We now prove our main result of this section, which is an enrichment of Freyd and Kelly's results \ref{FreydKellytheorems} and \ref{Kellytheorem}, as well as a generalization and extension of Kelly's result \cite[6.5]{Kelly} from \emph{presheaf} $\V$-categories to \emph{arbitrary} locally bounded $\V$-categories.

\begin{theo}
\label{orthoreflective}
Let $\B$ be a locally $\alpha$-bounded $\V$-category, and let $\Theta = \{\theta : M_\theta \to N_\theta\}$ be a class of morphisms in $\B$ such that the class $\{N_\theta : \theta \in \Theta, \ \theta \notin \E\}$ is essentially small. Let $\C = \Theta^{\perp_\V}$ be the $\V$-orthogonal sub-$\V$-category of $\B$ described by $\Theta$.
\begin{enumerate}
\item $\C$ is a reflective sub-$\V$-category of $\B$.

\item If $\B$ is $\E$-cowellpowered, then $\C$ is itself a locally bounded $\V$-category that is $\E_\C$-cowellpowered, and the inclusion $\C \hookrightarrow \B$ is a bounding right adjoint. If furthermore $M_\theta$ is an enriched $\alpha$-bounded object of $\B$ for each $\theta \in \Theta$ with $\theta \notin \E$, then $\C$ is a locally $\alpha$-bounded $\V$-category, and $\C \hookrightarrow \B$ is an $\alpha$-bounding right adjoint.
\end{enumerate} 
\end{theo}

\begin{proof}
By \cite[3.5]{Completion}, $\C$ is closed under cotensors in $\B$, and $\B$ is cotensored by \ref{locbdcomplete}, so to show that the inclusion $\V$-functor $i : \C \hookrightarrow \B$ has a left adjoint, it suffices by \cite[4.85]{Kelly} to show that $\C_0$ is reflective in $\B_0$.

Let $\G$ be the ordinary $(\E, \M)$-generator of $\V_0$.  By the $\Set$-enriched version of \ref{bifunctorlemma} with $\X = \V_0, \Y = \B_0, \Z = \B_0$ and $\ast = \tensor$ (where each $(-) \tensor B : \V_0 \to \B_0$ ($B \in \ob\B$) preserves the left class by \ref{compatiblelemma} and certainly preserves colimits), there is a class of $\B$-morphisms $\Omega \subseteq \E$ such that \[ \C = \Theta^{\perp_\V} = \left\{X \tensor \theta \mid X \in \ob\V, \theta \in \Theta\right\}^\perp = \left(\Delta \cup \Omega\right)^\perp \] (note the difference in superscripts), where $\Delta := \left\{G \tensor \theta \mid G \in \G, \theta \in \Theta\right\}$. Now let \linebreak $\Lambda := \left\{G \tensor \theta \mid G \in \G, \theta \in \Theta, \theta \notin \E\right\}$, and let $\Omega_1 := \Omega \cup \left\{G \tensor \theta : G \in \G, \theta \in \Theta \cap \E\right\}$.  We still have $\Omega_1 \subseteq \E$ because $\E$ is stable under tensoring (since $(\E, \M)$ is enriched). Then we clearly have $\C = \left(\Delta \cup \Omega\right)^\perp = \left(\Lambda \cup \Omega_1\right)^\perp = \Lambda^\perp \cap \Omega_1^\perp$. Now let $\Lambda'$ be the class of $\M$-morphisms obtained from the $(\E, \M)$-factorizations of the morphisms in $\Lambda$, let $\Lambda''$ be the class of $\E$-morphisms so obtained, and let $\Omega_2 := \Omega_1 \cup \Lambda''$, so that we still have $\Omega_2 \subseteq \E$. By \cite[4.1.2]{FreydKelly} we have $\Lambda^\perp = \left(\Lambda' \cup \Lambda''\right)^\perp$, and hence we obtain
$$
\begin{array}{lclclcl}
\C & = & \Lambda^\perp \cap \Omega_1^\perp & = & \left(\Lambda' \cup \Lambda''\right)^\perp \cap \Omega_1^\perp & = &  \left(\Lambda'\right)^\perp \cap \left(\Lambda''\right)^\perp \cap \Omega_1^\perp\\
   & = & \left(\Lambda'\right)^\perp \cap \left(\Lambda'' \cup \Omega_1\right)^\perp & = & \left(\Lambda'\right)^\perp \cap \Omega_2^\perp & = & \left(\Lambda' \cup \Omega_2\right)^\perp.
\end{array}
$$
Since $\G$ is small and by hypothesis we may assume that $\{N_\theta : \theta \in \Theta, \ \theta \notin \E\}$ is small, but $\B_0$ is $\M$-wellpowered by \cite[2.5.2]{FreydKelly}, so it follows that $\Lambda'$ is small.

By \ref{enrichedtoordinarylocallybounded}, $\B_0$ is a locally $\alpha$-bounded ordinary category, so it now follows by \ref{Kellytheorem} that $\C_0 =\left(\Lambda' \cup \Omega_2\right)^\perp \hookrightarrow \B_0$ is a reflective subcategory, which entails that $\C$ is reflective in $\B$, proving (1).          

Now assume in addition that $\B$ is $\E$-cowellpowered, and let us prove that $\C$ is a locally bounded and $\E$-cowellpowered $\V$-category. Since $\B$ is cocomplete (and complete, by \ref{locbdcomplete}), it follows by reflectivity that the $\V$-category $\C$ is cocomplete (and complete). By \ref{FreydKellytheorems} and the fact that $\C_0 = \left(\Lambda' \cup \Omega_2\right)^\perp$ with $\Lambda'$ small and $\Omega_2 \subseteq \E$, we have an ordinary proper factorization system $(\E_\C, \M_\C)$ on $\C_0$ with $\M_\C = \M_\B \cap \mor \C_0$ such that $\C$ is $\E_\C$-cowellpowered. It follows readily that $\C$ is a $\V$-factegory since $\B$ is so.

Therefore $\C$ is a cocomplete $\V$-factegory and $i : \C \hookrightarrow \B$ is a fully faithful, right adjoint right-class $\V$-functor.  Let $\beta$ be the smallest regular cardinal $\beta \geq \alpha$ such that the domain of every morphism in $\Lambda'$ is an ordinary $\beta$-bounded object of $\B_0$, which is possible because $\B_0$ is locally bounded (see \ref{locallyboundedcardinal}) and $\Lambda'$ is small. It then follows by \ref{FreydKellytheorems} that $i:\C = \left(\Lambda' \cup \Omega_2\right)^\perp \hookrightarrow \B$ is $\beta$-bounded.  Hence by \ref{reflectivelocallybounded}, $i:\C \hookrightarrow \B$ is a $\beta$-bounding right adjoint and $\C$ is a locally $\beta$-bounded $\V$-category.

If we also know that $M_\theta$ is an enriched $\alpha$-bounded object of $\B$ for each $\theta \in \Theta$ with $\theta \notin \E$, then we can take $\beta = \alpha$. To prove this, it suffices to show that the domain of every morphism in $\Lambda'$ is an ordinary $\alpha$-bounded object of $\B_0$. But each morphism $m \in \Lambda'$ is the $\M$-component of an $(\E, \M)$-factorization
$$G \tensor \theta = \left(G \tensor M_\theta \xrightarrow{e} B \xrightarrow{m} G \tensor N_\theta\right)$$
for some $G \in \G$ and $\theta \in \Theta, \theta \notin \E$.  Now $G$ is an ordinary $\alpha$-bounded object of $\V_0$ and $M_\theta$ is an enriched $\alpha$-bounded object of $\B$, so $G \tensor M_\theta$ is an ordinary $\alpha$-bounded object of $\B_0$ by \ref{enrichedboundedlemma}. It now follows by \cite[2.5]{Sousa} that $B$ is an ordinary $\alpha$-bounded object of $\B_0$ as desired.
\end{proof}

\section{Reflectivity and local boundedness of \texorpdfstring{$\V$}{V}-categories of models for theories}
\label{sketches}

We begin this section by using Theorem \ref{orthoreflective} to show that the $\V$-category of models of an enriched limit sketch in an arbitrary locally bounded $\V$-category $\C$ is reflective in the corresponding $\C$-valued functor $\V$-category and is itself locally bounded and $\E$-cowellpowered if $\C$ is $\E$-cowellpowered.

With the terminology of \cite[6.3]{Kelly}, an \emph{enriched limit sketch} is a pair $(\A, \sketch)$ consisting of a small $\V$-category $\A$ and a class $\sketch$ of cylinders (i.e. $\V$-natural transformations) $\varphi_\gamma : F_\gamma \to \A(A_\gamma, D_\gamma-)$ for $\gamma \in \Gamma$, where $\Gamma$ is a (not necessarily small) class, $F_\gamma : \K_\gamma \to \V$ and $D_\gamma : \K_\gamma \to \A$ are $\V$-functors with $\K_\gamma$ small, and $A_\gamma \in \ob\A$. If $\C$ is a $\V$-category, then a $\V$-functor $M : \A \to \C$ is a \emph{model} of the sketch $(\A, \sketch)$ in $\C$, or a $\Psi$\emph{-model} in $\C$, if for every $\gamma \in \Gamma$ the composite cylinder
\[ F_\gamma \xrightarrow{\varphi_\gamma} \A(A_\gamma, D_\gamma-) \xrightarrow{M} \C(MA_\gamma, MD_\gamma-) \] presents $MA_\gamma$ as a limit $\{F_\gamma,MD_\gamma\}$. We let $\sketch\Mod(\A, \C)$ be the full sub-$\V$-category of the functor $\V$-category $[\A, \C]$ on the models of $(\A, \sketch)$ in $\C$. We now show that $\sketch\Mod(\A, \C)$ can be represented as an enriched orthogonal subcategory of $[\A, \C]$; this was previously shown for $\C = \V$ by Kelly in \cite[6.11]{Kelly}, and in the unenriched general case by Freyd and Kelly in \cite[1.3.1]{FreydKelly}.

\begin{para}\label{param_tensor}
Given a small $\V$-category $\A$ and a tensored $\V$-category $\C$, we write
$$\tensorbar\;:\;[\A,\V] \otimes \C \longrightarrow [\A,\C]$$
for the $\V$-functor defined by $(F \tensorbar C)A = FA \tensor C$, $\V$-naturally in $F \in [\A,\V], C \in \C, A \in \A$.  For each object $C$ of $\C$, the $\V$-functor $(-) \tensorbar C$ may be obtained also by applying the 2-functor $[\A,-]:\V\CAT \rightarrow \V\CAT$ to the $\V$-functor $(-) \tensor C : \V \to \C$, so that $(-) \tensorbar C = [\A, (-) \tensor C]:[\A, \V] \rightarrow [\A, \C]$ and hence
\begin{equation}\label{eq:tensorbar_c}(-) \tensorbar C \dashv [\A, \C(C, -)]\;:\;[\A,\C] \longrightarrow [\A,\V]\end{equation}
because $(-) \tensor C \dashv \C(C, -)$ and the $2$-functor $[\A, -]$ preserves adjunctions.  Therefore
$$[\A,\C](F \tensorbar C,D) \;\cong\; [\A,\V](F,\C(C,D-)) \;\cong\; \C(C,\{F,D\})$$
$\V$-naturally in $F \in [\A,\V]$, $C \in \C$, $D \in [\A,\C]$.  In particular,
\begin{equation}\label{eq:ftensorbar_ladj}F \tensorbar (-) \dashv \{F, -\} : [\A, \C] \longrightarrow \C\end{equation}
for each $F : \A \to \V$. \qed
\end{para}

\begin{lem}
\label{presheaftensorbounded}
Let $\C$ be a tensored $\V$-factegory with $\M$-unions and $\A$ a small $\V$-category. If $F : \A \to \V$ is an enriched $\alpha$-bounded object of $[\A, \V]$ and $C \in \ob\C$ is an enriched $\alpha$-bounded object of $\C$, then $F \tensorbar C : \A \to \C$ is an enriched $\alpha$-bounded object of $[\A, \C]$.  
\end{lem}

\begin{proof}
Since $\C(C, -) : \C \to \V$ is an $\alpha$-bounded right-class $\V$-functor, we find that $[\A, \C(C, -)]:[\A,\C] \rightarrow [\A,\V]$ is an $\alpha$-bounded right-class $\V$-functor by \ref{2functorbounding}, so the result follows by \eqref{eq:tensorbar_c} and \ref{leftadjointpreservesbounded}.
\end{proof}

\begin{rmk}
\label{Kellysketch}
Given an enriched limit sketch $(\A, \Psi)$ as above, Kelly showed in \cite[6.11]{Kelly} that $\Psi\Mod(\A, \V)$ is the $\V$-orthogonal sub-$\V$-category $\Theta_\Psi^{\perp_\V}$ of $[\A,\V]$ described by a class of morphisms $\Theta_\Psi$ in $[\A, \V]$ defined as follows. For each $\gamma \in \Gamma$, we obtain a  $\V$-functor $\K_\gamma^\op \xrightarrow{D_\gamma^\op} \A^\op \xrightarrow{\y} [\A, \V]$ with a colimit $F_\gamma \ast \y D_\gamma^\op \in [\A, \V]$. Writing
$$\theta_\gamma : F_\gamma \ast \y D_\gamma^\op \longrightarrow \y A_\gamma$$
to denote the canonical comparison morphism obtained by applying $\y$ to the cylinder $\varphi_\gamma$, we take
$$\Theta_\Psi \;:=\; \left\{ \theta_\gamma \ \mid \gamma \in \Gamma\right\}\;.$$ \qed
\end{rmk} 

\begin{prop}
\label{sketchorthogonality}
Let $(\A, \sketch)$ be an enriched limit sketch and $\C$ a tensored $\V$-category. If $\Theta_\Psi$ is the class of morphisms in $[\A, \V]$ defined in \ref{Kellysketch}, then $\sketch\Mod(\A, \C)$ is the $\V$-orthogonal sub-$\V$-category $\Delta^{\perp_\V}$ described by the class $\Delta = \left\{\theta \tensorbar C \mid \theta \in \Theta_\Psi, C \in \ob\C\right\}$.
\end{prop}

\begin{proof}
Let the sketch $(\A, \sketch)$ be as described above. We must show for every $\V$-functor $M : \A \to \C$ that $M$ is a $\Psi$-model iff $\theta_\gamma \tensorbar C \perp_\V M$ for all $\gamma \in \Gamma$ and $C \in \ob\C$. Since the representable $\V$-functors $\C(C, -) : \C \to \V$ ($C \in \ob\C$) preserve and jointly reflect limits, we have that $M$ is a $\Psi$-model iff $\C(C, M-) : \A \to \V$ is a $\Psi$-model for each $C \in \ob\C$. For every $C \in \ob\C$, we deduce by  \ref{Kellysketch} that $\C(C, M-) : \A \to \V$ is a $\Psi$-model iff $\theta_\gamma \perp_\V \C(C, M-)$ for all $\gamma \in \Gamma$, which is equivalent to $\theta_\gamma \tensorbar C \perp_\V M$ by \cite[3.9]{Completion} since $(-) \tensorbar C \dashv [\A, \C(C, -)]$ by \ref{param_tensor}.
\end{proof} 

\noindent We can now prove the following enrichment of Freyd and Kelly's central results \cite[5.2.1, 5.2.2]{FreydKelly} about ``categories of continuous functors":

\begin{theo}
\label{limitsketchgeneral}
Let $\C$ be a locally $\alpha$-bounded $\V$-category, let $(\A, \sketch)$ be an enriched limit sketch, and write $\sketch\Mod(\A, \C)$ for the $\V$-category of models of $(\A,\Psi)$ in $\C$. Then $\sketch\Mod(\A, \C)$ is a reflective sub-$\V$-category of $[\A, \C]$. If $\C$ is $\E$-cowellpowered, then $\sketch\Mod(\A, \C)$ is locally bounded and $\E$-cowellpowered, and the inclusion $\sketch\Mod(\A, \C) \hookrightarrow [\A,\C]$ is a bounding right adjoint.
\end{theo}

\begin{proof}
Let $\sketch$ be as described before \ref{param_tensor}. By \ref{sketchorthogonality} we have $\sketch\Mod(\A, \C) = \Delta^{\perp_\V}$ where $\Delta = \left\{\theta \tensorbar C \mid \theta \in \Theta_\Psi, C \in \ob\C\right\}$ and $\Theta_\Psi$ is the class of morphisms in $[\A, \V]$ defined in \ref{Kellysketch}. In \ref{bifunctorlemma}, let $\X = \C, \Y = [\A, \V], \Z = [\A, \C]$ (invoking \ref{functorfactegory} and \ref{locbdcomplete}), and let
$$\ast : \C \tensor [\A, \V] \longrightarrow [\A, \C]$$
be the $\V$-functor defined by $C \ast F = F \tensorbar C$ with the notation of \ref{param_tensor}.  Then for each object $F$ of $[\A,\V]$, the $\V$-functor $(-) \ast F = F \tensorbar (-):\C \rightarrow [\A,\C]$ is left adjoint to $\{F,-\}$ by \eqref{eq:ftensorbar_ladj}, so $(-)\ast F$ preserves colimits and also preserves the left class by \ref{preservesrightclasslem}, because $\{F, -\}$ preserves the right class by \cite[4.5]{enrichedfact}. So by \ref{bifunctorlemma}, there is a class of (pointwise) $\E$-morphisms $\Omega$ in $[\A, \C]$ such that
\[ \sketch\Mod(\A, \C) = \Delta^{\perp_\V} = \left(\Delta_1 \cup \Omega\right)^{\perp_\V} \] where $\Delta_1 = \left\{\theta \tensorbar H \mid \theta \in \Theta_\Psi, H \in \h\right\}$ and $\h$ is the enriched $(\E, \M)$-generator of $\C$. We know by \ref{functorcategorylocbd} that $[\A, \C]$ is a locally $\alpha$-bounded $\V$-category. Since $\A$ and $\h$ are small, it follows by the definition of $\Theta_\Psi$ in \ref{Kellysketch} that the class of codomains of morphisms in $\Delta_1$ is small. Hence, the needed conclusions now follow from Theorem \ref{orthoreflective}, using the fact that if $\C$ is $\E$-cowellpowered then $[\A, \C]$ is also $\E$-cowellpowered since $\A$ is small.
\end{proof}

\noindent Before we can prove a certain refinement of \ref{limitsketchgeneral}, we require the following:

\begin{prop}
\label{smallsketchprop}
Let $(\A, \sketch)$ be an enriched limit sketch. If $\sketch$ is small, then there is a regular cardinal $\beta$ such that every cylinder in $\sketch$ has a $\beta$-bounded-small weight. 
\end{prop}

\begin{proof}
Let $\sketch$ be described as before \ref{Kellysketch}. For each $\gamma \in \Gamma$, since $\K_\gamma$ is small and $\V$ is locally bounded, we can (by \ref{boundedsmallrmk}) find a regular cardinal $\beta_\gamma$ such that $F_\gamma : \K_\gamma \to \V$ is a $\beta_\gamma$-bounded-small weight. Since $\Psi$ is small, we can then find a regular cardinal $\beta$ such that every $F_\gamma$ ($\gamma \in \Gamma$) is $\beta$-bounded-small.   
\end{proof}

We now have the following refinement of \ref{limitsketchgeneral}. We say that an enriched limit sketch $(\A, \sketch)$ is an $\alpha$\textbf{-bounded-small limit sketch} if the weight of every cylinder in $\sketch$ is $\alpha$-bounded-small. Thus, \ref{smallsketchprop} says that every small enriched limit sketch is $\beta$-bounded-small for some $\beta$.  

\begin{theo}
\label{limitsketchrefined}
Let $\C$ be a locally $\alpha$-bounded and $\E$-cowellpowered $\V$-category, let $(\A, \sketch)$ be an \mbox{$\alpha$-bounded-small} limit sketch, and write $\sketch\Mod(\A, \C)$ for the $\V$-category of models of $(\A,\Psi)$ in $\C$. Then $\sketch\Mod(\A, \C)$ is a locally $\alpha$-bounded and $\E$-cowellpowered $\V$-category, and the inclusion $\sketch\Mod(\A, \C) \hookrightarrow [\A,\C]$ is an $\alpha$-bounding right adjoint.
\end{theo}

\begin{proof}
As in the proof of \ref{limitsketchgeneral}, $\sketch\Mod(\A, \C) = \left(\Delta_1 \cup \Omega\right)^{\perp_\V}$ for a class of pointwise $\E$-morphisms $\Omega$ in $[\A, \C]$, where $\Delta_1 = \left\{\theta_\gamma \tensorbar H \mid \gamma \in \Gamma, H \in \h\right\}$ and we write $\h$ for the enriched $(\E, \M)$-generator of $\C$ and employ the notation of \ref{param_tensor} and \ref{Kellysketch}. To obtain the desired result by the final statement of \ref{orthoreflective}, it suffices to show that the domain of every morphism in $\Delta_1$ is an enriched $\alpha$-bounded object of $[\A, \C]$. In view of \ref{Kellysketch}, each morphism in $\Delta_1$ is of the form
$$\theta_\gamma : (F_\gamma \ast \y D_\gamma^\op)\tensorbar H \longrightarrow \y A_\gamma \tensorbar H$$
with $\gamma \in \Gamma$ and $H \in \h$.  Since the weight $F_\gamma$ is $\alpha$-bounded-small and $\left(\y D_\gamma^\op\right)K = \A(D_\gamma K, -)$ ($K \in \ob\K_\gamma$) is an enriched $\alpha$-bounded object of $[\A, \V]$ by \ref{representablesbounded}, it follows that $F_\gamma \ast \y D_\gamma^\op$ is an enriched $\alpha$-bounded object of $[\A,\V]$ by \ref{closurealphasmallcolimits}, so $(F_\gamma \ast \y D_\gamma^\op) \tensorbar H$ is an enriched $\alpha$-bounded object of $[\A,\C]$ by \ref{presheaftensorbounded}.
\end{proof}

\noindent If $\Phi$ is a class of small weights, then by a $\Phi$\emph{-theory} we mean a small $\Phi$-complete $\V$-category $\scrT$.  Given a $\Phi$-theory $\scrT$, we write $\Phi\text{-}\Cts(\scrT, \C)$ to denote the full sub-$\V$-category of $[\scrT, \C]$ consisting of $\Phi$-continuous $\V$-functors from $\scrT$ to $\C$, which we call \textit{models of $\scrT$ in $\C$}.

\begin{theo}
\label{limittheorycor}
Let $\C$ be a locally $\alpha$-bounded $\V$-category, let $\Phi$ be a class of small weights, and let $\scrT$ be a $\Phi$-theory. Then the $\V$-category $\Phi\text{-}\Cts(\scrT, \C)$ of models of $\scrT$ in $\C$ is a reflective sub-$\V$-category of $[\scrT, \C]$. If $\C$ is $\E$-cowellpowered, then $\Phi\text{-}\Cts(\scrT, \C)$ is also locally bounded and $\E$-cowellpowered, and the inclusion $i:\Phi\text{-}\Cts(\scrT, \C) \hookrightarrow [\scrT,\C]$ is a bounding right adjoint.  If every weight in $\Phi$ is $\alpha$-bounded-small and $\C$ is $\E$-cowellpowered, then $\Phi\text{-}\Cts(\scrT, \C)$ is locally $\alpha$-bounded and $\E$-cowellpowered, and the inclusion $i$ is an $\alpha$-bounding right adjoint.
\end{theo}

\begin{proof}
The $\Phi$-theory $\scrT$ carries the structure of an enriched limit sketch $(\scrT, \sketch)$, where $\sketch$ consists of all the $\Phi$-limit cylinders in $\scrT$. Then $\Phi\text{-}\Cts(\scrT, \C) = \sketch\Mod(\scrT, \C)$, so the result follows by Theorems \ref{limitsketchgeneral} and  \ref{limitsketchrefined}.   
\end{proof}

\begin{rmk}
\label{Phipresentable}
Given a locally small class of small weights $\Phi$ satisfying Axiom A from \cite{LR}, a $\V$-category $\C$ is said to be \textit{locally $\Phi$-presentable} in the sense of \cite{LR} if $\C \simeq \Phi\text{-}\Cts(\scrT, \V)$ for a $\Phi$-theory $\scrT$.   Hence if $\C$ is locally $\Phi$-presentable and the locally $\alpha$-bounded closed category $\V$ is $\E$-cowellpowered, then our result \ref{limittheorycor} entails that $\C$ is a locally bounded and $\E$-cowellpowered $\V$-category (and is locally $\alpha$\emph{-bounded} if every weight in $\Phi$ is $\alpha$-bounded-small). \qed 
\end{rmk}

\begin{defn}
\label{alphaboundedsmalltheory}
An $\alpha$\textbf{-bounded-small limit theory} is a $\Phi_\alpha$-theory $\scrT$ for the class $\Phi_\alpha$ of all $\alpha$-bounded-small weights, i.e. a small $\V$-category with all $\alpha$-bounded-small limits. \qed
\end{defn}

\noindent Theorem \ref{limittheorycor} now entails the following:

\begin{theo}
Let $\C$ be an $\E$-cowellpowered locally $\alpha$-bounded $\V$-category, and let $\scrT$ be an \mbox{$\alpha$-bounded-small} limit theory.  Then $\Phi_\alpha\text{-}\Cts(\scrT, \C)$ is a locally $\alpha$-bounded and $\E$-cowellpowered \mbox{$\V$-category}, and the inclusion $\Phi_\alpha\text{-}\Cts(\scrT, \C) \hookrightarrow [\scrT,\C]$ is an $\alpha$-bounding right adjoint.
\end{theo}

\noindent We now use Theorem \ref{limittheorycor} to prove the following result on the local boundedness, reflectivity, and monadicity of $\V$-categories of algebras of enriched algebraic theories in locally bounded $\V$-categories. Recall from \cite[3.8]{EAT} that a \emph{system of arities} $\J$ in $\V$ can be defined as a full sub-$\V$-category $\J \hookrightarrow \V$ that contains the unit object $I$ and is closed under $\tensor$. A $\J$\emph{-theory} is then a $\V$-category $\scrT$ equipped with an identity-on-objects $\V$-functor $\tau : \J^\op \to \scrT$ that preserves $\J$-cotensors, and a $\scrT$\emph{-algebra} in a $\V$-category $\C$ is a $\V$-functor $M : \scrT \to \C$ that preserves $\J$-cotensors. The following result generalizes (part of) \cite[8.6]{EAT}:  

\begin{cor}
\label{Jtheorycor}
Let $\C$ be a locally $\alpha$-bounded and $\E$-cowellpowered $\V$-category. Let $\J$ be a small system of arities in $\V$, and let $\scrT$ be a $\J$-theory. Then the $\scrT$-algebras form a full sub-$\V$-category $\scrTAlg(\C) \hookrightarrow [\scrT, \C]$ that is reflective and locally bounded, and the forgetful $\V$-functor $U^\scrT : \scrTAlg(\C) \to \C$ given by $M \mapsto M(I)$ is monadic and is a bounding right adjoint. If every $J \in \ob\J$ is an enriched $\alpha$-bounded object of $\V$, then $\scrTAlg(\C)$ is, moreover, locally $\alpha$-bounded, and $U^\scrT$ is an $\alpha$-bounding right adjoint.   
\end{cor}

\begin{proof}
By \cite[4.3]{EAT}, $\scrT$ has $\J$-cotensors and so is, in particular, a $\Phi$-theory where $\Phi$ is the class of weights for $\J$-cotensors, so that $\scrTAlg(\C) = \Phi\text{-}\Cts(\scrT, \C)$. Hence, by Theorem \ref{limittheorycor} there is some regular cardinal $\beta$ such that $\scrTAlg(\C)$ is an $\E$-cowellpowered locally $\beta$-bounded $\V$-category and the inclusion $i:\scrTAlg(\C) \hookrightarrow [\scrT,\C]$ is a $\beta$-bounding right adjoint, while if every $J \in \ob\J$ is an enriched $\alpha$-bounded object of $\V$ then we may take $\beta = \alpha$ (again by \ref{limittheorycor}, since each of the weights in $\Phi$ is $\alpha$-bounded-small by \ref{cotensorsboundedsmall}).  Since both $i$ and the evaluation $\V$-functor $\mathsf{Ev}_{I} : [\scrT, \C] \to \C$ are continuous $\beta$-bounded right-class $\V$-functors, the composite $U^\scrT = \mathsf{Ev}_{I} \circ i$ is also a continuous $\beta$-bounded right-class $\V$-functor.  Hence $U^\scrT$ has a left adjoint by \ref{enrichedadjfunctorthm}, so $U^\scrT$ is monadic by \cite[8.1]{EAT}.  Being monadic, $U^\scrT$ is also conservative and hence $\M$-conservative, so $U^\scrT$ is a $\beta$-bounding right adjoint by \ref{charn_bounding_radjs}.
\end{proof}

\begin{egg}
\label{lastrmk}
As a special case of \ref{Jtheorycor}, with $\V = \Set$ and $\J$ the finite cardinals (see \cite[3.3, 4.2.1, 5.3.1]{EAT}), if $\C$ is a locally $\alpha$-bounded and $\E$-cowellpowered category and $\scrT$ is a Lawvere theory, then the category $\scrTAlg(\C)$ of $\scrT$-algebras in $\C$ is reflective in the functor category $[\scrT,\C]$ and locally $\alpha$-bounded, and the forgetful functor $\scrTAlg(\C) \to \C$ is monadic.  Indeed, this follows from \ref{Jtheorycor} since finite cardinals are clearly $\alpha$-bounded objects of $\Set$. \qed
\end{egg}

\section{Locally bounded closed \texorpdfstring{$\V$}{V}-categories of models for monoidal limit theories}
\label{sec:lbclosed_vcats_sm_phitheories}

We conclude the paper with a result (\ref{Daylocbd} below) that uses \ref{limittheorycor}, along with Day convolution \cite{Dayclosed} and Day's reflection theorem for closed categories \cite{Dayclosed}, to produce further examples of locally bounded closed categories, namely categories of models of enriched \emph{symmetric monoidal limit theories}, in a sense to be defined shortly (see \ref{monoidalPhitheory}). In fact, we prove a slightly stronger result, for which we require the following definition, \ref{locbdclosedVcategory}. Recall first that a \emph{symmetric monoidal} $\V$\emph{-category} is a symmetric pseudomonoid in $\V\CAT$, i.e. a $\V$-category $\W$ equipped with a $\V$-functor $\tensor_\W : \W \tensor \W \to \W$ and an object $I_\W \in \ob\W$, together with $\V$-natural isomorphisms making $\W_0$ symmetric monoidal.  A \emph{symmetric monoidal closed} $\V$\emph{-category} is a symmetric monoidal $\V$-category $\W$ such that each $\V$-functor $W \tensor_\W (-) : \W \to \W$ ($W \in \ob\W$) has a right adjoint $[W, -]_\W : \W \to \W$.  

\begin{defn}
\label{locbdclosedVcategory}
Let $\alpha$ be a regular cardinal. A \textbf{locally} $\alpha$\textbf{-bounded (symmetric monoidal) closed} $\V$\textbf{-category} is a locally $\alpha$-bounded $\V$-category $\W$ that carries the structure of a symmetric monoidal closed $\V$-category such that the unit object $I_\W$ is a $\V$-enriched $\alpha$-bounded object, the monoidal product $H \tensor_\W H'$ is a $\V$-enriched $\alpha$-bounded object for all $H, H'$ in the $\V$-enriched $(\E,\M)$-generator $\h$ of $\W$, and $[W, -]_\W : \W \to \W$ preserves the right class for all $W \in \ob\W$ (equivalently, by \ref{preservesrightclasslem}, each $W \tensor_\W (-)$ preserves the left class).  A symmetric monoidal closed $\V$-category $\W$ is \textbf{locally bounded (as a symmetric monoidal closed} $\V$\textbf{-category)} if there is some regular cardinal $\alpha$ for which $\W$ is a locally $\alpha$-bounded closed $\V$-category. \qed
\end{defn} 

\begin{prop}
\label{locbdclosedprop}
Let $\W$ be a locally $\alpha$-bounded closed $\V$-category.  Then $\W_0$ is a locally $\alpha$-bounded closed category.
\end{prop}

\begin{proof}
The hypothesis clearly entails that $\W_0$ is a closed factegory.  Also, since $\W$ is a locally $\alpha$-bounded $\V$-category, we deduce by \ref{enrichedtoordinarylocallybounded} that $\W_0$ is a locally $\alpha$-bounded category with ordinary $(\E, \M)$-generator $\G \tensor \h = \left\{G \tensor H \mid G \in \G, H \in \h\right\}$, where $\G$ is the ordinary $(\E, \M)$-generator of $\V_0$ and $\h$ is the $\V$-enriched $(\E, \M)$-generator of $\W$.  Since $I_\W$ is a $\V$-enriched $\alpha$-bounded object of $\W$ by assumption, it follows by \ref{enrichedboundedcor} that $I_\W$ is an ordinary $\alpha$-bounded object of $\W_0$. Because $\tensor_\W : \W \tensor \W \to \W$ preserves $\V$-enriched weighted colimits (and hence $\V$-enriched tensors) in each variable separately, 
\[ \left(G_1 \tensor H_1\right) \tensor_\W \left(G_2 \tensor H_2\right) \cong \left(G_1 \tensor G_2\right) \tensor \left(H_1 \tensor_\W H_2\right) \] for all $G_1, G_2 \in \G$ and $H_1, H_2 \in \h$. Since $G_1 \tensor G_2$ is an ordinary $\alpha$-bounded object of $\V_0$ and $H_1 \tensor_\W H_2$ is a $\V$-enriched $\alpha$-bounded object of $\W$, it follows that $\left(G_1 \tensor G_2\right) \tensor \left(H_1 \tensor_\W H_2\right)$ is an ordinary $\alpha$-bounded object of $\W_0$, by \ref{enrichedboundedlemma}.
\end{proof}

\begin{defn}
\label{monoidalPhitheory}
Let $\Phi$ be a class of small weights. A \textbf{symmetric monoidal} $\Phi$\textbf{-theory} is a small symmetric monoidal $\V$-category $\scrT$ with $\Phi$-limits such that $\tensor_\scrT : \scrT \tensor \scrT \to \scrT$ preserves $\Phi$-limits in each variable separately. \qed 
\end{defn}

\noindent We now prove our final result of the paper:

\begin{theo}
\label{Daylocbd}
Suppose $\V$ is $\E$-cowellpowered, let $\Phi$ be a class of small weights, and let $\scrT$ be a symmetric monoidal $\Phi$-theory. Then $\Phi\text{-}\Cts(\scrT, \V)$ is an $\E$-cowellpowered locally bounded symmetric monoidal closed $\V$-category, and hence $\Phi\text{-}\Cts(\scrT, \V)_0$ is an $\E$-cowellpowered locally bounded closed category by \ref{locbdclosedprop}. If every weight in $\Phi$ is $\alpha$-bounded-small, then $\Phi\text{-}\Cts(\scrT, \V)$ is a locally $\alpha$-bounded closed $\V$-category, and $\Phi\text{-}\Cts(\scrT, \V)_0$ is a locally $\alpha$-bounded closed category. 
\end{theo}

\begin{proof}
Letting $\W = \Phi\text{-}\Cts(\scrT, \V)$, we first show that $\W$ is a symmetric monoidal closed $\V$-category, for which we use Day's reflection theorem for closed categories \cite{Dayclosed}. Since $\scrT$ is a small symmetric monoidal $\V$-category, we may equip the presheaf $\V$-category $[\scrT, \V]$ with the structure of a symmetric monoidal closed $\V$-category by \cite[3.3, 3.6]{Dayclosed}, whose symmetric monoidal product $\tensor_\Day$ is given by Day convolution, and whose unit object is the representable $\scrT\left(I_\scrT, -\right)$. By \ref{limittheorycor} (and \ref{VlocallyboundedVcat}) with $\C = \V$, we know that $\W$ is reflective in $[\scrT, \V]$. To apply Day's reflection theorem, we now show that if $F,M:\scrT \rightarrow \V$ are $\V$-functors and $M$ is a $\scrT$-model, then the internal hom $\llbracket F, M \rrbracket : \scrT \to \V$ in $[\scrT,\V]$ is still a $\scrT$-model. By definition
\begin{equation}\label{eq:hom_in_w} \llbracket F, M \rrbracket T = \int_{X \in \scrT} \V\left(FX, M\left(T \tensor_\scrT X\right)\right) = \biggl\{F,M\bigl(T \otimes_\scrT (-)\bigr)\biggr\}\end{equation}
$\V$-naturally in $T \in \scrT$, so $\llbracket F, M \rrbracket$ is the composite
\begin{equation}\label{eq:comp_day_hom}\scrT \xrightarrow{\widetilde{\otimes}} [\scrT,\scrT] \xrightarrow{[\scrT,M]} [\scrT,\V] \xrightarrow{\{F,-\}} \V\end{equation}
where $\widetilde{\otimes}$ is defined by $\widetilde{\otimes}\:T = T \otimes_\scrT (-)$.  But all three $\V$-functors in \eqref{eq:comp_day_hom} preserve $\Phi$-limits, since $\scrT$ is a symmetric monoidal $\Phi$-theory and $M$ is a $\scrT$-model, so $\llbracket F, M \rrbracket$ preserves $\Phi$-limits.

We now conclude by \cite[1.2]{Dayclosed} that $\W$ is a symmetric monoidal closed $\V$-category whose monoidal product is the reflection of the Day convolution monoidal product in $[\scrT, \V]$, whose internal homs are as in $[\scrT, \V]$, and whose unit object is the reflection of the representable $\scrT\left(I_\scrT, -\right)$, which is just $\scrT\left(I_\scrT, -\right)$ itself (because $\scrT\left(I_\scrT, -\right)$ is already a $\scrT$-model).   

Since $\V$ is $\E$-cowellpowered, we know by \ref{limittheorycor} that there is some $\beta \geq \alpha$ such that $\W$ is a locally $\beta$-bounded and $\E$-cowellpowered $\V$-category, and if every weight in $\Phi$ is $\alpha$-bounded-small then we may take $\beta = \alpha$.  In view of \ref{locbdclosedprop}, it suffices to show that $\W$ is a locally $\beta$-bounded closed $\V$-category.

The factorization system $(\E_\W,\M_\W)$ on $\W$ is induced from the pointwise factorization system on $[\scrT, \V]$ via the reflection $[\scrT, \V] \to \W$, so that $\M_\W$ consists of the $\V$-natural transformations that are pointwise in $\M$ (see the proof of \ref{orthoreflective}).  Hence, in view of the formula \eqref{eq:hom_in_w}, $\M_\W$ is stable under $\llbracket M, -\rrbracket$ since $\M$ is stable under weighted limits.  The $\V$-enriched $(\E_\W, \M_\W)$-generator $\h$ associated to $\W$ is the reflection of the $\V$-enriched $(\E, \M)$-generator of $[\scrT, \V]$, the latter being the set of representables $\mathscr{R} = \left\{ \y T \ \mid T \in \ob\scrT\right\}$ by \ref{locallyboundedpresheafcat}, where $\y:\scrT^\op \rightarrow [\scrT,\V]$ is the Yoneda embedding, but since the representables are already $\scrT$-models, $\h = \mathscr{R}$.  Hence, every representable $\y T$ ($T \in \ob\scrT$) is a $\V$-enriched $\beta$-bounded object of $\W$.  In particular, the unit object $\y  I_\scrT$ of $\W$ is a $\V$-enriched $\beta$-bounded object. Also, since $\y $ is strong monoidal with respect to Day convolution, we find that for all objects $T$ and $T'$ of $\scrT$, $\y T \otimes_\Day \y T' \cong \y (T \otimes_\scrT T')$ is a $\V$-enriched $\beta$-bounded object of $\W$ and is isomorphic to $\y T \otimes_\W \y T'$.
\end{proof}

\noindent We now develop some classes of examples of enriched symmetric monoidal limit theories and their locally bounded closed $\V$-categories of models. Throughout, we suppose that the given locally $\alpha$-bounded closed category $\V$ is also $\E$-cowellpowered.

\begin{egg}
\label{symmonlimtheories1}
Let $\V_\alpha$ be the full sub-$\V$-category of $\V$ consisting of the ordinary/enriched (see \ref{ordinaryequalsenrichedbounded}) $\alpha$-bounded objects of $\V$.  By \ref{unitbounded} and \ref{enrichedboundedprop}, $\V_\alpha$ contains the unit object of $\V$ and is closed under the monoidal product in $\V$, so $\V_\alpha$ is a symmetric monoidal $\V$-category.  Letting $\Phi_\alpha$ be the class of $\alpha$-bounded-small weights, we know that $\V_\alpha \hookrightarrow \V$ is closed in $\V$ under $\Phi_\alpha$-colimits by \ref{closurealphasmallcolimits}, so $\V_\alpha$ has $\Phi_\alpha$-colimits that are preserved by its monoidal product in each variable separately (since this is so for $\V$).  Letting $\scrT_\alpha$ be a skeleton of $\V_\alpha^\op$, we find that $\scrT_\alpha$ is small by \cite[2.9]{Sousa}, so $\scrT_\alpha$ is a \textit{symmetric monoidal $\alpha$-bounded-small limit theory}, i.e. a symmetric monoidal $\Phi_\alpha$-theory.  By \ref{Daylocbd}, it then follows that $\Phi_\alpha\text{-}\Cts\left(\scrT_\alpha, \V\right)_0$ is a locally $\alpha$-bounded and $\E$-cowellpowered closed category. \qed
\end{egg}

\begin{egg}
\label{symmonlimtheories2}
Let $j : \J \hookrightarrow \V$ be a small and eleutheric system of arities \cite[7.1]{EAT}, and let $\scrT$ be a commutative $\J$-theory \cite[5.9]{commutants}. Then the full sub-$\V$-category $\W = \scrTAlg^!(\V) \hookrightarrow \scrTAlg(\V)$ of \emph{normal} $\scrT$-algebras in $\V$ \cite[5.10]{EAT} is equivalent to $\scrTAlg(\V)$ \cite[5.14]{EAT} and so is a locally bounded $\V$-category by \ref{Jtheorycor}.  By \cite[3.4.1]{functional}, $\W$ is a symmetric monoidal closed $\V$-category equipped with a (lax) symmetric monoidal $\V$-adjunction $F \dashv U : \W \to \V$, where $U$ is the restriction of the $\V$-functor $U^\scrT:\scrTAlg(\V) \rightarrow \V$ of \ref{Jtheorycor}.  The left adjoint $F : \V \to \W$ is therefore strong monoidal by \cite[1.5]{Kellydoctrinal}. The Yoneda embedding $\y:\scrT^\op \rightarrow \scrTAlg(\V)$ sends each $J \in \ob\scrT = \ob\J$ to a $\scrT$-algebra $\scrT(J,-)$ that is free on $J$ (by the Yoneda lemma) and so is isomorphic to $FJ$.  Hence there is a fully faithful $\V$-functor $\scrT^\op \rightarrow \W$ that is given on objects by $J \mapsto FJ$ and therefore restricts to an equivalence $E:\scrT^\op \xrightarrow{\sim} \F$ where $\F \hookrightarrow \W$ is the full sub-$\V$-category consisting of all free normal $\scrT$-algebras on objects of $\J$.  Since $F$ is strong monoidal, $\F$ is closed under the monoidal product of $\W$ and contains its unit object, so $\scrT$ and $\F$ are symmetric monoidal $\V$-categories in such a way that $E$ is a symmetric monoidal equivalence.  The left adjoint $F$ preserves tensors, so $\F$ is closed under $\J$-tensors in $\W$, and hence $\F$ has $\J$-tensors that are preserved by its monoidal product in each variable separately (because the same is true in $\W$).  Hence $\scrT$ is a small symmetric monoidal $\J$-cotensor theory.  We therefore deduce by \ref{Daylocbd} that $\Phi\text{-}\Cts(\scrT, \V) = \scrTAlg(\V)$ is an $\E$-cowellpowered locally bounded symmetric monoidal closed $\V$-category, whose underlying ordinary category is a locally bounded closed category, while if every $J \in \ob\J$ is an enriched $\alpha$-bounded object of $\V$, then by \ref{cotensorsboundedsmall} and \ref{Daylocbd} we may refine these conclusions by replacing the phrase \textit{locally bounded} with \textit{locally $\alpha$-bounded}. \qed     
\end{egg}

\bibliographystyle{amsplain}
\bibliography{mybib}

\end{document}